\documentclass[11pt]{article}
\usepackage{amsmath,amsthm,amssymb}
\usepackage[usenames,dvipsnames]{xcolor}
\usepackage{enumerate}
\usepackage{graphicx}
\usepackage{cite}
\usepackage{comment}
\usepackage{oands}
\usepackage{tikz}
\usepackage{changepage}
\usepackage{bbm}
\usepackage{mathtools}
\usepackage[margin=1in]{geometry}
\usepackage[pagewise,mathlines]{lineno}
\usepackage{appendix}
\usepackage{stmaryrd}
\usepackage{multicol}
\usepackage{microtype}
\usepackage[colorinlistoftodos]{todonotes}

\usepackage[pdftitle={Geodesics and metric ball boundaries in Liouville quantum gravity},
  pdfauthor={Ewain Gwynne, Joshua Pfeffer, and Scott Sheffield},
colorlinks=true,linkcolor=NavyBlue,urlcolor=RoyalBlue,citecolor=PineGreen,bookmarks=true,bookmarksopen=true,bookmarksopenlevel=2,unicode=true,linktocpage]{hyperref}

\theoremstyle{plain}
\newtheorem{thm}{Theorem}[section]

\newtheorem{lem}[thm]{Lemma}
\newtheorem{prop}[thm]{Proposition}

\def\@rst #1 #2other{#1}
\newcommand\MR[1]{\relax\ifhmode\unskip\spacefactor3000 \space\fi
  \MRhref{\expandafter\@rst #1 other}{#1}}
\newcommand{\MRhref}[2]{\href{http://www.ams.org/mathscinet-getitem?mr=#1}{MR#2}}

\theoremstyle{definition}
\newtheorem{defn}[thm]{Definition}
\newtheorem{remark}[thm]{Remark}

\numberwithin{equation}{section}

\newcommand{\dsb}{\begin{adjustwidth}{2.5em}{0pt}
\begin{footnotesize}}
\newcommand{\dse}{\end{footnotesize}
\end{adjustwidth}}

\newcommand{\ssb}{\begin{adjustwidth}{2.5em}{0pt}}
\newcommand{\sse}{\end{adjustwidth}}

\newcommand{\aryb}{\begin{eqnarray*}}
\newcommand{\arye}{\end{eqnarray*}}
\def\alb#1\ale{\begin{align*}#1\end{align*}}
\def\allb#1\alle{\begin{align}#1\end{align}}
\newcommand{\eqb}{\begin{equation}}
\newcommand{\eqe}{\end{equation}}
\newcommand{\eqbn}{\begin{equation*}}
\newcommand{\eqen}{\end{equation*}}

\newcommand{\BB}{\mathbbm}
\newcommand{\ol}{\overline}

\newcommand{\op}{\operatorname}

\newcommand{\eqD}{\overset{d}{=}}
\newcommand{\ep}{\varepsilon}
\newcommand{\rta}{\rightarrow}

\newcommand{\wt}{\widetilde}
\newcommand{\wh}{\widehat}
\newcommand{\mcl}{\mathcal}

\newcommand{\bdy}{\partial}

\let\originalleft\left
\let\originalright\right
\renewcommand{\left}{\mathopen{}\mathclose\bgroup\originalleft}
\renewcommand{\right}{\aftergroup\egroup\originalright}

\title{Geodesics and metric ball boundaries in Liouville quantum gravity}
 \date{ }
 \author{
\begin{tabular}{c} Ewain Gwynne\\[-5pt]\small Chicago \end{tabular}
\begin{tabular}{c} Joshua Pfeffer\\[-5pt]\small Columbia \end{tabular} 
\begin{tabular}{c} Scott Sheffield\\[-5pt]\small MIT \end{tabular} 
}

\begin{document}

\maketitle

\begin{abstract}
Recent works have shown that there is a canonical way to to assign a metric (distance function) to a Liouville quantum gravity (LQG) surface for any parameter $\gamma \in (0,2)$. We establish a strong confluence property for LQG geodesics, which generalizes a result proven by Angel, Kolesnik and Miermont for the Brownian map. Using this property, we also establish zero-one laws for the Hausdorff dimensions of geodesics, metric ball boundaries, and metric nets w.r.t.\ the Euclidean or LQG metric. In the case of a metric ball boundary, our result combined with earlier work of Gwynne (2020) gives a formula for the a.s.\ Hausdorff dimension for the boundary of the metric ball stopped when it hits a fixed point in terms of the Hausdorff dimension of the whole LQG surface. We also show that the Hausdorff dimension of the metric ball boundary is carried by points which are not on the boundary of any complementary connected component of the ball. 
\end{abstract}

\tableofcontents

\section{Introduction}
\label{sec-intro}

In the 1980s, physicists working in conformal field theory introduced a theory of random surfaces called Liouville quantum gravity (LQG) as canonical models of random two-dimensional Riemannian manifolds~\cite{polyakov-qg1,david-conformal-gauge,dk-qg}.  The subject has attracted a substantial amount of mathematical attention in recent years, because of both its relevance to several areas of mathematical physics and its relationship to random discrete surfaces called random planar maps. See~\cite{gwynne-ams-survey,berestycki-lqg-notes} for introductory articles on LQG from a mathematical perspective.
We can define LQG heuristically as follows.

\begin{defn}[Heuristic formulation of LQG]
Let $\gamma \in (0,2)$.  A $\gamma$-Liouville quantum gravity ($\gamma$-LQG) surface is a random Riemannian manifold with random Riemannian metric tensor
\eqb
e^{\gamma h} (dx^2 + dy^2),
\label{eqn-metric-tensor}
\eqe
where $h$ is a variant of the \emph{Gaussian free field} (GFF) on some domain $U\subset\BB C$ and $dx^2+dy^2$ is the Euclidean metric tensor on $U$.  
\end{defn}

See, e.g.,~\cite{shef-gff,berestycki-lqg-notes,pw-gff-notes} for an introduction to the GFF. 
The metric tensor~\eqref{eqn-metric-tensor} is not literally well-defined since $h$ is a distribution, not a function, so cannot be exponentiated pointwise.  Despite this obstacle, probabilists have rigorously defined both a random measure~\cite{kahane,shef-kpz,rhodes-vargas-log-kpz} and a random metric (distance function)~\cite{dddf-lfpp,gm-uniqueness} associated to~\eqref{eqn-metric-tensor} via renormalization procedures.  In this paper, we focus on the formulation of LQG as a random metric space, and we describe some fundamental properties of LQG geodesics and metric ball boundaries.
\bigskip

\noindent\textbf{Acknowledgments.} We thank an anonymous referee for helpful comments on an earlier version of this paper.
E.G.\ was partially supported by a Clay research fellowship and a Trinity college, Cambridge junior research fellowship. 
J.P.\ was partially supported by the National Science Foundation under Grant No. 2002159.
S.S.\ was partially supported by NSF Grant DMS 1712862. 
No code or data was involved in this work.

\subsection{A stronger confluence property for geodesics}

Throughout this paper, we will focus primarily on the case when $h$ is the whole-plane GFF (results for other variants of the GFF can be extracted via local absolute continuity). The whole-plane GFF is defined only modulo additive constant, but we will almost always fix the constant by requiring that the average of the field over the unit circle is zero.  

Miller and Sheffield in~\cite{lqg-tbm1,lqg-tbm2,lqg-tbm3} defined the metric associated with an LQG surface (i.e., the Riemannian distance function associated with~\eqref{eqn-metric-tensor}) in the special case $\gamma = \sqrt{8/3}$.  Their work also showed that certain special $\sqrt{8/3}$-LQG surfaces are isometric to \emph{Brownian surfaces}~\cite[Corollary 1.5]{lqg-tbm2}.  Brownian surfaces (such as the Brownian map) are random metric spaces that arise as scaling limits of uniform random planar maps with respect to the Gromov-Hausdorff topology; see, e.g.,~\cite{legall-uniqueness,miermont-brownian-map,bet-mier-disk}. For several years, this was the only value of $\gamma$ for which mathematicians could define an LQG metric.
 
More recently, \cite{gm-uniqueness} defined the $\gamma$-LQG metric $D_h$ for \emph{all} values of $\gamma$ as the culmination of a long series of papers~\cite{dddf-lfpp,local-metrics,lqg-metric-estimates,gm-confluence,gm-coord-change}. We recall their definition of the $\gamma$-LQG metric in Section~\ref{sec-metric-def}. 

An important difficulty in the study of LQG surfaces is the fact that LQG geodesics are not locally determined by the field $h$, since one needs to see the LQG lengths of all possible paths between two points to see which one has minimal length.   
One possible way to get around this difficulty is by means of \emph{confluence of geodesics}, which was first established for the LQG metric in~\cite{gm-confluence}.
The version of confluence in~\cite{gm-confluence} says that for any fixed point $z$, a.s.\ any two geodesics started from $z$ (with arbitrary target points) coincide for a non-trivial initial time interval. We note that this is very different from the behavior of geodesics in a smooth Riemannian manifold.
Very roughly speaking, confluence of geodesics is used in~\cite[Section 4]{gm-uniqueness} to show that LQG geodesics between typical points are stable in the sense that changing the field $h$ in a small neighborhood of a point on an LQG geodesic is unlikely to result in a macroscopic change to the geodesic. 
This provides an ``approximate Markov property" of LQG geodesics which plays a crucial role in the proof of the uniqueness of the LQG metric.  
The work~\cite{legall-geodesics} proved a similar confluence property for the Brownian map, which was used in the proof that uniform quadrangulations converge to the Brownian map.

Although the above confluence property is useful, it only concerns geodesics started from a \emph{fixed} point, not geodesics between arbitrary points, so its use is limited when we want to analyze finer properties of the geodesic structure of LQG. In the setting of a Brownian surface,~\cite{akm-geodesics} establish several properties of the geodesic structure of the Brownian map by first proving a \emph{stronger} version of the confluence property~\cite[Proposition 12]{akm-geodesics}.   Roughly speaking, they show that geodesics will merge, not only  when started from the same point, but also when started \emph{near} a typical point. 
See also~\cite{mq-strong-confluence} for an even stronger form of confluence for the Brownian map, which holds for geodesics between \emph{arbitrary} points.  
Our first result is the analog of the confluence property of~\cite{akm-geodesics} for the $\gamma$-LQG metric, for general $\gamma\in (0,2)$ (it remains open to extend the stronger result of~\cite{mq-strong-confluence} to the case of general $\gamma \in (0,2)$).

\begin{thm}[Confluence of LQG geodesics started near a typical point] \label{thm-general-confluence}
Let $\gamma\in(0,2)$, let $h$ be the whole-plane GFF and let $D_h$ be the associated $\gamma$-LQG metric. 
Almost surely, for each  neighborhood $U$ of 0 there is a neighborhood $U'\subset U$ of 0 and a point $z_0 \in U\setminus U'$ such that every $D_h$-geodesic from a point in $U'$ to a point in $\BB C\setminus U$ passes through $z_0$.
\end{thm}

\begin{figure}[ht!]
 \begin{center}
\includegraphics[scale=1]{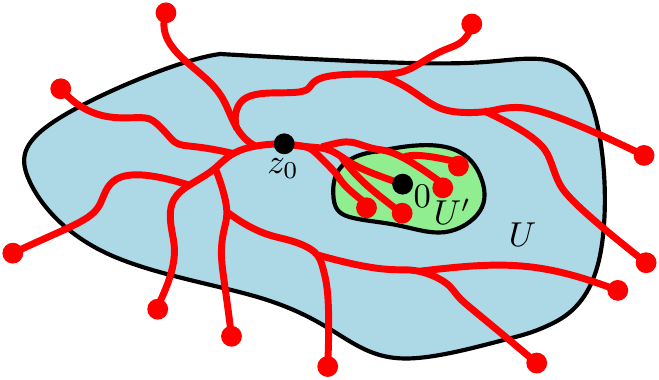}
\vspace{-0.01\textheight}
\caption{Illustration of the statement of Theorem~\ref{thm-general-confluence}. All of the geodesics from points inside the light green region $U'$ to points outside the union of the light blue and light green regions pass through the point $z_0$ (several such geodesics are shown in red). Note that this is stronger than the confluence result in~\cite[Theorem 1.3]{gm-confluence} which only concerns geodesics started from 0. 
}\label{fig-general-confluence}
\end{center}
\vspace{-1em}
\end{figure} 

See Figure~\ref{fig-general-confluence} for an illustration of the statement. We prove Theorem~\ref{thm-general-confluence} in Section~\ref{sec-strong-confluence}.
The proof is in some ways similar to that of the analogous property~\cite[Proposition 12]{akm-geodesics} of the Brownian map, but with different inputs.
 
\begin{remark} \label{remark-log-singularity}
Our proof of Theorem~\ref{thm-general-confluence} also works if instead of a whole-plane GFF we consider a whole-plane GFF plus $-\alpha\log|\cdot|$ for $\alpha \in (-\infty,2/\gamma+\gamma/2)$ (this is the range of $\alpha$ values for which the metric is locally finite, see~\cite[Theorem 1.10]{lqg-metric-estimates}). 
In particular, taking $\alpha=\gamma$, we see that our results hold with the origin replaced by a typical point sampled from the LQG area measure. 
This uses the fact that the results of~\cite{gm-confluence} also work for a whole-plane GFF plus $-\alpha\log|\cdot|$, see~\cite[Remark 1.5]{gm-confluence}.
\end{remark} 

As in the setting of the Brownian map, the  confluence property of Theorem~\ref{thm-general-confluence} is an important ingredient in analyzing the metric properties of LQG.  For example,~\cite{akm-geodesics} used~\cite[Proposition 12]{akm-geodesics} to classify the types of geodesic networks that are dense. Specifically, they showed that for each $k \in \{1,2,3,4,6,9\}$, there is a certain topological configuration of exactly $k$ distinct geodesics joining two points which occurs for a dense set of points in the Cartesian product of the Brownian map with itself; and the set of pairs of points joined by any other possible configuration of geodesics is nowhere dense. 
The paper~\cite{gwynne-geodesic-network} uses Theorem~\ref{thm-general-confluence} to prove the same result for LQG geodesics.

\subsection{Zero-one laws for dimensions of geodesics and metric ball boundaries}
\label{sec-intro-zero-one}

In this paper, we will apply Theorem~\ref{thm-general-confluence} in Sections~\ref{sec-two-ingredients}-\ref{sec-ball-bdry}, in which we prove results concerning the \emph{Hausdorff dimensions} of several random fractals associated with the LQG metric. We now describe these results.

\begin{defn} \label{def-ball}
For a metric space $(X,D)$, a subset $A\subset  X$, and a radius $r >0$, we write $\mcl B_r(A;D)$ for the closed metric ball consisting of the set of points in $X$ lying at $D$-distance at most $r$ from $A$. 
For $A\subset\BB C$, we write $B_r(A) = \mcl B_r(A;|\cdot|)$ for the Euclidean $r$-neighborhood of $z$.
When $A = \{x\}$ is a singleton, we abbreviate $ \mcl B_r(\{x\} ; D) = \mcl B_r(x;D)$ or $B_r(\{x\}) = B_r(x)$. 
\end{defn}

\begin{defn}[Definition of Hausdorff dimension]
The Hausdorff dimension of a metric space $(X,D)$ is the infimum of the set of $d> 0$ such that the following is true: for each $\ep > 0$, we can cover $X$ by a collection of balls $\{\mcl B_{r_j}(x_j;D)\}_{j \in \BB N}$ for which
\[
\sum_{j \in \BB N} r_j^d < \ep.
\] 
When $X\subset \BB C$ and $D = |\cdot|$ is the Euclidean metric, we write $\dim_{\mcl H}^0 X$ for the Hausdorff dimension of $(X,|\cdot|)$ and call it the \emph{Euclidean dimension} of $X$.  When $D = D_h$ is the $\gamma$-LQG metric, we write $\dim_{\mcl H}^\gamma X$ for the Hausdorff dimension of $(X,D_h)$ and call it the \emph{$\gamma$-quantum dimension} of $X$.  
\end{defn}

In the context of LQG, a natural first question to ask is, what is the Hausdorff dimension of the $\gamma$-LQG metric space---i.e., what is the $\gamma$-quantum dimension of $\BB{C}$?  This question has been studied by physicists since the 1990s, long before the rigorous construction of $\gamma$-LQG metric made it possible to state the question rigorously.  The value of this dimension as a function of $\gamma$ is not explicitly known except in the case when $\gamma=\sqrt{8/3}$, when we know that the dimension is 4. There is not even a plausible guess for the dimension for other values of $\gamma$ (the best-known physics guess, due to Watabiki~\cite{watabiki-lqg}, was disproven in~\cite{ding-goswami-watabiki}). 
It is shown in~\cite{gp-kpz} that $\dim_{\mcl H}^\gamma\BB C$ is a.s.\ equal to the so-called \emph{LQG dimension exponent} $d_\gamma$ from~\cite{dg-lqg-dim,dzz-heat-kernel}, which can be defined in terms of various approximations of LQG (such as random planar maps) and which features in the definition of the $\gamma$-LQG metric.  See~\cite{dg-lqg-dim,gp-lfpp-bounds,ang-discrete-lfpp} for the best currently known upper and lower bounds on $d_\gamma$ as a function of $\gamma$.

A next natural question is whether there is any relation between the Euclidean and $\gamma$-quantum dimensions of a random fractal set $X\subset\BB C$.
When $X$ is deterministic---or random but independent from the underlying field $h$---the KPZ formula~\cite{kpz-scaling,shef-kpz,rhodes-vargas-log-kpz} gives an explicit relationship between $\dim_{\mcl H}^0 X$ and $\dim_{\mcl H}^\gamma X$ in terms of $\gamma$ and $d_\gamma$ (see~\cite{gp-kpz} for a proof of the KPZ formula for the metric).  However, many of the most natural fractals to study in the LQG setting, such as geodesics and metric ball boundaries, are \emph{not} independent of the underlying field $h$, so the KPZ formula does not apply.  

In this paper, we will prove that the Euclidean and $\gamma$-quantum dimensions of several such fractals (such as ball boundaries, geodesics, and the metric net) are a.s.\ equal to deterministic constants (see theorem statements below). To explain why such zero-one law results are important, we first need to recall the most common approach for computing the Hausdorff dimension of a random fractal, which is based on a result from fractal geometry called \emph{Frostman's lemma} (see, e.g.,~\cite[Theorem 4.32]{peres-bm}).

\begin{lem}[Frostman's lemma]
A (deterministic) metric space $(X,D)$ has Hausdorff dimension at least $\Delta > 0$ if and only if for each $d < \Delta$, $X$ supports a \emph{Frostman measure} of dimension $d$; i.e., a measure $\nu$ on $X$ with positive mass such that
\[
\iint_{X \times X} D(z,w)^{-d} d\nu(z) d\nu(w) < \infty.
\]
\end{lem}

Probabilists have used Frostman's lemma to compute the a.s.\ dimensions of many random fractals in Euclidean space, for example the graph of Brownian motion~\cite{peres-bm}, Schramm-Loewner evolution curves~\cite{beffara-dim}, conformal loop ensembles~\cite{mww-nesting}, and thick points of a GFF~\cite{hmp-thick-pts}, via the following general approach. To show that the random fractal $X$ in question has Hausdorff dimension $\Delta$ almost surely, one first proves that $\dim_{\mcl H} X \leq \Delta$ by establishing an upper bound for the probability that a given point is ``close" to $X$. One then uses first and second moment estimates to prove that, for every $d < \Delta$, the fractal $X$ supports a dimension-$d$ Frostman measure with positive probability. 
The reason this holds only \emph{with positive probability} is that the method for proving that the measure has positive mass uses the Payley-Zygmund inequality.
Combining this positive probability result and the almost sure upper bound, one gets that $\Delta$ is the \emph{essential supremum} of $\dim_{\mcl H} X$; i.e.,
\[
\Delta = \sup \{ d \geq 0 : \BB{P}(\dim_{\mcl H} X \geq d) > 0 \} .
\]

To upgrade this to the statement that $\dim_{\mcl H} X = \Delta$ a.s., one typically shows that the dimension of the random fractal $X$ satisfies a zero-one law; i.e., the dimension of $X$ must equal \emph{some} deterministic constant almost surely. 
For many interesting random fractals (e.g., many sets defined in terms of SLE or Brownian motion), one has some sort of Markovian or long-range independence property which allows for a relatively straightforward proof of the zero-one law. 
However, the zero-one laws for many interesting sets associated with the LQG metric are less straightforward, in large part because (as noted above) LQG geodesics are not locally determined by $h$. 

The most difficult zero-one law argument in this paper is for the boundary of an LQG metric ball. The essential suprema of the Euclidean and $\gamma$-quantum dimensions of this fractal have already been computed in~\cite{gwynne-ball-bdy}. Combining this with our zero-one law gives the following theorem. 

\begin{thm}[Dimension of LQG ball boundaries]
\label{thm-ball-bdy-bound}
For each fixed $z \in \BB{C}$, a.s.\ the boundary of the LQG metric ball $\mcl B_{D_h(0,z)}(0;D_h)$ centered at $0$ and run until it hits $z$ has $\gamma$-quantum dimension
\eqb \label{eqn-ball-bdy-gamma}
d_\gamma - 1
\eqe
and Euclidean dimension
\eqb \label{eqn-ball-bdy-0}
2 - \frac{\gamma}{d_\gamma} \left(\frac{2}{\gamma} + \frac{\gamma}{2}\right) + \frac{\gamma^2}{2 d_\gamma^2}
\eqe
(or, equivalently, $2 - \xi Q + \xi^2/2$ with $\xi,Q$ defined in~\eqref{eqn-xi-Q}).
\end{thm}
 
Since $d_{\sqrt{8/3}}=4$, for $\gamma = \sqrt{8/3}$ the quantum and Euclidean dimensions in Theorem~\ref{thm-ball-bdy-bound} are equal to $3$ and $5/4$, respectively. 

We next consider $\gamma$-LQG geodesics. We know that a.s.\ the $\gamma$-quantum dimension of every such geodesic equals $1$ (the dimension of a geodesic is always equal to 1 w.r.t.\ the metric for which it is a geodesic). We do not even have a conjecture for the Euclidean dimension, although~\cite[Corollary 1.10]{gp-kpz} gives a rigorous upper bound. In this paper we establish the following zero-one law.

\begin{thm}[Zero-one law for LQG geodesics] \label{thm-geo-dim}
There is a deterministic constant $\Delta_{\op{geo}} > 0$ such that a.s.\ the Euclidean dimension of every $D_h$-geodesic started from 0 is equal to $\Delta_{\op{geo}}$. 
\end{thm}

As a ``warm-up" for the proofs of Theorems~\ref{thm-ball-bdy-bound} and~\ref{thm-geo-dim}, we will also prove a zero-one law for the so-called \emph{metric net}, which is much easier than the proofs in the case of ball boundaries and geodesics. To define the metric net, we first introduce the notion of a \emph{filled metric ball}. To motivate the definition, we note that the complement of an LQG metric ball is typically not connected.

\begin{defn}[The filled metric ball] \label{def-filled}
Let $w \in \BB{C}$ and $z \in \mathbb{C} \cup \{\infty\}$.  We define the \emph{filled metric ball centered at $w$ and targeted at $z$} with radius $s>0$ as
\[
\mcl B^{z,\bullet}_s(w;D_h) :=
\begin{cases} 
\text{the union of the closed metric ball $\ol{\mcl B_s(w;D_h)}$} \\
\text{and the set of points that this closed}  &\qquad\mbox{for } s < D_h(w,z) \\
\text{metric ball disconnects from $z$} \\\\
\BB{C} &\qquad\mbox{for } s  \geq D_h(w,z)\\\\
\end{cases}
\]
We will most often work with filled metric balls centered at zero and filled metric balls targeted at infinity, so to lighten notation, we abbreviate
\eqb
\mcl B^\bullet_s(w;D_h) := \mcl B^{\infty,\bullet}_s(w;D_h), \qquad
\mcl B_s^{z,\bullet} := \mcl B_s^{z,\bullet}(0;D_h) \quad \text{and} \quad \mcl B_s^\bullet := \mcl B_s^\bullet(0;D_h).
\eqe
\end{defn}


The metric net is the region of space traced by the boundary of a growing filled metric ball targeted at infinity. 

\begin{defn}[The metric net]
\label{def-metric-net}
The \emph{metric net} at time $s > 0$ is 
\eqb \label{eqn-metric-net}
\mcl N_s(w;D_h) := \bigcup_{t \leq s} \bdy \mcl B_t^\bullet(w;D_h).
\eqe 
The metric net at time infinity is $\mcl N_\infty(w;D_h) := \bigcup_{s > 0} \mcl N_s(w;D_h)$. We abbreviate $\mcl N_s  := \mcl N_s (0;D_h)$.
\end{defn}

The metric net at time $s$ is a closed subset of the complex plane.  If we were working in a smooth metric space, the metric net would have full Lebesgue measure and Hausdorff dimension; and at time infinity, it would be the entire complex plane.  In contrast, the metric net in the LQG metric space has \emph{zero} Lebesgue measure almost surely, since the probability that any fixed point $z \in\BB C$ lies on the boundary of the unbounded connected component of $\BB C\setminus \mcl B_{D_h(w,z)}(w;D_h)$ is zero (see, e.g., the argument of Section~\ref{sec-outer}).  In Section~\ref{sec-net-dim}, we analyze the $\gamma$-quantum and Euclidean dimension of the metric net, and we show that the scale invariance of the metric net (Lemma~\ref{lem-net-scale}) and a locality property of the metric net (Lemma~\ref{lem-net-local}) easily imply a zero-one law.

\begin{thm} \label{thm-net-dim}
There are deterministic constants $\Delta_{\op{net	}}^0, \Delta_{\op{net	}}^\gamma > 0$ such that a.s.\ $\dim_{\mcl H}^0 \mcl N_s  = \Delta_{\op{net}}^0$ and $\dim_{\mcl H}^\gamma \mcl N_s  = \Delta_{\op{net}}^\gamma$ for every $s >0$. 
\end{thm}

Proving zero-one laws for LQG geodesics and metric ball boundaries is considerably more challenging: these fractals are neither scale-invariant nor locally determined by the underlying field.  We overcome this hurdle as follows.
\begin{itemize}
\item
Since
LQG geodesics and the boundaries of LQG metric balls are not themselves scale-invariant,  we instead consider an appropriate ``infinite-volume" object whose law is exactly scale invariant. 
In the case of geodesics, this object is an infinite geodesic ray from 0 to $\infty$. In the case of the metric ball boundary, this object is the boundary of a ``metric ball started from $\infty$ and grown until it hits 0", which will be defined as the limit of $\mcl B_{D_h(0,w)}(w;D_h)$ as $w\rta\infty$. 
We construct these objects in Section~\ref{sec-infty} using the strong confluence of geodesics property established in Theorem~\ref{thm-general-confluence}.
\item
To address the issue that LQG geodesics and the boundaries of LQG metric balls are not locally determined by the underlying field, we again use the strong confluence of geodesics property established in Theorem~\ref{thm-general-confluence}.  Specifically, we apply this theorem to construct ``good" events on which the random fractals are in some sense locally determined (see Lemma~\ref{lem-conf-event}).
\end{itemize}
This strategy yields zero-one laws (Propositions~\ref{prop-infty-geo} and~\ref{prop-ball-infty-dim}) for the infinite-volume versions of LQG geodesics and metric ball boundaries.  We then transfer these results from the infinite-volume setting to the finite-volume objects appearing in Theorems~\ref{thm-ball-bdy-bound} and~\ref{thm-geo-dim}.

For metric ball boundaries, we are able to derive not just a zero-one law, but explicit expressions~\eqref{eqn-ball-bdy-gamma} and~\eqref{eqn-ball-bdy-0} for the Euclidean and $\gamma$-quantum dimensions.  This is because we can apply the work of~\cite{gwynne-ball-bdy}, which identified~\eqref{eqn-ball-bdy-gamma} and~\eqref{eqn-ball-bdy-0} as the essential suprema of the dimension of an LQG metric ball of a fixed radius with respect to the $\gamma$-LQG and Euclidean metrics, respectively.  We apply these results to prove Theorem~\ref{thm-ball-bdy-bound} by transferring the results for metric balls of a fixed radius to the case of metric balls run until they hit a fixed \emph{point}---the type of metric ball for which we have a zero-one law.

\begin{remark}[Intersection with the the thick points] \label{remark-thick-pts}
Following~\cite{gwynne-ball-bdy}, for $\alpha \in [-2,2]$ we define the set of \emph{metric $\alpha$-thick points} of $h$ by
\eqb \label{eqn-thick-pts}
\wh{\mcl T}_h^\alpha := \left\{z\in\BB C : \lim_{\ep\rta 0} \frac{\log \sup_{u,v\in B_\ep(z)} D_h(u,v)}{\log \ep } = \frac{\gamma}{d_\gamma}\left(2/\gamma+\gamma/2 - \alpha \right) \right\}
\eqe
where $B_\ep(z)$ is the Euclidean ball of radius $\ep$ centered at $z$. The reason for the somewhat strange looking number $\frac{\gamma}{d_\gamma}\left(2/\gamma+\gamma/2 - \alpha \right)$ in~\eqref{eqn-thick-pts} is that this makes it so that $\wh{\mcl T}_h^\alpha$ has similar properties to the ordinary $\alpha$-thick points, as considered, e.g., in~\cite{hmp-thick-pts}. 
Essentially the same arguments as in the proofs of Theorems~\ref{thm-ball-bdy-bound}, \ref{thm-geo-dim}, and Theorem~\ref{thm-net-dim} also yield zero-one laws for the Euclidean and $\gamma$-LQG dimensions of the intersection of the random fractals in the theorem statements with $\wh{\mcl T}_h^\alpha$. In the setting of Theorem~\ref{thm-ball-bdy-bound}, we get a formula for the dimension of the intersection in terms of $\gamma,d_\gamma,\alpha$ from~\cite[Theorem 1.2]{gwynne-ball-bdy}. Similar considerations also apply with the metric thick points replaced by the ordinary thick points (defined using circle averages as in~\cite{hmp-thick-pts,shef-kpz}).
\end{remark}

\subsection{Exterior boundaries of metric balls}

Finally, in Section~\ref{sec-outer}, we will analyze the \emph{exterior boundaries} of LQG metric balls.  

\begin{defn}[The exterior boundary of an LQG metric ball]
\label{def-outer}
Fix $w\in \BB{C}$ and a radius $s  > 0$, we define the \emph{exterior boundary} $\mcl O_s(w;D_h)$ of the metric ball $\mcl B_s(w;D_h)$ to be the union of the boundaries of the connected components of $\BB{C} \backslash \ol{\mcl B_s(w;D_h)}$. Equivalently, $\mcl O_s(w;D_h)$ is the union of the boundaries of the filled metric balls $\mcl B^{z,\bullet}_{s}(w;D_h)$ over all $z\in\BB Q^2$.
\end{defn}

For a smooth metric, the notions of boundary and exterior boundary of a metric ball are equivalent.  This is not the case for LQG metric balls.  The points on the boundary of $\mcl B_{D_h(z,w)}(w;D_h)$ that are not on the exterior boundary  
arise as accumulation points of connected components of $\BB C\setminus \mcl B_s$ with arbitrarily small diameters. Our main result for exterior boundaries of metric balls is the following theorem, which asserts that ``most" points of an LQG metric ball boundary are \emph{not} in the exterior boundary.

\begin{thm} \label{thm-outer-bdy-compare}
There is a constant $q > 0$ such that for each fixed $s > 0$ and each fixed $z\in\BB C$, the Euclidean and $\gamma$-quantum dimensions of $\mcl O_{s}(0;D_h)$ are a.s.\ at most  \eqbn
2- \frac{\gamma}{d_\gamma}\left( \frac{2}{\gamma}+\frac{\gamma}{2} \right) + \frac{\gamma^2}{2d_\gamma^2} -q \quad \text{and} \quad d_\gamma - 1 - q , 
\eqen
respectively. The same dimension upper bounds hold for $\mcl O_{D_h(0,z)}(0;D_h)$.
In particular, due to~\cite[Theorem 1.1]{gwynne-ball-bdy} (resp.\ Theorem~\ref{thm-ball-bdy-bound}), with positive probability (resp.\ almost surely), the set of points in $\bdy\mcl B_s(0;D_h)$ (resp.\ $\bdy \mcl B_{D_h(0,z)}(0;D_h)$) which do not lie on the boundary of any complementary connected component of $\mcl B_s(0;D_h)$ (resp.\ $\mcl B_{D_h(0,z)}(0;D_h)$) has full Hausdorff dimension, hence is uncountable. 
\end{thm}

The proof of Theorem~\ref{thm-outer-bdy-compare}, given in Section~\ref{sec-outer}, is based on a generalization of the argument used to prove the upper bound for the Euclidean and LQG dimensions of an LQG metric ball boundary in~\cite{gwynne-ball-bdy}, see Theorem~\ref{thm-gen-upper}. 
 

\subsection{Outline}
 
In Section~\ref{sec-metric-def}, we review the definition of the $\gamma$-LQG metric.  In Section~\ref{sec-strong-confluence}, we prove our main confluence result (Theorem~\ref{thm-general-confluence}). In Section~\ref{sec-two-ingredients}, we prove a zero-one law for the metric net (Theorem~\ref{thm-net-dim}), and we describe the versions of scale invariance and locality that we will use to prove zero-one laws for LQG geodesics and metric ball boundaries. In Section~\ref{sec-geo-dim}, we prove a zero-one law for geodesics (Theorem~\ref{thm-geo-dim}). Finally, in Section~\ref{sec-ball-bdry}, we compute the Euclidean and $\gamma$-quantum dimensions of metric ball boundaries (Theorem~\ref{thm-ball-bdy-bound}), and we study the exterior boundaries of LQG metric balls (Theorem~\ref{thm-outer-bdy-compare}).

\section{Background: definition of the LQG metric}
\label{sec-metric-def}

In this section, we review the definition of the $\gamma$-LQG metric.
The $\gamma$-LQG metric can be defined in two equivalent ways: as the limit of an explicit approximation scheme (called \emph{Liouville first passage percolation}), and as the unique metric satisfying a list of axioms. In this paper, we will use the axiomatic definition  of the LQG metric, which we now state after introducing some metric space terminology.

\begin{defn}[Terminology for general metric spaces] \label{def-metric-stuff}
Let $(X,D)$ be a metric space.
\begin{itemize}
\item
For a curve $P : [a,b] \rta X$, the \emph{$D$-length} of $P$ is defined by 
\eqbn
\op{len}\left( P ; D  \right) := \sup_{T} \sum_{i=1}^{\# T} D(P(t_i) , P(t_{i-1})) 
\eqen
where the supremum is over all partitions $T : a= t_0 < \dots < t_{\# T} = b$ of $[a,b]$. Note that the $D$-length of a curve may be infinite.
\item
We say that $(X,D)$ is a \emph{length space} if for each $x,y\in X$ and each $\ep > 0$, there exists a curve of $D$-length at most $D(x,y) + \ep$ from $x$ to $y$. 
\item
For $Y\subset X$, the \emph{internal metric of $D$ on $Y$} is defined by
\eqb \label{eqn-internal-def}
D(x,y ; Y)  := \inf_{P \subset Y} \op{len}\left(P ; D \right) ,\quad \forall x,y\in Y 
\eqe 
where the infimum is over all paths $P$ in $Y$ from $x$ to $y$. 
Note that $D(\cdot,\cdot ; Y)$ is a metric on $Y$, except that it is allowed to take infinite values.  
\item
If $X$ is an open subset of $\BB C$, we say that $D$ is  a \emph{continuous metric} if it induces the Euclidean topology on $X$. 
We equip the set of continuous metrics on $X$ with the local uniform topology on $X\times X$ and the associated Borel $\sigma$-algebra.
\end{itemize}
\end{defn}

We now define the $\gamma$-LQG metric axiomatically. The definition is phrased in terms of the two parameters $Q$ and $\xi$, defined as
\eqb
Q = Q_\gamma =  \frac{2}{\gamma} + \frac{\gamma}{2} \quad \text{and} \quad \xi = \xi_\gamma := \frac{\gamma}{d_\gamma} \label{eqn-xi-Q}
\eqe 
where, as above, $d_\gamma  = \dim_{\mcl H}^\gamma \BB C$ is the LQG dimension exponent~\cite{dzz-heat-kernel,dg-lqg-dim,gp-kpz}.

\begin{defn}[The LQG metric]
\label{def-lqg-metric}
For $U\subset \BB C$, let $\mcl D'(U)$ be the space of distributions (generalized functions) on $\BB C$, equipped with the usual weak topology.   
A \emph{$\gamma$-LQG metric} is a collection of measurable functions $h\mapsto D_h$, one for each open set $U\subset\BB C$, from $\mcl D'(U)$ to the space of continuous metrics on $U$ with the following properties.\footnote{Our axioms for a $\gamma$-LQG metric only concern a.s.\ properties of $D_h$ when $h$ is a GFF plus a continuous function. So, once we have defined $D_h$ a.s.\ when $h$ is a GFF plus a continuous function, we can take $D$ to be any measurable mapping $\mcl D'(U) \rta \{\text{continuous metrics on $U$}\}$ which is a.s.\ consistent with our given definition when $h$ is a GFF plus a continuous function. 
In fact, the construction of the metric in~\cite{dddf-lfpp,lqg-metric-estimates,gm-confluence,gm-uniqueness,gm-coord-change} only gives an explicit definition of $D_h$ in the case when $h$ is a GFF plus a continuous function. }
Let $U\subset \BB C$ and let $h$ be a \emph{GFF plus a continuous function} on $U$: i.e., $h$ is a random distribution on $U$ which can be coupled with a random continuous function $f$ in such a way that $h-f$ has the law of the (zero-boundary or whole-plane, as appropriate) GFF on $U$.  Then the associated metric $D_h$ satisfies the following axioms.
\begin{enumerate}[I.]
\item \textbf{Length space.} Almost surely, $(U,D_h)$ is a length space, i.e., the $D_h$-distance between any two points of $U$ is the infimum of the $D_h$-lengths of $D_h$-continuous paths (equivalently, Euclidean continuous paths) in $U$ between the two points. \label{item-metric-length}
\item \textbf{Locality.} Let $V \subset U$ be a deterministic open set. 
The $D_h$-internal metric $D_h(\cdot,\cdot ; V)$ is a.s.\ equal to $D_{h|_V}$, so in particular it is a.s.\ determined by $h|_V$.  \label{item-metric-local}
\item \textbf{Weyl scaling.} Let $\xi = \gamma/d_\gamma$ be as in~\eqref{eqn-xi-Q}. For a continuous function $f : U\rta \BB R$, define
\eqb \label{eqn-metric-f}
(e^{\xi f} \cdot D_h) (z,w) := \inf_{P : z\rta w} \int_0^{\op{len}(P ; D_h)} e^{\xi f(P(t))} \,dt , \quad \forall z,w\in U,
\eqe 
where the infimum is over all continuous paths from $z$ to $w$ in $U$ parametrized by $D_h$-length.
Then a.s.\ $ e^{\xi f} \cdot D_h = D_{h+f}$ for every continuous function $f: U\rta \BB R$. \label{item-metric-f}
\item \textbf{Conformal coordinate change.} Let $\wt U\subset \BB C$ and let $\phi : U \rta \wt U$ be a deterministic conformal map. Then, with $Q$ as in~\eqref{eqn-xi-Q}, a.s.\ \label{item-metric-coord}
\eqb \label{eqn-metric-coord}
 D_h \left( z,w \right) = D_{h\circ\phi^{-1} + Q\log |(\phi^{-1})'|}\left(\phi(z) , \phi(w) \right)  ,\quad  \forall z,w \in U.
\eqe    
\end{enumerate}
\end{defn}

The following theorem~\cite{dddf-lfpp,gm-uniqueness,gm-coord-change} asserts that the $\gamma$-LQG metric defined in Definition~\ref{def-lqg-metric} exists and is unique.

\begin{thm}[Existence and uniqueness of the LQG metric] \label{thm-lqg-metric} 
For each $\gamma \in (0,2)$, there exists a metric satisfying the axioms of Definition~\ref{def-lqg-metric}.  This metric is unique in the following sense: if $D$ and $\wt D$ are two such metrics, then there is a deterministic constant $C>0$ such that whenever $h$ is a GFF plus a continuous function, a.s.\ $\wt D_h = C D_h$.
\end{thm}

More precisely, it is shown in~\cite[Theorem 1.2]{gm-uniqueness}, building on~\cite{dddf-lfpp,local-metrics,lqg-metric-estimates,gm-confluence}, that for each $\gamma \in (0,2)$, there is a measurable function $h\mapsto D_h$ from $\mcl D'(\BB C)$ to the space of continuous metrics on $\BB C$ which satisfies the conditions of Definition~\ref{def-lqg-metric} for $U=\BB C$ (note that this means $\phi$ in Axiom~\ref{item-metric-coord} is required to be a complex affine map) and is unique in the sense of Theorem~\ref{thm-lqg-metric}. As explained in~\cite[Remark 1.5]{gm-uniqueness}, this gives a way to define $D_h$ whenever $h$ is a GFF plus a continuous function on an open domain $U\subset\BB C$ in such a way that Axioms~\ref{item-metric-length} through~\ref{item-metric-f} hold.
Note that the metric in the whole-plane case determines the metric on other domains due to Axiom~\ref{item-metric-local}.
It is shown in~\cite[Theorem 1.1]{gm-coord-change} that with the above definition, Axiom~\ref{item-metric-coord} holds for general conformal maps. 
 
Because of Theorem~\ref{thm-lqg-metric}, we may refer to the unique metric satisfying Definition~\ref{def-lqg-metric} as \emph{the $\gamma$-LQG metric}. Technically, the metric is unique only up to a global deterministic multiplicative constant.  When referring to the $\gamma$-LQG metric, we are implicitly fixing the constant in some arbitrary way. For example, we could require that the median distance between the left and right sides of $[0,1]^2$ is 1 when $h$ is a whole-plane GFF normalized so that its average over the unit circle is zero. The choice of constant will not play any role in our results or proofs.

 \subsection{Boundaries of filled LQG metric balls are Jordan curves}
\label{sec-metric-curve}

Here we record a basic topological fact about LQG metric balls which will be convenient to use at several places later in the paper. The lemma is a minor variant of~\cite[Proposition 2.1]{tbm-characterization} and is proven in the same way. 

\begin{lem} \label{lem-metric-curve}
Let $D$ be a length metric on $\BB C$ which induces the same topology as the Euclidean metric and satisfies $\lim_{z\rta\infty} D(w,z) = \infty$ for some (equivalently, every) $w\in\BB C$. 
For $z\in\BB C$ and $s > 0$, write $\mcl B_s(z;D)$ for the $D$-metric ball of radius $s$ centered at $z$. 
The boundary of each connected component of $\BB C\setminus \ol{\mcl B_s(z;D)}$ is a Jordan curve.
\end{lem}
\begin{proof}
Our hypotheses on $D$ imply that $(\BB C,D)$ is boundedly compact, i.e., each closed $D$-bounded subset of $\BB C$ is compact.
By~\cite[Corollary 2.5.20]{bbi-metric-geometry}, this implies that there is a $D$-geodesic between any two points of $\BB C$.
The lemma now follows from exactly the same argument as~\cite[Proposition 2.1]{tbm-characterization}, which is the analogous statement for metrics on the sphere rather than the plane.
\end{proof}

Note that the hypotheses of Lemma~\ref{lem-metric-curve} apply a.s.\ when $D = D_h$ is the $\gamma$-LQG metric associated with a whole-plane GFF (see~\cite[Lemma 3.8]{lqg-metric-estimates} for a proof that $\lim_{z\rta \infty} D(w,z) = \infty$). 

We also note that each connected component $U$ of $\BB C\setminus \ol{\mcl B_s(z;D)}$ is simply connected (since it is connected with connected complement). Since the boundary of such a connected component is a Jordan curve, it follows from Carath\'eodory's theorem~\cite[Theorem 2.6]{pom-book} that if $U$ is bounded, then any conformal map $\phi : \BB D \rta U$ extends to a homeomorphism $\ol{\BB D} \rta \ol U$. If $U$ is unbounded, the same is true with $U$ replaced by $U\cup \{\infty\}$ (viewed as a subset of the Riemann sphere). 

Strictly speaking, Lemma~\ref{lem-metric-curve} is not actually necessary in this paper since, whenever we would be inclined to use it, we can work with prime ends instead of actual boundary points. However, the lemma allows us to avoid some technical annoyances and makes some geometric arguments more transparent.

\section{Strong confluence of LQG geodesics}
\label{sec-strong-confluence}

The goal of this section is to prove Theorem~\ref{thm-general-confluence}. We will make frequent use of the notation for filled LQG metric balls from Definition~\ref{def-filled}. 

\subsection{Confluence at a single point}
\label{sec-conf}

In this subsection we will review some results concerning confluence of geodesics from~\cite{gm-confluence} and also prove some minor improvements on these results.
To do so we first recall the notion of a \emph{leftmost geodesic}. 
Each point $x\in \bdy \mcl B_s^{z,\bullet}(w;D_h) $ lies at $D_h$-distance exactly $s$ from $w$, so every $D_h$-geodesic from $w$ to $x$ stays in $\mcl B_s^{z,\bullet}(w;D_h)$. For some points $x$ there might be many such $D_h$-geodesics. But, it is shown in~\cite[Lemma 2.4]{gm-confluence} that there is always a distinguished $D_h$-geodesic from $w$ to $x$, called the \emph{leftmost geodesic}, which lies (weakly) to the left of every other $D_h$-geodesic from $w$ to $x$ if we stand at $x$ and look outward from $\mcl B_s^{z,\bullet}(w;D_h)$. Strictly speaking,~\cite[Lemma 2.4]{gm-confluence} only treats the case of filled metric balls targeted at $\infty$, but the same proof works for filled metric balls with different target points. 

The following theorem is a compilation of results from~\cite{gm-confluence}.
See Figure~\ref{fig-conf}, right, for an illustration.

\begin{thm}[\cite{gm-confluence}] \label{thm-conf}
Almost surely, for every $0 < t  < s$ the following is true. 
\begin{enumerate}
\item There is a finite set of points $\mcl X = \mcl X_{t,s} \subset \bdy\mcl B_t^\bullet$ such that every leftmost $D_h$-geodesic from 0 to a point of $\bdy \mcl B_s^\bullet$ passes through some $x\in\mcl X$. \label{item-conf}
\item There is a unique $D_h$-geodesic from 0 to $x$ for each $x\in \mcl X$. \label{item-conf-unique}
\item For $x\in \mcl X$, let $I_x$ be the set of $y\in\bdy\mcl B_s^\bullet$ such that that the leftmost $D_h$-geodesic from 0 to $y$ passes through $x$.
Each $I_x $ for $x\in\mcl X $ is a connected arc of $\bdy\mcl B_s^\bullet $ (possibly a singleton) and $\bdy\mcl B_s^\bullet$ is the disjoint union of the arcs $I_x $ for $x\in \mcl X $. \label{item-conf-arcs}
\item The counterclockwise cyclic ordering of the arcs $I_x $ is the same as the counterclockwise cyclic ordering of the corresponding points $x\in \mcl X \subset\bdy \mcl B_t^\bullet$. \label{item-conf-order}
\end{enumerate}
\end{thm}

We note that $\bdy\mcl B_s^\bullet$ is a Jordan curve by Lemma~\ref{lem-metric-curve}. This allows us to talk about arcs of $\bdy\mcl B_s^\bullet$ without worrying about prime ends, etc.

\begin{proof}[Proof of Theorem~\ref{thm-conf}]
Assertion~\ref{item-conf} is immediate from~\cite[Theorem 1.4]{gm-confluence}.
Assertion~\ref{item-conf-unique} can be easily deduced from the uniqueness of geodesics between rational points together with the fact that leftmost $D_h$-geodesics can be approximated by geodesics to rational points~\cite[Lemma 2.4]{gm-confluence}; see the proof of~\cite[Theorem 3.1]{gm-confluence}.
Assertion~\ref{item-conf-arcs} follows from~\cite[Lemma 2.7]{gm-confluence}.
Assertion~\ref{item-conf-order} is implicit in the proof of~\cite[Lemma 2.7]{gm-confluence}, or alternatively can be extracted from the fact that distinct leftmost $D_h$-geodesics from 0 to points of $\bdy\mcl B_s$ do not cross. 
\end{proof}

\begin{figure}[t!]
 \begin{center}
\includegraphics[scale=.85]{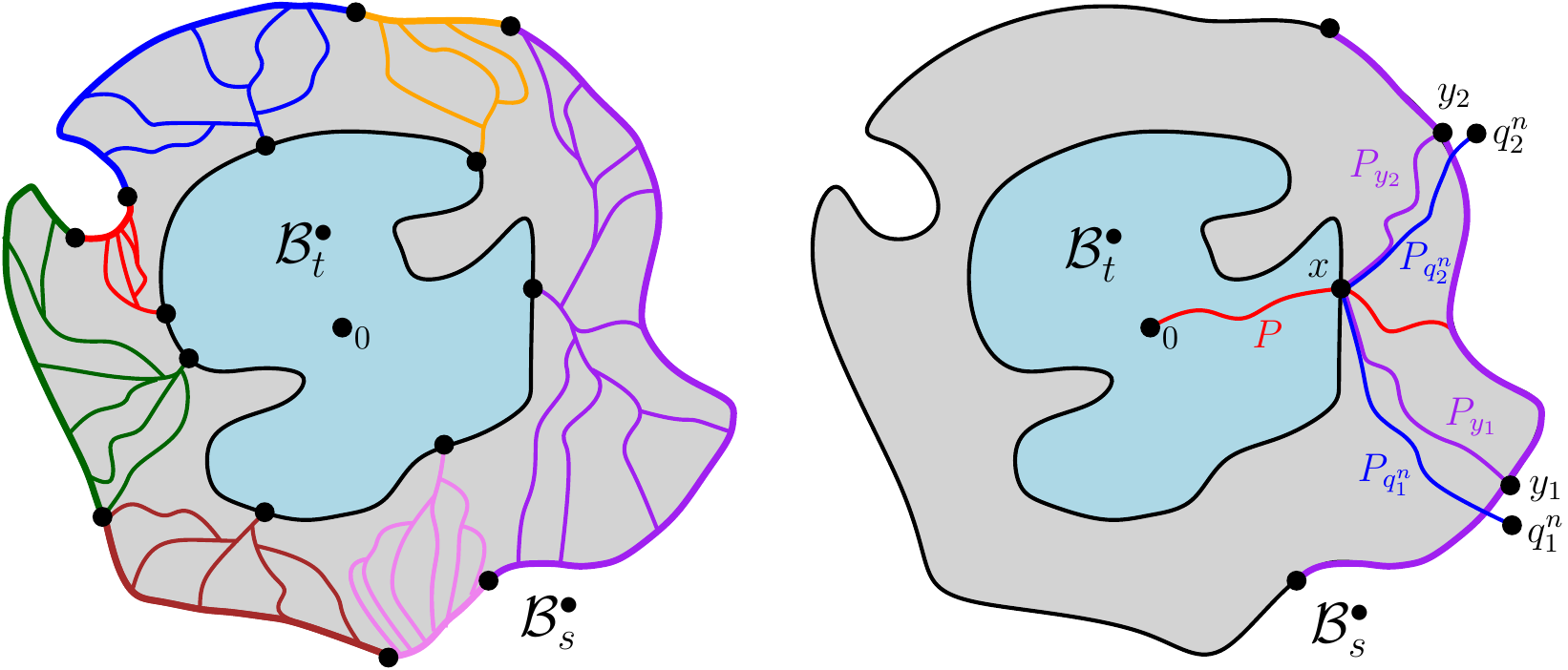}
\vspace{-0.01\textheight}
\caption{ \textbf{Left:} Illustration of the statement of Theorem~\ref{thm-conf}. The confluence points $\mcl X\subset\bdy\mcl B_t^\bullet$ are shown in black. For each $x\in\mcl X$, we have shown in a different color the arc $I_x \subset\bdy\mcl B_s^\bullet$ and the segments $P([s,t])$ for several representative leftmost $D_h$-geodesics $P$ from $x$ to points of $\bdy\mcl B_s^\bullet$. 
\textbf{Right:} Illustration of the proof of Lemma~\ref{lem-conf-no-leftmost}. We want to show that the given red $D_h$-geodesic $P$ (which is not necessarily a leftmost $D_h$-geodesic) passes through $x$. To do this, we construct $D_h$-geodesics from 0 to points $q_1^n,q_2^n\in\BB Q^2\setminus \mcl B_s^\bullet$ which pass through $x$ (blue) with the property that if $P$ does not pass through $x$, then it must cross one of these two blue geodesics. 
}\label{fig-conf}
\end{center}
\vspace{-1em}
\end{figure}

The following minor improvement on Theorem~\ref{thm-conf} allows us to avoid worrying about whether $D_h$-geodesics are leftmost and about what happens at the endpoints of the arcs $I_x$. 
Note that now we fix $t$ and $s$, whereas Theorem~\ref{thm-conf} is required to hold simultaneously for all $t$ and $s$. 

\begin{prop} \label{prop-conf-endpt}
Fix $0<t < s$ and let $\mcl X = \mcl X_{t,s}$ be the set of confluence points as in Theorem~\ref{thm-conf}. 
Almost surely, for every $D_h$-geodesic $P$ from 0 to a point of $\BB C\setminus \mcl B_s^\bullet$ there is an $x\in\mcl X$ such that $P(t)  = x$ and $P(s) $ is a point of the arc $I_x$ which is not one of the endpoints of $I_x$. 
\end{prop}

The rest of this subsection is devoted to the proof of Proposition~\ref{prop-conf-endpt}. 
We first establish an improvement on Theorem~\ref{thm-conf} which does not require leftmost geodesics.

\begin{lem} \label{lem-conf-no-leftmost}
Almost surely, for each $0 < t < s$ the following is true. 
Define $\mcl X = \mcl X_{t,s}$ and the arcs $I_x \subset\bdy\mcl B_s^\bullet$ for $x\in\mcl X$ as in Theorem~\ref{thm-conf}.
If $P$ is a $D_h$-geodesic from 0 to a point of $I_x$ (not necessarily a leftmost $D_h$-geodesic) and $P(s)$ is not one of the endpoints of $I_x$, then $P(t) = x$.
\end{lem}
\begin{proof}
See Figure~\ref{fig-conf}, right, for an illustration. Throughout the proof, we let $P$ be a geodesic as in the lemma statement and we omit the qualifier ``a.s.".

Choose $y_1,y_2 \in I_x$ such that the counterclockwise arc of $\bdy \mcl B_s^\bullet$ from $y_1$ to $y_2$ is contained in $I_x$ and contains $P(s)$.
We can arrange that neither $y_1$ nor $y_2$ is equal to $P(s)$ or to one of the endpoints of $I_x$.
Let $P_{y_1}$ (resp.\ $P_{y_2}$) be the leftmost $D_h$-geodesic from 0 to $y_1$ (resp.\ $y_2$). Then $P_{y_1}(t) = P_{y_2}(t) = x$.
By~\cite[Lemma 2.4]{gm-confluence}, there are sequences of points $q_1^n , q_2^n \in \BB Q^2 \setminus \mcl B_s^\bullet$ such that the following is true.
If we let $P_{q_1^n} $ be the (a.s.\ unique) $D_h$-geodesic from 0 to $q_1^n$, then $P_{q_1^n} \rta P_{y_1}$ uniformly; and the same is true with 2 in place of 1. 

For each $n$, $P_{q_1^n}|_{[0,s]}$ is the unique (hence leftmost) geodesic from 0 to $q_1^n$, and the same is true with 2 in place of 1. 
Hence each of $P_{q_1^n}(t)$ and $P_{q_2^n}(t)$ must belong to the finite set $\mcl X$. 
By the above uniform convergence for large enough $n$ we have $P_{q_1^n}(t) = P_{q_2^n}(t) = x$. 
Furthermore, for large enough $n$ the counterclockwise arc of $\bdy \mcl B_s^\bullet$ from $P_{q_1^n}(s)$ to $P_{q_2^n}(s)$ contains $P(s)$. 
Henceforth assume $n$ is large enough that these conditions are satisfied.

If $P(t)\not= x$, then topological considerations imply that $P$ must intersect either $P_{q_1^n}(s)$ or $P_{q_2^n}(s)$ between time $t$ and time $s$. 
Assume without loss of generality that $P$ hits $P_{q_1^n}$. 
By the uniqueness of $D_h$-geodesics to rational points (see~\cite[Lemma 2.3]{gm-confluence}) if this is the case then there must be a time $\tau \in [t,s]$ such that $P|_{[0,\tau]} = P_{q_1^n}|_{[0,\tau]}$. 
But, this implies that $P(t) = P_{q_1^n}(t) = x$. 
\end{proof}

We now argue that $D_h$-geodesics cannot hit the endpoints of the arcs $I_x$. 

\begin{lem} \label{lem-conf-endpt'}
Fix $0 < t < s$. Almost surely, no $D_h$-geodesic from 0 to a point of $\BB C\setminus \mcl B_s^\bullet$ hits any of the endpoints of the arcs $I_x$ for $x\in\mcl X$.
\end{lem}

We will prove Lemma~\ref{lem-conf-endpt'} using~\cite[Lemma 3.6]{gm-confluence}, which we re-state just below, and which allows us to prevent $D_h$-geodesics from hitting particular points of $\bdy\mcl B_s^\bullet$.

\begin{lem}[\cite{gm-confluence}] \label{lem-geo-kill}
Let $\tau$ be a stopping time for the filtration generated by $(\mcl B_s^\bullet , h|_{\mcl B_s^\bullet})$ and let $y \in\bdy\mcl B_\tau^\bullet$ and $\ep \in (0,1)$ be chosen in a $\sigma(\mcl B_\tau^\bullet , h|_{\mcl B_\tau^\bullet})$-measurable manner. 
There exists an event $G_y^\ep \in \sigma(h)$ and a $\sigma(h)$-measurable random variable $R^\ep(\mcl B_\tau^\bullet) \in (0,\infty)$ (which does not depend on $y$) such that $R^\ep(\mcl B_\tau^\bullet) \rta 0$ in probability as $\ep\rta 0$ and the following is true. 
\begin{enumerate}[A.]
\item If $R^\ep(\mcl B_\tau^\bullet ) \leq \op{diam} \mcl B_\tau^\bullet$ and $G_y^\ep$ occurs, then no $D_h$-geodesic from 0 to a point of $\BB C\setminus B_{R^\ep(\mcl B_\tau^\bullet)}(\mcl B_\tau^\bullet)$ can enter $B_\ep(y) \setminus \mcl B_\tau^\bullet$ (recall that $B_r(\cdot)$ denotes a Euclidean neighborhood).  \label{item-geo-kill-enter}
\item There are deterministic constants $C,\alpha > 0$ depending only on $\gamma$ such that a.s.\ $\BB P\left[ G_y^\ep \,|\, \mcl B_\tau^\bullet , h|_{\mcl B_\tau^\bullet} \right] \geq 1 - C \ep^\alpha$. \label{item-geo-kill-prob}
\end{enumerate} 
\end{lem}
\begin{proof}
This is~\cite[Lemma 3.6]{gm-confluence} with $\BB r =1$. Note that the random variable $R^\ep(\mcl B_\tau^\bullet)$ is defined in~\cite[Equation (3.16)]{gm-confluence} and, as explained just after that equation, it converges to zero in probability as $\ep\rta 0$ by~\cite[Lemma 3.5]{gm-confluence}. 
\end{proof}
 
In the setting of Lemma~\ref{lem-geo-kill}, roughly speaking, $G_y^\ep$ is the event that there is a ``shield" around $B_\ep(y) \setminus \mcl B_\tau^\bullet$ in $\BB C\setminus \mcl B_\tau^\bullet$ which no $D_h$-geodesic started from 0 can pass through. The number $R^\ep(\mcl B_\tau^\bullet)$ is the maximum possible Euclidean radius of one of these shields.
 
\begin{proof}[Proof of Lemma~\ref{lem-conf-endpt'}]
Let $\mcl Y$ be the set of endpoints of the arcs $I_x$ for $x\in\mcl X$. 
For $y\in\mcl Y$, let $G_y^\ep$ be the event of Lemma~\ref{lem-geo-kill} with $\tau = s$. 
Also let $R^\ep(\mcl B_s^\bullet)$ be as in Lemma~\ref{lem-geo-kill}. 
Since $\mcl Y$ is a finite, $\sigma(\mcl B_\tau^\bullet , h|_{\mcl B_\tau^\bullet})$-measurable set, it follows from Lemma~\ref{lem-geo-kill}\ref{item-geo-kill-prob} and the fact that $R^\ep(\mcl B_\tau^\bullet) \rta 0$ in probability as $\ep\rta 0$ that
\eqb
\lim_{\ep\rta 0} \BB P\left[\left\{ R^\ep(\mcl B_\tau^\bullet) \leq \op{diam} \mcl B_\tau^\bullet\right\} \cap \bigcap_{y\in \mcl Y} G_y^\ep \right] = 1 .
\eqe
The lemma statement now follows from Lemma~\ref{lem-geo-kill}\ref{item-geo-kill-enter} together with the fact that $R^\ep(\mcl B_\tau^\bullet) \rta 0$ in probability. 
\end{proof}

\begin{proof}[Proof of Proposition~\ref{prop-conf-endpt}]
Let $P$ be a $D_h$-geodesic from 0 to a point of $\BB C\setminus \mcl B_s^\bullet$. 
By Lemma~\ref{lem-conf-endpt'}, there is an $x\in\mcl X$ such that $P(s) \in I_x$ and $P(s)$ is not one of the endpoints of $I_x$.
By Lemma~\ref{lem-conf-no-leftmost}, applied to the geodesic $P|_{[0,s]}$, we have $P(t) = x$.
\end{proof}

\subsection{Confluence across metric annuli with finite target points}
\label{sec-conf-finite}

The results of Section~\ref{sec-conf} concern geodesic across an annulus between two filled metric balls targeted at $\infty$.
We now show that the same results are also true for filled metric balls targeted at any $z\in\BB C$ using the conformal covariance of the metric~\cite{gm-coord-change} and local absolute continuity.
Due to the translation invariance of the law of $h$, modulo additive constant, it is no loss of generality to restrict attention to filled metric balls centered at 0. 
The following is a generalization of Theorem~\ref{thm-conf}.

\begin{prop} \label{prop-conf-finite} 
Almost surely, for each $z\in\BB C\setminus \{0\}$ and each $0 < t < s < D_h(0,z)$, the following is true.
\begin{enumerate}
\item There is a finite set of points $\mcl X = \mcl X_{t,s}^z \subset \bdy\mcl B_t^{z,\bullet}$ such that every leftmost $D_h$-geodesic from 0 to a point of $\bdy \mcl B_s^{z,\bullet}$ passes through some $x\in\mcl X$. \label{item-conf-finite}
\item There is a unique $D_h$-geodesic from 0 to $x$ for each $x\in \mcl X$. \label{item-conf-finite-unique}
\item For $x\in \mcl X$, let $I_x$ be the set of $y\in\bdy\mcl B_s^{z,\bullet}$ such that that the leftmost $D_h$-geodesic from 0 to $y$ passes through $x$.
Each $I_x $ for $x\in\mcl X $ is a connected arc of $\bdy\mcl B_s^{z,\bullet} $ (possibly a singleton) and $\bdy\mcl B_s^{z,\bullet}$ is the disjoint union of the arcs $I_x $ for $x\in \mcl X $. \label{item-conf-finite-arcs}
\item The counterclockwise cyclic ordering of the arcs $I_x $ is the same as the counterclockwise cyclic ordering of the corresponding points $x\in \mcl X \subset\bdy \mcl B_t^{z,\bullet}$. \label{item-conf-finite-order} 
\end{enumerate}
\end{prop}
\begin{proof}
The analogous statement for filled metric balls targeted at $\infty$ instead of $z$ is Theorem~\ref{thm-conf}. 
To get the desired statement for filled metric balls targeted at $z$, we use a conformal invariance argument.
We first consider a fixed choice of target point $z\in\BB C\setminus \{0\}$ (we will explain how to get the proposition for all $z$ simultaneously at the end of the proof). 
 
Let $\phi(w) := z w/(w-z)$, so that $\phi(0) =0$, $\phi(z) = \infty$, and $\phi(\infty) = z$. 
Define the field 
\eqb \label{eqn-conf-field}
\wt h := h\circ\phi^{-1}  + Q\log |(\phi^{-1})'| . 
\eqe
By the LQG coordinate change formula (Axiom~\ref{item-metric-coord}), a.s.\ $D_{\wt h}(\phi(u) , \phi(v)) = D_h(u,v)$ for each $u,v\in \BB C$. 
Consequently, a.s.\ $D_{\wt h}(0,z) = \infty$ and 
\eqb
\mcl B_s^{z,\bullet}(0;D_{\wt h}) = \phi\left(\mcl B_s^\bullet \right) ,\quad\forall s >0. 
\eqe
Therefore, the statement of the proposition is true with $\wt h$ in place of $h$. 

By the conformal invariance of the whole-plane GFF, modulo additive constant, the law of the field $\wt h$ of~\eqref{eqn-conf-field} is that of a whole-plane GFF (with some choice of additive constant) plus a deterministic function which is smooth on $\BB C\setminus \{0,z\}$. 
Consequently, for each bounded open set $U\subset \BB C$ which lies at positive distance from $z$, the laws of $h|_U$ and $\wt h|_U$, viewed modulo additive constant, are mutually absolutely continuous (see, e.g.,~\cite[Proposition 2.9]{ig4}). 

Now fix such an open set $U$ with contains 0. 
By the locality of the metric (Axiom~\ref{item-metric-local}), if $S > 0$ then on the event $\{\mcl B_S \subset U\}$ the metric ball $\mcl B_S $ and the restriction $h|_{\mcl B_S }$ are both a.s.\ determined by $h|_U$. 
Almost surely, each filled metric ball $ \mcl B_s^{z,\bullet} $ for $s\in [0,S]$ is determined by $\mcl B_S $ and $h|_{\mcl B_S }$.
Moreover, each $D_h$-geodesic from 0 to a point of $ \mcl B_S $ is contained in $ \mcl B_S $, which implies that $D_h(0,u) = D_h(0, u ; \mcl B_S )$ for each $u\in \mcl B_S $. 

For $S > 0$, let $E_S$ be the event that $  D_h(0,z) < S$ and the four conditions of the proposition statement hold with the given choice of $z$ and for all $0 < t  < s  \leq S$. 
From the preceding paragraph and Axiom~\ref{item-metric-local} (locality), we infer that $E_S$ is a.s.\ determined by $h|_U$ on $\{\mcl B_S  \subset U\}$. 

By the conclusion of the preceding paragraph and the absolute continuity of the laws of $h|_U$ and $\wt h|_U$, the event $E_S$ occurs a.s.\ on the event $\{\mcl B_S  \subset U\}$. 
Letting $U$ increase to $\BB C\setminus \{  z\}$ and then sending $S\rta\infty$ shows that a.s.\ the proposition statement holds for our fixed choice of $z$. 

We now upgrade to a statement which holds for all $z\in\BB C\setminus \{0\}$ simultaneously. 
Indeed, we know from the fixed $z$ case that a.s.\ the proposition statement holds simultaneously for all $z\in\BB Q^2\setminus \{0\}$.
For any $z\in\BB C\setminus \{0\}$ and any $0 <  s < D_h(0,z)$, there exists $z' \in \BB Q^2 \setminus \mcl B_s^{z,\bullet} $. For such a choice of $z'$ we have $\mcl B_s^{z,\bullet} = \mcl B_s^{z' ,\bullet} $ for each $t\in [0,s]$. 
Hence the proposition statement for all $z\in\BB Q^2\setminus \{0\}$ implies the lemma statement for all $z\in\BB C\setminus \{0\}$.
\end{proof}

We will also need an analog of Proposition~\ref{prop-conf-endpt} for filled metric balls with arbitrary target points.

\begin{prop} \label{prop-conf-endpt-finite}
Fix $0<t < s$ and let $\mcl X = \mcl X_{t,s}^z$ be the set of confluence points as in Proposition~\ref{prop-conf-finite}. 
Almost surely, on the event $\{s < D_h(0,z)\}$, the following is true. For every $D_h$-geodesic $P$ from 0 to a point of $\BB C\setminus \mcl B_s^{z,\bullet}$ there is an $x\in\mcl X$ such that $P(t)  = x$ and $P(s) $ is a point of the arc $I_x$ which is not one of the endpoints of $I_x$. 
\end{prop}
\begin{proof}
For $\ep > 0$, let $G(\ep) = G_{t,s}^z(\ep)$ be the event that the following is true.
For every $D_h$-geodesic $P$ from 0 to a point of $\bdy \mcl B_{s+\ep}^{z,\bullet}$ there is an $x\in\mcl X$ such that $P(t)  = x$ and $P(s) $ is a point of the arc $I_x$ which is not one of the endpoints of $I_x$. 
We observe that $G(\ep) \in \sigma\left( \mcl B_{s+\ep}^{z,\bullet} , h|_{\mcl B_{s+\ep}^{z,\bullet}} \right)$ (due to Axiom~\ref{item-metric-local}). 
Furthermore, every $D_h$-geodesic from 0 to a point of $\BB C\setminus \mcl B_s^{z,\bullet}$ which lies at $D_h$-distance at least $\ep$ from $\bdy\mcl B_s^{z,\bullet}$ has a sub-path which is a $D_h$-geodesic from 0 to a point of $\bdy \mcl B_{s+\ep}^{z,\bullet}$. 
From this, we get that if $\ep' < \ep$ then $G(\ep') \subset G(\ep)$ and moreover the event described in the lemma statement is equal to $\bigcap_{\ep>0} G(\ep)$. 

Via exactly the same argument as in the proof of Proposition~\ref{prop-conf-finite}, we get that for each fixed $\ep > 0$, a.s.\ $G(\ep)$ occurs on the event $\{s+\ep  <D_h(0,z)\}$. Sending $\ep\rta 0$ now concludes the proof.
\end{proof}

\subsection{Confluence in a neighborhood of a typical point}
\label{sec-general-confluence}

In this subsection we will prove Theorem~\ref{thm-general-confluence}, following roughly the argument used to prove the Brownian map analog~\cite[Proposition 12]{akm-geodesics}.  
Our next lemma says that any two geodesics started from 0 with nearby target points coincide along a large initial segment. 
It is the $\gamma$-LQG analog of~\cite[Lemma 18]{akm-geodesics}, but it is proven in a very different manner. 

\begin{lem} \label{lem-stable}
Almost surely, the following is true for each $z\in\BB C$ such that the $D_h$-geodesic from $0$ to $z$ is unique. For each open set $U$ containing $z$, there exists an open set $U'\subset U$ containing $z$ such that each $D_h$-geodesic from 0 to a point of $U'$ coincides with the $D_h$-geodesic from 0 to $z$ outside of $U$.
\end{lem}

\begin{figure}[t!]
 \begin{center}
\includegraphics[scale=1]{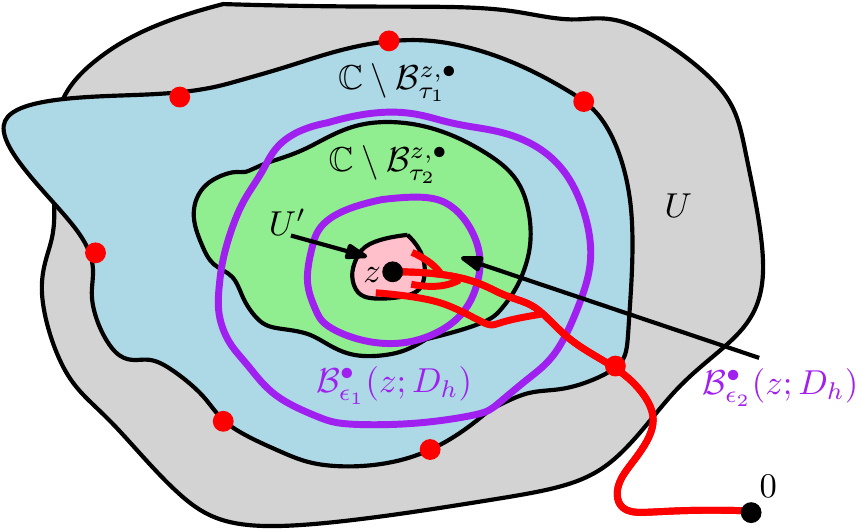}
\vspace{-0.01\textheight}
\caption{Illustration of the sets involved in the proof of Lemma~\ref{lem-stable}. 
The sets $\BB C\setminus \mcl B_{\tau_1}^{z,\bullet}$ and $\BB C\setminus\mcl B_{\tau_2}^{z,\bullet}$ are shown in blue and green, respectively. The boundaries of the $D_h$-balls of radius $\ep_1$ and $\ep_2$ centered at 0 are shown in purple. The red points are the elements of the set $X$ of points of $\bdy \mcl B_{\tau_1}^{z,\bullet} $ which are hit by leftmost or rightmost $D_h$-geodesics from 0 to $\bdy\mcl B_{\tau_2}^{z,\bullet} $. 
This set is finite by Proposition~\ref{prop-conf-finite}. The red curves are $D_h$-geodesics from 0 to points in $U'$, which all coincide outside of $U$. 
}\label{fig-stable}
\end{center}
\vspace{-1em}
\end{figure} 

We recall that it is shown in~\cite[Theorem 1.2]{mq-geodesics} that if $z \in \BB C$ is fixed, then a.s.\ the $D_h$-geodesic from 0 to $z$ is unique.  

\begin{proof}[Proof of Lemma~\ref{lem-stable}] 
By possibly shrinking $U$, we can assume without loss of generality that $U$ is bounded and that $U$ lies at positive distance from 0. 
Let $0 < \ep_2 < \ep_1 < D_h(z,\bdy U)/3$. 
Let $\tau_1$ (resp.\ $\tau_2$) be the smallest $s  >0$ for which $\mcl B_s^{z,\bullet} $ intersects $\mcl B_{\ep_1}(z;D_h)$ (resp.\ $\mcl B_{\ep_2}(z;D_h)$), so that $0 < \tau_1 < \tau_2 < D_h(0,z)$.
See Figure~\ref{fig-stable} for an illustration. 
 
By Proposition~\ref{prop-conf-finite} and its analog with rightmost geodesics in place of leftmost geodesics, there is a finite set of points $\mcl X \subset \bdy\mcl B_{\tau_1}^{z,\bullet} $ such that each leftmost or rightmost $D_h$-geodesic from 0 to a point of $\bdy \mcl B_{\tau_2}^{z,\bullet} $ passes through some $x\in \mcl X $.\footnote{We cannot apply Proposition~\ref{prop-conf-endpt-finite} instead of Proposition~\ref{prop-conf-finite} here since the radii $\tau_1$ and $\tau_2$ are random.} 
Since the $D_h$-geodesic $P$ from 0 to $z$ is unique, the restriction $P |_{[0,\tau_2]}$ is both a leftmost and a rightmost $D_h$-geodesic from 0 to $\bdy\mcl B_{\tau_2}^{z,\bullet} $, hence $P (\tau_1) \in \mcl X $.

Now consider a sequence of points $z_n\rta z$ and for each $n\in\BB N$ let $P_n$ be a $D_h$-geodesic from 0 to $z_n$ (we do not know that $P_n$ is unique). 
We claim that 
\eqb \label{eqn-stable-claim}
P_n(\tau_1)  = P(\tau_1), \quad \text{for each large enough $n\in\BB N$}.
\eqe 
Indeed, each $P_n$ is 1-Lipschitz w.r.t.\ $D_h$, so by the Arz\'ela-Ascoli theorem for every sequence of $n$'s tending to $\infty$, there is a subsequence along which $P_n $ converges uniformly.
The uniform limit is necessarily a $D_h$-geodesic from 0 to $z$, so must be equal to $P$. 
Hence $P_n\rta P$ uniformly. 

Let $P_n^L$ (resp.\ $P_n^R$) be the leftmost (resp.\ rightmost) $D_h$-geodesic from 0 to $P_n(\tau_2)$. 
Since $P_n(\tau_2) \rta P(\tau_2)$ and the $D_h$-geodesic $P|_{[0,\tau_2]}$ from 0 to $P(\tau_2)$ is unique, the Arz\'ela-Ascoli theorem applied as above shows that $P_n^L$ and $P_n^R$ each converge uniformly to $ P|_{[0,\tau_2]}$. 
Each of the points $P_n^L(\tau_1) , P_n^R(\tau_1), P(\tau_1)$ belongs to the finite set $\mcl X $, so by the above uniform convergence it follows that for large enough $n\in\BB N$ we have $P_n^L(\tau_1) = P_n^R(\tau_1) = P(\tau_1)$. 
Since $P_n|_{[0,\tau_2]}$ lies between $P_n^L$ and $P_n^R$, this implies~\eqref{eqn-stable-claim}.

We now deduce from~\eqref{eqn-stable-claim} that there is an open set $U' \subset U$ containing $z$ such that each $D_h$-geodesic from 0 to a point of $U'$ passes through $P(\tau_1)$. 
Indeed, if there were no such $U'$ then we could find a sequence of $D_h$-geodesics from 0 to points arbitrarily close to $z$ which do not pass through $P(\tau_1)$, which would contradict~\eqref{eqn-stable-claim}. 

We can assume without loss of generality that the open set $U'$ in the preceding paragraph has $D_h$-diameter at most $  \max\{\ep_1 , D_h(z,\bdy U) - 3 \ep_1 \}  $. 
Then by the definition of $\tau_1$ and the triangle inequality, for each $z' \in U'$, 
\alb
 D_h(z',\bdy U) 
&\geq D_h(z,\bdy U) - D_h(z,z') 
\geq 3 \ep_1 
\quad \text{and} \notag \\
D_h\left( z' , \bdy \mcl B_{\tau_1}^{z,\bullet}  \right)
&\leq D_h\left( z  , \bdy \mcl B_{\tau_1}^{z,\bullet}  \right) +  D_h(z,z')
\leq 2\ep_1 .
\ale 
It follows that no $D_h$-geodesic from 0 to $z'$ can exit $U$ after time $\tau_1$. 
By the definition of $U'$, each such $D_h$-geodesic passes through $P(\tau_1)$.
By the uniqueness of $P$, it follows that each such $D_h$-geodesic in fact coincides with $P$ on $[0,\tau_1]$.
Therefore, each such $D_h$-geodesic coincides with $P$ outside of $U$, as required.
\end{proof}

The following lemma will allow us to apply Proposition~\ref{prop-conf-finite} to sub-segments of general geodesics started from 0.

\begin{lem} \label{lem-regular-geo}
Almost surely, the following is true for each $z\in\BB C$ and each $D_h$-geodesic $P : [0,D_h(0,z)] \rta \BB C$ from 0 to $z$. 
For each $0 < r < D_h(0,z)$, the segment $P|_{[0,r]}$ is the only $D_h$-geodesic from 0 to $P(r)$.
\end{lem}
\begin{proof}
Let $q \in \BB Q^2$ and radii $0 < t < s$ with $t,s\in \BB Q$. We claim that on the event $\{s < D_h(0,q)\}$, a.s.\ the following is true. 
For each $z\in \BB C\setminus \mcl B_s^{q,\bullet}$ and each $D_h$-geodesic $P$ from 0 to $z$, $P|_{[0,t]}$ is the only $D_h$-geodesic from 0 to $P(t)$.
To see this, let $\mcl X = \mcl X_{t,s}^q \subset\bdy\mcl B_t^\bullet$ be the set of confluence points as in Proposition~\ref{prop-conf-finite}. 
By assertion~\ref{item-conf-finite-unique} of Proposition~\ref{prop-conf-finite}, for each $x\in \mcl X$ there is a unique $D_h$-geodesic from 0 to $x$. 
By Proposition~\ref{prop-conf-endpt-finite}, a.s.\ each $D_h$-geodesic $P$ from 0 to a point of $ \BB C\setminus \mcl B_s^{q,\bullet}$ passes through some $x\in \mcl X$.
Hence $P|_{[0,t]}$ must be the unique $D_h$-geodesic from 0 to $P(t) =x$, as required.  

Let us now deduce the lemma statement from the above claim. 
The event described in the above claim holds a.s.\ for every $q\in\BB Q^2$ and every $0 < t < s < D_h(0,q)$ with $t,s\in\BB Q$. 
We work on the (full probability) event that this is the case.
For $z\in\BB C$ and $0 < r < D_h(0,z)$, we choose $t,s \in \BB Q $ such that $r < t < s < D_h(0,z)$. 
The point $z$ lies in a connected component of $\BB C\setminus \mcl B_s$, which is necessarily equal to $\BB C\setminus \mcl B_s^{q,\bullet}$ for some $q\in\BB Q^2$. 
By the above claim applied to this choice of $q,t,s$, we get that $P|_{[0,t]}$ is the only $D_h$-geodesic from 0 to $P(t)$. 
Since $r < t$, this implies that $P|_{[0,r]}$ is the only $D_h$-geodesic from 0 to $P(r)$ (otherwise, we could obtain two distinct geodesics from 0 to $P(t)$ by concatenating geodesics from 0 to $P(r)$ with $P|_{[r,t]}$).
\end{proof}

The rest of the proof of Theorem~\ref{thm-general-confluence} is very similar to the proof of~\cite[Proposition 12]{akm-geodesics}. 
We will work with the following setup. 
Fix $z\in\BB C \setminus \{0\}$ and a $D_h$-geodesic $P$ from 0 to $z$.  
All of the almost sure statements below are required to hold for every choice of $z$ and $P$ simultaneously.
To simplify the geometry of the problem, we also fix a homeomorphism $\psi : \BB C\rta \BB C$ which takes the geodesic $P$ to the line segment $[0,1]$. 
The existence of such a homeomorphism follows from general topological theorems, as explained in~\cite[Section 3]{akm-geodesics}. 
The following lemma asserts the existence of an ``auxiliary" geodesic which will allow us to force $D_h$-geodesics near $P$ to merge into $P$. 
See Figure~\ref{fig-H-shape}, left, for an illustration of the statement.

\begin{figure}[t!]
 \begin{center}
\includegraphics[scale=1]{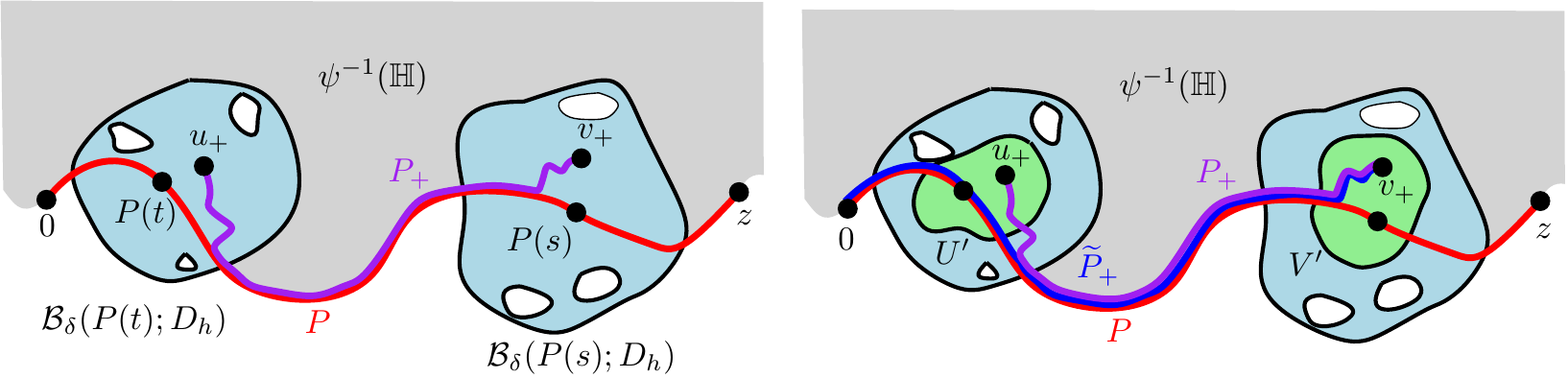}
\vspace{-0.01\textheight}
\caption{ \textbf{Left.} Illustration of the statement of Lemma~\ref{lem-H-shape}. The lemma asserts the existence of the purple $D_h$-geodesic in the figure.
\textbf{Right.} Illustration of the proof of Lemma~\ref{lem-H-shape}. We first apply Lemma~\ref{lem-stable} with $P(s)$ in place of $z$ and $ \mcl B_\delta(P(s) ; D_h)$ in place of $U$ to produce the open set $V'$. 
We then apply Lemma~\ref{lem-stable} a second time with $v_+ \in V'$ in place of zero, $P(t)$ in place of $z$, and $ \mcl B_\delta(P(t) ; D_h)$ in place of $U$ to produce the open set $U'$. 
Note that the red, blue, and purple $D_h$-geodesics all coincide outside of the union of the light blue $D_h$-metric balls. 
}\label{fig-H-shape}
\end{center}
\vspace{-1em}
\end{figure}

\begin{lem} \label{lem-H-shape}
Let $0 < t < s < D_h(u,v)$. 
Almost surely, for each $\delta >0$ there are points $u_+   \in \BB Q^2 \cap \mcl B_\delta(P(t);D_h)$
and $v_+   \in \BB Q^2 \cap \mcl B_\delta(P(s);D_h)$
such that if $P_+$ is the (a.s.\ unique) $D_h$-geodesic from $u_+$ to $v_+$, then 
\eqb \label{eqn-H-shape}
P_+ \setminus P \subset   \left[ \mcl B_\delta(P(t);D_h) \cup \mcl B_\delta(P(s);D_h)  \right] \cap  \psi^{-1}(\BB H)   . 
\eqe
\end{lem}
\begin{proof} 
We can assume without loss of generality that $\delta < (s-t)/100$, so that $P(s) \notin \mcl B_\delta(P(t);D_h)$ and the same is true with $s$ and $t$ interchanged.

By Lemma~\ref{lem-regular-geo}, $P|_{[0,s]}$ is a.s.\ the unique $D_h$-geodesic from 0 to $P(s)$.
Hence Lemma~\ref{lem-stable} (applied with $P(s)$ in place of $z$) implies that there is an open set $V' \subset \mcl B_\delta(P(s) ; D_h)$ containing $P(s)$ such that each $D_h$-geodesic from 0 to a point of $V'$ coincides with $P|_{[0,s]}$ outside of $\mcl B_\delta(P(s) ; D_h)$.
Let $v_+ \in V' \cap \BB Q^2 \cap \psi^{-1}(\BB H)$ and let $\wt P_+$ be the $D_h$-geodesic from 0 to $v_+$ (which is a.s.\ unique, since $v_+ \in \BB Q^2$). 
Note that the symmetric difference of $\wt P_+$ and $P([0,s])$ is contained in $\mcl B_\delta(P(s) ; D_h)$. 

Since $P(t)\notin \mcl B_\delta(P(s) ;D_h)$, we have $P(t) = \wt P_+(t)$. 
Almost surely, the $D_h$-geodesic from $v_+$ to $P(t)$ is unique (it coincides with a segment of the time reversal of $\wt P_+$) since otherwise there would be more than one $D_h$-geodesic from 0 to $v_+$. 
Since $v_+\in\BB Q^2$, we can apply Lemma~\ref{lem-stable} with $v_+$ in place of 0 and $P(t)$ in place of $z$ to find that a.s.\ there is an open set $U'\subset \mcl B_\delta(P(t) ; D_h) $ containing $P(t)$ such that each $D_h$-geodesic from $v_+$ to a point of $U'$ coincides with a segment of $\wt P_+$ outside of $\mcl B_\delta(P(t) ; D_h)$. 
We now choose $u_+ \in U' \cap \BB Q^2 \cap \psi^{-1}(\BB H)$ and let $P_+$ be the (a.s.\ unique, since $u_+,v_+\in\BB Q^2$) $D_h$-geodesic from $u_+$ to $v_+$.

By our choice of $U'$, the geodesic $P_+$ coincides with a segment of $P_+$ outside of $\mcl B_\delta(P(t) ; D_h)$, which in turn coincides with a segment of $P|_{[0,s]}$ outside of $\mcl B_\delta(P(s) ; D_h)$. Therefore, 
\eqb \label{eqn-H-shape0}
P_+ \setminus P \subset \mcl B_\delta(P(t);D_h) \cup \mcl B_\delta(P(s);D_h)   . 
\eqe 

It remains to show that $P_+ \setminus P \subset \psi^{-1}(\BB H)$. 
By possibly shrinking $\delta$, we can assume without loss of generality that $0,z\notin \mcl B_\delta(P(t) ; D_h) \cup \mcl B_\delta(P(s) ; D_h)$ and $  \mcl B_\delta(P(t) ; D_h) \cap \mcl B_\delta(P(s) ; D_h) = \emptyset$.
Since $\mcl B_\delta(P(t) ; D_h) $ and $ \mcl B_\delta(P(s) ; D_h)$ are connected, this implies that no path in $\psi(\mcl B_\delta(P(t) ; D_h))$ or $\psi(\mcl B_\delta(P(s) ; D_h))$ can cross $\psi^{-1}(\BB R)$ without first crossing $P$. 
The set of times $t$ for which $P_+(t) \in P$ must be an interval, for otherwise by replacing a segment of $P_+$ by a segment of $P$ we would violate the uniqueness of $P$. 
If follows that the segments of $P_+$ contained in $\mcl B_\delta(P(t);D_h) $ and $ \mcl B_\delta(P(s);D_h) $ cannot cross $P$.
From~\eqref{eqn-H-shape0} and since $\psi(u_+ ) , \psi(v_+) \in \BB H$, it therefore follows that $\psi(P_+ \setminus P) \subset \BB H$, as required.
\end{proof}

\begin{figure}[t!]
 \begin{center}
\includegraphics[scale=1]{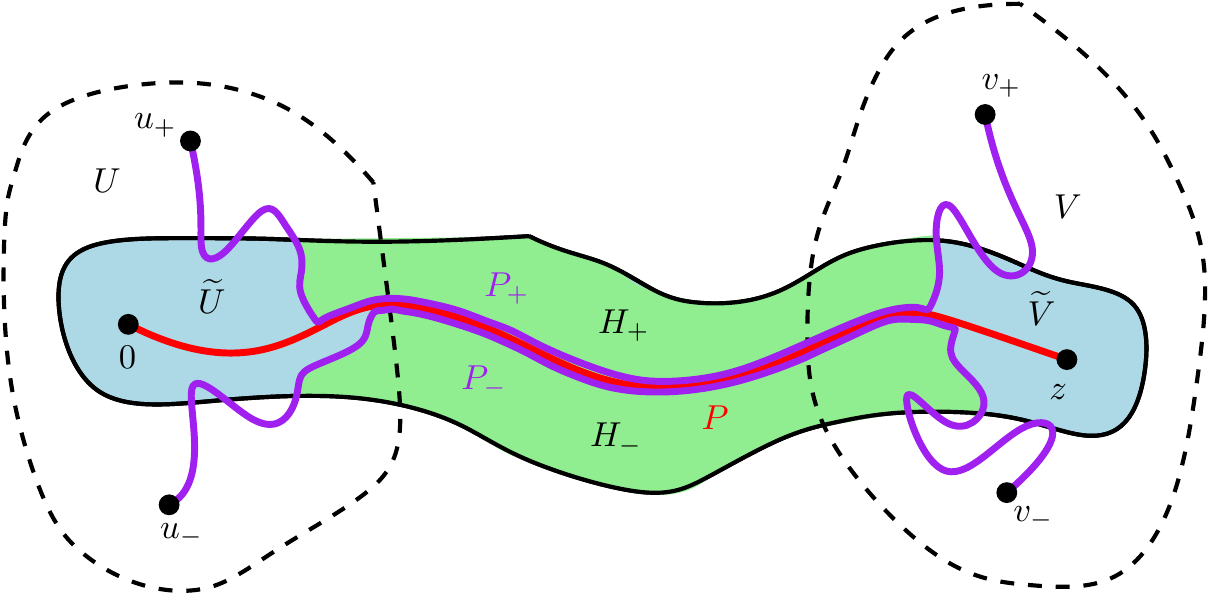}
\vspace{-0.01\textheight}
\caption{Illustration of the proof of Lemma~\ref{lem-H-cross}. The purple $D_h$-geodesics $P_+$ and $P_-$ come from Lemma~\ref{lem-H-shape}. They divide the region $\psi^{-1}(B_\ep([0,1]))$ (light green and light blue, not labelled) into four sub-regions $\wt U , \wt V , H_+ , H_-$. For large enough $n$, the $D_h$-geodesic $P_n$ must be contained in $\psi^{-1}(B_\ep([0,1]))$ and its initial (resp.\ terminal) point must be contained in $\wt U$ (resp.\ $\wt V$). Furthermore, due to the uniqueness of the $D_h$-geodesics between $u_+$ and $v_+$ and between $u_-$ and $v_-$ this $D_h$-geodesic cannot enter $H_+ \cup H_-$. This forces $P_n$ to merge into $P$. 
}\label{fig-H-cross}
\end{center}
\vspace{-1em}
\end{figure}

\begin{lem} \label{lem-H-cross}
Almost surely, the following is true simultaneously for each $z\in\BB C$, each neighborhood $U$ of 0, and each neighborhood $V$ of $z$.
Let $P$ be a $D_h$-geodesic from 0 to $z$ and let $\{P_n\}_{n\in\BB N}$ be a sequence of $D_h$-geodesics which converge uniformly to $P$.
Then
\eqb \label{eqn-H-cross}
P_n \setminus (U \cup V) \subset P ,\quad \text{for each large enough $n\in\BB N$}. 
\eqe
\end{lem}
 \begin{proof}
Choose $0 < s < t < D_h(z,w)$ and $\delta >0$ so that $ \mcl B_\delta(P(t);D_h) \subset U$, $ \mcl B_\delta(P(s);D_h) \subset V$, and $  \mcl B_\delta(P(t) ; D_h)$ and $ \mcl B_\delta(P(s) ; D_h)$ lie at positive distance from each other and from 0 and $z$.  
Define $u_+ , v_+$, and $P_+$ as in Lemma~\ref{lem-H-shape} and let $u_-,v_-$, and $P_-$ be as in Lemma~\ref{lem-H-shape} with $-\BB H$ in place of $\BB H$, so that $P_-$ is the unique $D_h$-geodesic from $u_-$ to $v_-$ and $P_-\setminus  P  \subset \left[\mcl B_\delta(P(t);D_h) \cup \mcl B_\delta(P(s);D_h) \right] \cap \psi^{-1}(- \BB H)$.

Recall that $\psi(P) = [0,1]$, so we can find $\ep > 0$ small enough that $u_+,v_+,u_-,v_- \notin \psi^{-1}(B_\ep([0,1]))$. 
Let $P_\pm'$ be the segment of $P_\pm$ between the last time that it enters $\psi^{-1}(B_\ep([0,1]))$ before hitting $P$ and the first time that it exits $\psi^{-1}(B_\ep([0,1]))$ after  hitting $P$. 
The set $\psi^{-1}(B_\ep([0,1]))$ is homeomorphic to the open disk and $P_\pm'$ is a simple curve joining two of its boundary points, so $P_\pm'$ divides $\psi^{-1}(B_\ep([0,1]))$ into two connected components.
Let $H_\pm$ be the one of these two connected components which is contained in $\psi^{-1}(\pm \BB H)$. 

Since $  \mcl B_\delta(P(t) ; D_h)$ and $ \mcl B_\delta(P(s) ; D_h)$ lie at positive distance from each other, it follows from the defining property of $P_\pm$ that  $\bdy H_+\cap \bdy H_-$ is a non-trivial segment of $P$. 
Furthermore, since $  \mcl B_\delta(P(t) ; D_h)$ and $ \mcl B_\delta(P(s) ; D_h)$ lie at positive distance from 0 and $z$, it is easily seen that $0$ and $z$ each lie at positive distance from $P_+ \cup P_-$ and hence also from $\ol H_+ \cup \ol H_-$. 
From this, we see that $\psi^{-1}(B_\ep([0,1])) \setminus (\ol H_+ \cup \ol H_-)$ is the union of two open sets $\wt U , \wt V$ which lie at positive distance from each other such that $0\in \wt U$ and $z \in \wt V$. 
Since $P_+\setminus P$ and $P_-\setminus P$ are each contained in $U\cup V$, we have $\wt U\subset U$ and $\wt V\subset V$. 
 
Now consider a sequence of $D_h$-geodesics $P_n \rta P$ as in the lemma statement. 
For each large enough $n\in\BB N$, $P_n$ is entirely contained in $\psi^{-1}(B_\ep([0,1]))$ and the starting and ending points of $P_n$ are contained in $\wt U$ and $\wt V$, respectively. 
Henceforth assume that this is the case. 
Then $P_n$ must pass through $\ol H_+ \cup \ol H_-$.

For any two times $\tau < \sigma$ for which $P_n(\tau) , P_n(\sigma) \in P_+$, it must be the case that $P_n$ traces $P_+$ during the time interval $[\tau,\sigma]$: indeed, otherwise there would be two distinct $D_h$-geodesics between the endpoints of $P_+$.
If $P_n$ enters $H_+$, then it must subsequently exit $H_+$ and hence it must hit $\bdy H_+ \cap \psi^{-1}(B_\ep([0,1])) \subset P_+$ twice. 
Hence $P_n$ cannot enter $H_+$.
Similarly, $P_n$ cannot enter $H_-$. 
Since $P_n\subset  \psi^{-1}(B_\ep([0,1]))$, we therefore get $P_n \subset \wt U \cup \wt V  \cup P_+ \cup P_- $ which implies that $P_n\setminus (U\cup V) \subset (P_+\cup P_-)\setminus (U\cup V) \subset P$. 
Thus~\eqref{eqn-H-cross} holds.
\end{proof}

\begin{lem} \label{lem-nbd-confluence}
Almost surely, for each $z\in\BB C$, each neighborhood $U$ of 0, and each neighborhood $V$ of $z$, there are open sets $U',V'$ with $0\in U'\subset U$ and $z\in V' \subset V$ such that every $D_h$-geodesic from a point of $U'$ to a point of $V'$ coincides with a $D_h$-geodesic from 0 to $z$ outside of $U\cup V$.
\end{lem}
\begin{proof} 
Assume by way of contradiction that no sets $U',V'$ in the lemma statement exist.
Then there is a sequence of points $w_n \rta 0$, a sequence of points $z_n \rta z$, and a sequence of $D_h$-geodesics $P_n$ from $0$ to $z$ such that $P_n \setminus (U\cup V)$ is not contained in $P$ for any $n\in\BB N$.
By the Arz\'ela-Ascoli theorem, after possibly passing to a subsequence, we can arrange that $P_n$ converges uniformly to a path $P$ from 0 to $z$ which is necessarily a $D_h$-geodesic from 0 to $z$. 
We then obtain a contradiction from Lemma~\ref{lem-H-cross}.
\end{proof}

\begin{proof}[Proof of Theorem~\ref{thm-general-confluence}]
By possibly shrinking $U$, we can assume without loss of generality that $U$ is bounded. 
By~\cite[Theorem 1.3]{gm-confluence}, a.s.\ there exists $z_0 \in U\setminus \{0\}$ such that every $D_h$-geodesic from 0 to a point of $\BB C\setminus U$ passes through $z_0$. 
By Lemma~\ref{lem-nbd-confluence} (applied with each of $U$ and $V $ replaced by $\BB C\setminus \{z_0\}$), for each $z\in \bdy U$ we can choose a neighborhood $U'_z $ of 0 and a neighborhood $V_z'  $ of $z$ such that $z_0 \notin U_z'\cup V_z'$ and each $D_h$-geodesic from a point of $U_z'$ to a point of $ V_z'$ coincides with a $D_h$-geodesic from 0 to $z$ outside of $U_z'\cup V_z'$.
In particular, each such $D_h$-geodesic must pass through $z_0$. 

By the compactness of $\bdy U$, there is a finite set $Z$ of points in $\bdy U$ such that $\bdy U\subset \bigcup_{z\in Z} V_z'$. 
Hence, if we set $U' = \bigcap_{z\in Z} U_z'$, then every $D_h$-geodesic from a point in $U'$ to a point in $\bdy U$ must pass through $z_0$. 
Every $D_h$-geodesic from a point in $U'$ to a point in $\BB C\setminus U$ has a sub-segment which is a $D_h$-geodesic from $0$ to a point of $\bdy U'$.
Hence, the theorem statement holds for this choice of $U'$.
\end{proof}

\section{Zero-one laws for geodesics and ball boundaries: two essential ingredients}
\label{sec-two-ingredients}

As we described in Section~\ref{sec-intro-zero-one}, the two ingredients we need to prove zero-one laws for random fractals are (1) a version of scale invariance, and (2) a locality property.  We demonstrate the basic idea of our proofs in Section~\ref{sec-net-dim} by proving Theorem~\ref{thm-net-dim}, the zero-one law for the metric net.

In the context of the LQG metric space, we are most interested in LQG geodesics and LQG ball boundaries; but these fractals are neither scale-invariant nor local.  We address these two issues in Sections~\ref{sec-infty} and~\ref{sec-geo-ball-local} in the manner discussed just after the statement of Theorem~\ref{thm-net-dim}.  We then prove our results for LQG geodesics (Theorem~\ref{thm-geo-dim}) and LQG ball boundaries (Theorem~\ref{thm-ball-bdy-bound}) in Sections~\ref{sec-geo-dim} and~\ref{sec-ball-bdry}, respectively.

\subsection{Zero-one law for the metric net}
\label{sec-net-dim}

As we mentioned in Section~\ref{sec-intro-zero-one}, the metric net satisfies both a scale invariance property and a locality property.

\begin{lem}[Scale invariance of the metric net] \label{lem-net-scale}
For each $r > 0$, a.s.\ 
\eqb \label{eqn-net-scale}
r \mcl N_{r^{- \xi Q} e^{-\xi h_r(0)} s}(0 ; D_{h(r\cdot) -h_r(0)})  = \mcl N_{   s}   ,\quad\forall s > 0 .
\eqe
In particular, $(r  \mcl N_s)_{s\geq 0} $ has the same law as $( \mcl N_s )_{s \geq 0}$ modulo a linear change of the time parametrization. 
\end{lem} 
\begin{proof}
By Axioms~\ref{item-metric-f} and~\ref{item-metric-coord} (Weyl scaling and coordinate change), a.s.\ 
\eqb
D_{h(r\cdot)  - h_r(0)}(u,v) = r^{- \xi Q} e^{-\xi h_r(0)} D_{h(r\cdot) + Q\log r}(u,v) = r^{-\xi Q} e^{-\xi h_r(0)} D_h(r u , r v) ,\quad\forall u ,v\in\BB C .
\eqe
From this the relation~\eqref{eqn-net-scale} is immediate. 
Since $h(r\cdot) -h_r(0) \eqD h$, the last statement follows from~\eqref{eqn-net-scale}. 
\end{proof}

\begin{lem}[Locality of the metric net] \label{lem-net-local}
For $r>0$, let 
\eqb \label{eqn-net-tau}
\tau_r := \inf\{s > 0 : \mcl B_s \not\subset B_r(0)\} = \inf\{s > 0 : \mcl N_s\not\subset B_r(0) \} .
\eqe
For each $r > 0$, $\mcl N_{\tau_r}$ is a.s.\ determined by $h|_{B_{\tau_r}(0)}$
\end{lem} 
\begin{proof}
This is immediate from the locality property of the LQG metric (Axiom~\ref{item-metric-local}). 
\end{proof}

Theorem~\ref{thm-net-dim} follows easily from these two properties of the metric net.  Roughly speaking, we derive the zero-one law as follows.
\begin{itemize}
\item The scale invariance property of $\mcl N_s$ (Lemma~\ref{lem-net-scale}) allows us to get lower bounds for the probability that the Hausdorff dimension is bounded below which hold uniformly over all Euclidean scales.
\item The locality property (Lemma~\ref{lem-net-local}) allows us to use the tail triviality of the $\sigma$-algebras $\sigma(h|_{B_r(0)})$ as $r\rta0$ to upgrade from positive probability to probability one.
\end{itemize}

\begin{proof}[Proof of Theorem~\ref{thm-net-dim}]
We prove the result for Euclidean dimension; the proof for $\gamma$-quantum dimension is identical. 
Suppose $c>0$ such that $\BB P[\dim_{\mcl H}^0 \mcl N_\infty \geq c] > 0$.  
For $r> 0$, let $\tau_r$ be as in~\eqref{eqn-net-tau}. We will show that
\eqb \label{eqn-net-dim-as}
\BB P\left[ \dim_{\mcl H}^0 \mcl N_{\tau_r} \geq c ,\: \forall  r > 0 \right] = 1 .
\eqe
Since $\mcl N_\infty = \bigcup_{r > 0} \mcl N_{\tau_r}$, by the countable stability of Hausdorff dimension~\eqref{eqn-net-dim-as} implies that a.s.\ $\dim_{\mcl H}^0 \mcl N_\infty \geq c$, so $\dim_{\mcl H}^0 \mcl N_\infty $ is a.s.\ equal to a deterministic constant.
For every $s>0$, there exists $r>0$ such that $\mcl N_{\tau_r} \subset \mcl N_s \subset \mcl N_\infty$, so by~\eqref{eqn-net-dim-as}, a.s.\ $\mcl N_s$ is equal to this same deterministic constant simultaneously for every $s>0$. 

It remains to prove~\eqref{eqn-net-dim-as}. By our choice of $c$ and the countable stability of Hausdorff dimension, for each $\delta>0$ there exists $r_0 > 0$ and $p > 0$ such that $\BB P[\dim_{\mcl H}^0 \mcl N_{\tau_{r_0}} \geq c - \delta ] \geq p$. 
By Lemma~\ref{lem-net-scale}, the law of $r^{-1} \mcl N_{\tau_r}$ does not depend on $r$.  
Therefore,
\eqb \label{eqn-net-dim-pos}
\BB P\left[ \dim_{\mcl H}^0 \mcl N_{\tau_r} \geq c -\delta \right] \geq p,\quad\forall r > 0. 
\eqe
This means that, if we let $\mcl T$ denote the event that $\dim_{\mcl H}^0 \mcl N_{\tau_r} \geq c - \delta$ for arbitrarily small values of $r>0$, then the event $\mcl T$ has probability at least $p$.  By Lemma~\ref{lem-net-local}, the event $\mcl T$ is measurable w.r.t. the tail $\sigma$-algebra $\bigcap_{r > 0} \sigma\left( h|_{B_{\tau_r}(0)} \right)$.  Since this $\sigma$-algebra is trivial (see, e.g.,~\cite[Lemma 2.2]{hs-euclidean}), the event $\mcl T$ has probability one.  Since $\mcl N_{\tau_r}$ is increasing in $r$ and $\delta>0$ can be made arbitrarily small, this implies~\eqref{eqn-net-dim-as}.
\end{proof}

\begin{remark}
Our proof of Theorem~\ref{thm-net-dim} shows that for every open set $U\subset\BB C$ containing 0, a.s.\ $\dim_{\mcl H}^0 (\mcl N_s \cap U)  = \Delta_{\op{net}}^0$ for every $s > 0$. We do not rule out the possibility that the dimension of $\mcl N_s$ is ``concentrated near 0", i.e., there could be open sets $U\subset\BB C$ which do not contain 0 such that $\mcl N_s \cap U \not=\emptyset$ but $\dim_{\mcl H}^0(\mcl N_s \cap U) < \Delta_{\op{net}}^0$ (we expect, but do not prove, that no such open sets exist). 
The same is true for the quantum dimension. 
Similar considerations apply for the other zero-one laws proven in this paper: our proof of Theorem~\ref{thm-geo-dim} does not rule out the possibility that the Euclidean dimension of an LQG geodesic is ``concentrated'' at the starting point of the geodesic. Likewise, Theorem~\ref{thm-ball-bdy-bound} does not rule out the possibility that the $\gamma$-quantum and Euclidean dimensions of $\bdy\mcl B_{D_h(0,z)}(0;D_h)$ are ``concentrated'' at the point $z$.
\end{remark}

\subsection{Scale invariance: defining geodesic rays and metric balls centered at $\infty$}
\label{sec-infty}

As we described in the introduction, the dimensions of LQG geodesics and metric ball boundaries are more difficult to study, because they are neither scale invariant nor locally determined by the field.  In the rest of this section, we describe how we tackle these two challenges and obtain versions of scale invariance and locality for these fractal that we can use to derive zero-one laws.
To get scale invariance, we define ``infinite-volume" versions of LQG geodesics and metric ball boundaries whose laws are exactly scale invariant. 
First, for LQG geodesics, we define an infinite geodesic ray from 0 to $\infty$.

\begin{prop} \label{prop-infty-geo}
Almost surely, for each $z\in \BB C$ there exists a (not necessarily unique) infinite geodesic ray $P_z^\infty$ started from $z$, called a \emph{$D_h$-geodesic from $z$ to $\infty$}. These infinite geodesic rays satisfy the following properties. 
\begin{enumerate}[(i)]
\item For each fixed $z\in\BB C$, a.s.\ the geodesic ray $P_z^\infty$ is unique. \label{item-infty-geo-unique}
\item Almost surely, for each $r >0$ there exists $R > r$ such that for each $z\in B_r(0)$, the symmetric difference of $P_0^\infty$ and $P_z^\infty$ is contained in $B_R(0)$. \label{item-infty-geo-conf}
\item For each $r > 0$,  \label{item-infty-geo-scale}
\eqb \label{eqn-infty-geo-scale}
\left( h(r\cdot) - h_r(0) , r^{-1} P_0^\infty(r^{\xi Q} e^{\xi h_r(0)} \cdot)  \right) \eqD \left(h , P_0^\infty(\cdot) \right)
\eqe 
\end{enumerate}
\end{prop}

Proposition~\ref{prop-infty-geo} gives the existence of many one-sided infinite geodesics for $D_h$. It is easy to see from confluence that there are no \emph{two}-sided infinite geodesics for $D_h$.

\begin{lem} \label{lem-no-bi-infinite}
Almost surely, there are no bi-infinite $D_h$-geodesics, i.e., there are no paths $P : \BB R\rta \BB C$ such that $P|_{[s,t]}$ is a $D_h$-geodesic for each $s<t$. 
\end{lem}
\begin{proof}
Assume by way of contradiction that there is a bi-infinite $D_h$-geodesic with positive probability. 
Then there exists $q\in (0,1)$ and $R>0$ such that with positive probability, there is a bi-infinite $D_h$-geodesic $P$ which passes through $B_R(0)$. 
By Weyl scaling and the LQG coordinate change formula (Axioms~\ref{item-metric-f} and~\ref{item-metric-coord}), if $r>0$ then on the event that $P$ exists, a.s.\ the path $t\mapsto r^{-1} P(e^{\xi h_r(0)} t)$ is a bi-infinite $D_h$-geodesic for the field $h(r\cdot) - h_r(0)$. 
Since $h(r\cdot) -h_r(0) \eqD h$, we get that with probability at least $q$ there is a bi-infinite $D_h$-geodesic for $h$ which passes through $B_{R r}(0)$. 
Since this holds for every $r>0$, we get that with probability at least $q$ there is a (random) sequence $r^n\rta 0$ such that for each $n$, there is a bi-infinite $D_h$-geodesic $P^n$ which passes through $B_{R r^n}(0)$. 
If we parametrize $P^n$ so that $P^n(0) \in B_{R r^n}(0)$, then by the Arz\'ela-Ascoli theorem, the paths $P^n$ admit a subsequential limit with respect to the local uniform topology for paths $\BB R\rta \BB C$. The limiting path is a bi-infinite $D_h$-geodesic which passes through 0. 
But, a.s.\ there is no bi-infinite $D_h$-geodesic which passes through 0 due to confluence of geodesics~\cite[Theorem 1.3]{gm-confluence}.
\end{proof} 

For the proof of Proposition~\ref{prop-infty-geo} we need the following variant of Theorem~\ref{thm-general-confluence}. 

\begin{lem} \label{lem-general-confluence-scale}
For each $p\in (0,1)$, there exists $A = A(p,\gamma) > 1$ such that for each fixed $r > 0$, it holds with probability at least $p$ that the following is true.
There is a point $z_{0} \in B_{A r}(0) \setminus B_r(0)$ such that every $D_h$-geodesic from a point of $B_r(0)$ to a point of $\BB C\setminus B_{A r}(0)$ passes through $z_{0}$. 
\end{lem}
\begin{proof}
By Theorem~\ref{thm-general-confluence} applied with $U = \BB D$, there exists $A = A(p,\gamma) >1$ such that the statement of the lemma holds with this choice of $A$ and with $r = 1/A$. 
By Weyl scaling (Axiom~\ref{item-metric-f}), for each fixed $r>1$, the occurrence event in the lemma statement does not depend on the choice of additive constant for $h$. 
By the LQG coordinate change formula (Axiom~\ref{item-metric-coord}) and the scale invariance of the law of $h$, modulo additive constant, we see that for a fixed choice of $A$, the probability of the event in the lemma statement does not depend on $r$. 
\end{proof}

\begin{proof}[Proof of Proposition~\ref{prop-infty-geo}]
By Lemma~\ref{lem-general-confluence-scale}, for each $p\in (0,1)$ there exists $A = A(p,\gamma) > 1$ such that for each $k\in\BB N$, we have $\BB P[E_k] \geq p$, where $E_k$ is the event that the following is true. 
There is a point $z_k \in B_{A^k}(0) \setminus B_{A^{k-1}}(0)$ such that every $D_h$-geodesic from a point of $B_{A^{k-1}}(0)$ to a point of $\BB C\setminus B_{A^k}(0)$ passes through $z_k$. 
With probability at least $p$, the event $E_k$ occurs for infinitely many $k\in\BB N$.

Since $p$ can be made arbitrarily close to 1, we get that a.s.\ the following is true.
There is a sequence of positive radii $R_n \rta \infty$ and points $z_n \in B_{R_n}(0) \setminus B_{R_{n-1}}(0)$ such that every $D_h$-geodesic from a point of $B_{R_{n-1}}(0)$ to a point of $\BB C\setminus B_{R_n}(0)$ passes through $z_n$.
We set $z_0 := 0$.

We claim that for each $n\in\BB N$, there is a unique $D_h$-geodesic from $z_{n-1}$ to $z_n$. 
To see this, let $q\in \BB Q^2\setminus B_{R_n}(0)$. Then a.s.\ the $D_h$-geodesic $P_q$ from 0 to $q$ is unique and this $D_h$-geodesic must pass through both $z_{n-1}$ and $z_n$. 
If there were more than one $D_h$-geodesic from $z_{n-1}$ to $z_n$, then we could replace the segment of $P_q$ between $z_{n-1}$ and $z_n$ by one of these geodesics to get a contradiction to the uniqueness of $P_q$. 

In particular, for each $n\in\BB N$ the $D_h$-geodesic from 0 to $z_n$ is unique. Moreover, the $D_h$-geodesic from 0 to $z_{n-1}$ is the sub-path of the $D_h$-geodesic from 0 to $z_n$ between 0 and $z_{n-1}$. 
By sending $n\rta\infty$, we get that a.s.\ there is a unique infinite geodesic ray $P_0^\infty$ from 0 to $\infty$.
By the translation invariance of the law of $h$, modulo additive constant, this shows that for each fixed $z\in\BB C$ there is a.s.\ a unique infinite geodesic ray from $z$ to $\infty$, i.e., assertion~\eqref{item-infty-geo-unique} holds.

To construct infinite geodesic rays for all possible starting points simultaneously, consider $z\in\BB C$ and let $n\in\BB N$ be chosen so that $z\in B_{R_{n-1}}(0)$. 
There is a $D_h$-geodesic $P$ from $z$ to $z_{n+1}$, which must pass through $z_n$. 
In particular, $P$ coincides with the unique $D_h$-geodesic from $z_n$ to $z_{n+1}$ between the times when it hits $z_n$ and $z_{n+1}$.
From this, it follows that the concatenation of $P$, stopped upon hitting $z_n$, with the $D_h$-geodesic ray from $z_n$ to $\infty$ is a $D_h$-geodesic ray from $z$ to $\infty$. 

We next prove assertion~\eqref{item-infty-geo-conf}. 
Given $r > 0$, choose $n\in\BB N$ such that $r \leq R_{n-1}$ and let $R := R_n$.
A $D_h$-geodesic from a point of $B_r(0)$ to $\infty$ stopped when it first hits $\BB C\setminus B_R(0)$ is a $D_h$-geodesic from a point of $B_r(0)$ to a point of $\BB C\setminus B_R(0)$. 
By the definition of the $R_n$'s, every $D_h$-geodesic from a point of $B_r(0)$ to $\infty$ must pass through $z_n$. 
Similarly, each such $D_h$-geodesic must hit $z_N$ for each $N \geq n$. 
By the uniqueness of the $D_h$-geodesic from $z_N$ to $z_{N+1}$, each such $D_h$-geodesic must coincide with $P_0^\infty$ after hitting $z_n$, so must coincide with $P_0^\infty$ after its first exit time from $B_R(0)$.

Finally, we prove assertion~\eqref{item-infty-geo-scale}. By Axioms~\ref{item-metric-f} and~\ref{item-metric-coord}, applied in the same manner as in Lemma~\ref{lem-net-scale}, we get that $r^{-1} P_0^\infty(r^{\xi Q} e^{\xi h_r(0)} \cdot)$ is an infinite geodesic ray started from 0 for $D_{h(r\cdot)  -h_r(0)}$. 
Since such an infinite geodesic ray is unique (assertion~\eqref{item-infty-geo-unique}) and $h(r\cdot) - h_r(0) \eqD h$, we obtain~\eqref{eqn-infty-geo-scale}. 
\end{proof}

\begin{figure}[t!]
 \begin{center}
\includegraphics[scale=1]{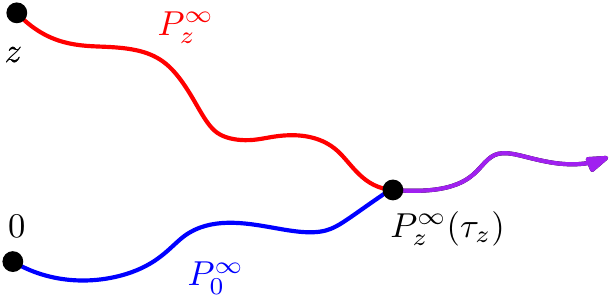}
\vspace{-0.01\textheight}
\caption{The infinite $D_h$-geodesic ray started from 0 (resp.\ $z$) is the union of the blue and purple (resp.\ red and purple) curves. 
The point $z$ belongs to $\mcl B_0^\infty$ if and only if the $D_h$-length of the red geodesic segment is smaller than or equal to the length of the blue geodesic segment. 
}\label{fig-geodesic-merge}
\end{center}
\vspace{-1em}
\end{figure} 

For LQG metric ball boundaries, we define a ``metric ball started from $\infty$ and grown until it hits 0" as the limit of $\mcl B_{D_h(0,w)}(w;D_h)$ as $w\rta\infty$.

\begin{prop} \label{prop-ball-infty}
There is a random unbounded set $\mcl B_0^\infty = ``\mcl B_{D_h(0,\infty)}(\infty ; D_h)" \subset \BB C$ such that the balls $\mcl B_{D_h(0,w)}(w;D_h)$ converge to $\mcl B_0^\infty$ as $w\rta\infty$ in the following sense.
Almost surely, for each $r> 0$, there exists $R > r$ such that 
\eqb \label{eqn-ball-infty-conv}
\mcl B_{D_h(0,w)}(w;D_h)\cap B_r(0) = \mcl B_0^\infty \cap B_r(0),\quad\forall w \in \BB C\setminus B_R(0) . 
\eqe 
Furthermore, $\mcl B_0^\infty$ is a.s.\ determined by $h$ viewed modulo additive constant and the law of $\mcl B_0^\infty$ is scale invariant in the sense that 
\eqb \label{eqn-ball-infty-scale}
\left( h , \mcl B_0^\infty \right) \eqD \left( h(r \cdot) - h_r(0) , r^{-1} \mcl B_0^\infty\right) ,\quad\forall r > 0 .
\eqe
\end{prop}

Due to the translation invariance of the law of $h$, modulo additive constant, Proposition~\ref{prop-ball-infty} allows us to define the ball $\mcl B_z^\infty = ``\mcl B_{D_h(z,\infty)}(\infty ; D_h)"$ for each $z\in\BB C$. In particular, $\mcl B_z^\infty$ is constructed from $h(\cdot+z)$ in the same manner that $\mcl B_0^\infty$ is constructed from $h$.

\begin{proof}[Proof of Proposition~\ref{prop-ball-infty}]
For $z\in\BB C$, let $P_z^\infty$ be a $D_h$-geodesic from $z$ to $\infty$, as in Proposition~\ref{prop-infty-geo} (chosen in an arbitrary manner if there is more than one such $D_h$-geodesic).
By Proposition~\ref{prop-infty-geo}, $P_0^\infty$ is a.s.\ unique and each $P_z^\infty$ merges into $P_0^\infty$ at some finite time.
Let $\tau_z := \inf\left\{ t\geq 0 : P_z^\infty(t) \in P_0^\infty\right\}$ and define
\eqb \label{eqn-ball-infty-def}
\mcl B_0^\infty := \left\{ z\in\BB C : \tau_z \leq D_h\left( 0 , P_z^\infty(\tau_z) \right) \right\}.
\eqe
In other words, $\mcl B_0^\infty$ is the set of $z\in\BB C$ for which the segment of $P_z^\infty$ before it first hits $P_0^\infty$ is shorter than the segment of $P_0^\infty$ before it reaches $P_z^\infty(\tau_z)$. See Figure~\ref{fig-geodesic-merge} for an illustration. 
We note that $P_z^\infty$ merges into $P_0^\infty$ at time $\tau_z$, i.e., $P_z^\infty([\tau_z,\infty)) \subset P_0^\infty$: if not, then we could find two distinct $D_h$-geodesic rays from 0 to $\infty$ by replacing a segment of $P_0^\infty$ by a segment of $P_z^\infty$, which would contradict the uniqueness of $P_0^\infty$. 

To make sure that $\mcl B_0^\infty$ is well-defined, we need to check that the definition of $\mcl B_0^\infty$ does not depend on the choice of $P_z^\infty$ in the case when the geodesic from $z$ to $\infty$ is not unique. 
Indeed, suppose $P_z^\infty$ and $\wt P_z^\infty$ are two $D_h$-geodesics from $z$ to $\infty$ and let $\wt\tau_z$ be defined as above with $\wt P_z^\infty$ in place of $P_z^\infty$.
By re-labeling, we can assume without loss of generality that $D_h\left( 0 , \wt P_z^\infty(\wt \tau_z) \right) \leq D_h\left( 0 , P_z^\infty(\tau_z) \right)$. 
Then $P_z^\infty|_{[0,\tau_z]}$ and the concatenation of $\wt P_z^\infty|_{[0,\wt\tau_z]}$ with $P_0^\infty|_{[D_h\left( 0 , \wt P_z^\infty(\wt \tau_z) \right) , D_h\left( 0 ,   P_z^\infty( \tau_z) \right)]}$ are each $D_h$-geodesics from 0 to $P_z^\infty(\tau_z)$, so their lengths agree, i.e., 
\eqbn
\tau_z = \wt\tau_z   +  D_h\left( 0 ,   P_z^\infty( \tau_z) \right)  -    D_h\left( 0 , \wt P_z^\infty(\wt \tau_z) \right) .
\eqen
This shows that the definition of~\eqref{eqn-ball-infty-def} is unaffected by replacing $P_z^\infty$ with $\wt P_z^\infty$. 

We now check the convergence property~\eqref{eqn-ball-infty-conv}.
By Lemma~\ref{lem-general-confluence-scale}, a.s.\ for each $r >0$ there exists $R > 0$ and a point $Z_r\in B_R(0)\setminus B_r(0)$ such that every $D_h$-geodesic from a point of $B_r(0)$ to a point of $\BB C \setminus B_R(0)$ passes through $Z_r$.
Hence
\eqb \label{eqn-ball-infty-out}
D_h\left( z , w \right) = D_h(z ,Z_r) + D_h(Z_r,w) ,\quad \forall z\in B_r(0),\quad \forall w\in\BB C\setminus B_R(0).
\eqe
Furthermore, for each $z\in B_r(0)$, the $D_h$-geodesic $P_0^\infty$ hits $P_z^\infty(\tau_z)$ before hitting $Z_r$ and hence
\eqb \label{eqn-ball-infty-merge}
  D_h\left( 0, Z_r \right)   = D_h\left( 0 , P_z^\infty(\tau_z) \right)   + D_h\left(  P_z^\infty(\tau_z) , Z_r \right) \quad \text{and} \quad
  D_h\left(z , Z_r \right)  = \tau_z +  D_h\left(  P_z^\infty(\tau_z) , Z_r \right) .
\eqe
By subtracting the two equations in~\eqref{eqn-ball-infty-merge} then applying~\eqref{eqn-ball-infty-out} once for $z$ and once with 0 in place of $z$, we obtain that for each $w\in \BB C\setminus B_R(0)$, 
\eqbn
\tau_z - D_h\left( 0 , P_z^\infty(\tau_z) \right) 
 =    D_h\left(z , Z_r \right)  - D_h\left( 0, Z_r \right)
 =  D_h(z ,w) - D_h(0,w) .
\eqen
Recalling~\eqref{eqn-ball-infty-def}, we now get~\eqref{eqn-ball-infty-conv}.

Due to Weyl scaling (Axiom~\ref{item-metric-f}), adding a constant to $h$ causes us to multiply both $\tau_z$ and $D_h\left( 0 , P_z^\infty(\tau_z) \right)$ in~\eqref{eqn-ball-infty-def} by the factor $e^{\xi C}$. It follows that adding a constant to $h$ does not affect the definition of $\mcl B_0^\infty$, so $\mcl B_0^\infty$ is a.s.\ determined by $h$ modulo additive constant. 

It remains to prove the scaling relation~\eqref{eqn-ball-infty-scale}. 
From~\eqref{eqn-ball-infty-def} and Weyl scaling (Axiom~\ref{item-metric-f}), it is immediate that the definition of $\mcl B_0^\infty$ is unaffected by adding a constant to $h$. 
From this, Axiom~\ref{item-metric-coord}, and the scale invariance of the law of $h$ modulo additive constant, we get~\eqref{eqn-ball-infty-scale}.
\end{proof}

\subsection{Locality: applying the strong confluence property}
\label{sec-geo-ball-local}

We now apply Theorem~\ref{thm-general-confluence} to construct an event of positive probability on which LQG geodesics and metric ball boundaries are in some sense locally determined by $h$.

\begin{lem} \label{lem-conf-event}
For each $p\in (0,1)$, there exists $ A  = A(p,\gamma)>1$ such that for each $z\in\BB C$ and $r>0$, there is a random point $Z_r(z) \in B_{A r}(z) \setminus B_r(z)$ and an event $E_r(z)$, both of which are a.s.\ given by measurable functions of $h|_{B_{A^2 r}(z) \setminus B_{r/A}(z)}$ viewed modulo additive constant, such that the following is true.  
We have $\BB P[E_r(z)] \geq p$. 
Moreover, the metric $D_h$ has the following properties on the event $E_r(z)$.
\begin{enumerate}[(i)]
\item
Each $D_h$-geodesic from a point of $B_r(z)$ to a point of $\BB C\setminus B_{A r}(z)$ passes through $Z_r(z)$.
\label{item-confluence}
\item 
Each $D_h$-geodesic between two points of $\ol{B_{A r}(z)}$ is contained in $B_{A^2 r}(z)$. 
\label{item-avoids0}
\item \label{item-avoids} Each $D_h$-geodesic between two points of $\BB C \setminus B_{r}(z)$ is contained in $\BB C \setminus \ol{B_{r/A}(z)}$. More strongly, there is a path $\pi \subset B_r(z) \setminus \ol{B_{r/A}(z)}$ such that 
\eqb \label{eqn-avoids-path}
\left(\text{$D_h$-length of $\pi$} \right) < D_h\left(\pi , B_{r/A}(z) \right) .
\eqe 
\end{enumerate}
\end{lem}
\begin{proof}
\noindent\textit{Step 1: defining the event.}
Let $A >1$ to be chosen later and for $z\in\BB C$, let $ E_r(z)$ be the event that the following is true.
\begin{enumerate}
\item There is a point $Z \in B_{A r}(z) \setminus B_r(z)$ such that every $D_h\left( \cdot,\cdot; B_{A^2 r}(z)\setminus \ol{B_{r/A}(z)} \right)$-geodesic from a point of $\bdy B_r(z)$ to a point of $\bdy B_{A r}(z)$ passes through $Z$. \label{item-conf-event-conf} 
\item There is a number $\rho \in [A r , A^2 r / 4]$ with the following property. There is path in the annulus $B_{2\rho}(z) \setminus B_\rho(z)$ which disconnects the inner and outer boundaries of the annulus and whose $D_h$-length is at most $\frac12 D_h(\bdy B_{2\rho}(z) , \bdy B_{4\rho}(z))$. \label{item-conf-event-dist}
\item There is a number $\rho' \in [r/A , r / 4]$ with the following property. There is path $\pi$ in the annulus $B_{4\rho'}(z) \setminus B_{2\rho'}(z)$ which disconnects the inner and outer boundaries of the annulus and whose $D_h$-length is at most $\frac12 D_h(\bdy B_{ \rho'}(z) , \bdy B_{2\rho'}(z))$.  \label{item-conf-event-dist'}
\end{enumerate}
By the locality of the metric (Axiom~\ref{item-metric-local}), we have $ E_r(z) \in \sigma\left( h|_{B_{A^2 r}(z) \setminus B_{r/A}(z)} \right)$ (this is why we use the internal metric on $B_{A^2 r}(z)$ in condition~\ref{item-conf-event-conf}).
It is easily seen from Axiom~\ref{item-metric-f} that adding a constant to $h$ does not affect the occurrence of $E_r(z)$, so $E_r(z)$ is determined by $h|_{B_{A^2 r}(z) \setminus B_{r/A}(z)}$ viewed modulo additive constant.
On $E_r(z)$, we can choose $Z_r(z) \in B_{A r}(z) \setminus B_r(z)$ in a manner depending only on $h|_{B_{A^2 r}(z) \setminus B_{r/A}(z)}$ viewed modulo additive constant, such that condition~\ref{item-conf-event-conf} in the definition of $E_r(z)$ occurs with $Z = Z_r(z)$. 
On the complement of $E_r(z)$, we arbitrarily define $Z_r(z) := z + (r,0)$.  
Then the pair $(E_r(z) , Z_r(z))$ satisfies the measurability condition in the lemma statement. 
\medskip

\noindent\textit{Step 2: properties of the event.}
We now assume that $E_r(z)$ occurs and check the three numbered properties in the lemma statement.
We start with property~\eqref{item-avoids0}. 
Let $\rho$ be as in condition~\ref{item-conf-event-dist} in the definition of $E_r(z)$ and let $\pi$ be the path in $B_{2\rho}(z)\setminus B_\rho(z)$ as in that condition. 
Suppose $P : [0,T] \rta \BB C$ is a path between two points of $\ol{B_{A r}(z)}$ which exits $B_{A^2 r}(z)$. 
We claim that $P$ is not a $D_h$-geodesic.
Indeed, since each of $\pi$ and $B_{4\rho}(z) \setminus B_{2\rho}(z)$ disconnects $\ol{B_{A r}(z)}$ from $\BB C\setminus B_{A^2 r}(z)$ there must be times $0 < s < t < T$ with the following properties. We have $P(s) , P(t) \in \pi$ and $P$ crosses between the inner and outer boundaries of the annulus $B_{4\rho}(z) \setminus B_{2\rho}(z)$ between time $s$ and time $t$. 
Since the $D_h$-length of $\pi$ is at most $\frac12 D_h(\bdy B_{ \rho'}(z) , \bdy B_{2\rho'}(z))$ it follows that the $D_h$-distance from $P(s)$ to $P(t)$ is at most half of the $D_h$-length of $P|_{[s,t]}$. Therefore $P$ is not a $D_h$-geodesic.
This gives property~\eqref{item-avoids0}.
We similarly obtain property~\eqref{item-avoids} from condition~\ref{item-conf-event-dist'} in the definition of $E_r(z)$, with the path $\pi$ as in condition~\ref{item-conf-event-dist'}. 

We now check property~\eqref{item-confluence}.
The combination of properties \eqref{item-avoids0} and~\eqref{item-avoids} of $E_r(z)$ tells us that every $D_h$-geodesic between points of $\ol{B_{A  r}(z)\setminus B_r(z)}$ is contained in $B_{A^2 r}(z)\setminus \ol{B_{r/A}(z)}$. 
This implies that the set of $D_h\left(\cdot,\cdot ; B_{A^2 r}(z)\setminus \ol{B_{r/A}(z)} \right)$-geodesics between any two points of $\ol{B_{A  r}(z)\setminus B_r(z)}$ is the same as the set of $D_h$-geodesics between these two points.
Consequently, condition~\ref{item-conf-event-conf} in the definition of $E_r(z)$ (together with the definition of $Z_r(z)$) implies that every $D_h$-geodesic from a point of $\bdy B_r(z)$ to a point of $\bdy B_{A r}(z)$ passes through $Z_r(z)$. 
A $D_h$-geodesic from a point of $B_r(z)$ to a point of $\BB C\setminus B_{A r}(z)$ has a sub-segment which is a $D_h$-geodesic from a point of $\bdy B_r(z)$ to a point of $\bdy B_{A r}(z)$, so any such $D_h$-geodesic must also pass through $Z_r(z)$. 
This gives property~\eqref{item-confluence}. 
\medskip

\noindent\textit{Step 3: estimating the probability of $E_r(z)$.}
It remains to show that we can choose $A$ in such a way that $\BB P[E_r(z)] \geq p$ for each $z\in\BB C$ and $r>0$. 
By the scale and translation invariance of the law of $h$, modulo additive constant, and Axioms~\ref{item-metric-f} and~\ref{item-metric-coord}, $\BB P[ E_r(z) ]$ does not depend on $z$ or $r$. 
Hence it suffices to choose $A$ so that $\BB P[E_1(0))] \geq p$. 

We first deal with condition~\ref{item-conf-event-dist} as follows. 
For $\rho > 0$, let $G_\rho$ be the event that there is a path in $B_{2\rho}(0) \setminus B_\rho(0)$ which disconnects the inner and outer boundaries of the annulus and whose $D_h$-length is at most $\frac12 D_h(\bdy B_{2\rho}(0) , \bdy B_{4\rho}(0))$. 
By the scale invariance of the law of $h$, modulo additive constant, together with Axioms~\ref{item-metric-f} and~\ref{item-metric-coord} (Weyl scaling and coordinate change), we see that $\BB P[G_\rho]$ does not depend on $\rho$.
By Axiom~\ref{item-metric-local} we see that $G_\rho$ is a.s.\ determined by $h|_{B_{4\rho}(0) \setminus B_\rho(0)}$ viewed modulo additive constant.
By an easy absolute continuity argument (see, e.g.,~\cite[Lemma 6.1]{gwynne-ball-bdy}) we have $q := \BB P[G_1] >  0$. Since $\BB P[G_\rho] = q $ for every $\rho  >0$ and the tail $\sigma$-algebra $\bigcap_{\rho > 1} \sigma\left( h|_{\BB C\setminus B_\rho(0)} \right)$ is trivial, it follows that a.s.\ $G_\rho$ occurs for infinitely many positive integer values of $\rho$. 

Therefore, we can choose $A_0 > 4$ such that with probability at least $1-p/3$, the event $G_\rho$ occurs for at least one value of $\rho$ in $[1,A_0/4]$.
By scale invariance, if $A > 0$ it also holds with probability at least $1-p/3$ that $G_\rho$ occurs for at least one value of $\rho$ in $[A  , A_0 A/4]$. 
Hence if $A \geq A_0$ then condition~\ref{item-conf-event-dist} in the definition of $E_1(0)$ occurs with probability at least $1-(1-p)/3$. 
By an identical argument, we see that after possibly increasing $A_0$, for any $A\geq A_0$ it holds with probability at least $1-2(1-p)/3$ that conditions~\ref{item-conf-event-dist} and~\ref{item-conf-event-dist'} in the definition of $E_1(0)$ both occur.
 
By Lemma~\ref{lem-general-confluence-scale} (with $r=1$) combined with the preceding paragraph, there exists $A = A(p,\gamma)  \geq A_0$ such that with probability at least $p$, conditions~\ref{item-conf-event-conf} and~\ref{item-conf-event-dist'} in the definition of $E_1(0)$ both occur and also there is a point $Z\in B_A(0) \setminus B_1(0)$ such that every $D_h$-geodesic from a point of $\bdy B_1(0)$ to a point of $\bdy B_A(0)$ passes through $Z$. 
As explained in step 2, if conditions~\ref{item-conf-event-dist} and~\ref{item-conf-event-dist'} in the definition of $E_1(0)$ both occur then the set of $D_h\left(\cdot,\cdot ;  B_{A^2 }(0)\setminus \ol{B_{1/A}(0)}  \right)$-geodesics between any two points of $B_A(0) \setminus \ol{B_1(0)}$ is the same as a $D_h$-geodesics between these two points. 
We therefore have $\BB P[E_1(0)] \geq p$, as required. 
\end{proof}

\section{Zero-one law for LQG geodesics}
\label{sec-geo-dim}

In this section, we prove a zero-one law for LQG geodesics (Theorem~\ref{thm-geo-dim}).  We begin by proving a zero-one law for the Euclidean dimension of the infinite geodesic ray $P_0^\infty$ from Proposition~\ref{prop-infty-geo}. 
This case is easier than the case of general geodesics since the law of $P_0^\infty$ is scale invariant.

\begin{prop} \label{prop-geo-zero-one}
Let $P_0^\infty$ be the geodesic ray from 0 to $\infty$ as in Proposition~\ref{prop-infty-geo}. 
There is a deterministic constant $\Delta_{\op{geo}} >0$ such that the random variable $\dim_{\mcl H}^0 P_0^\infty$ is a.s.\ equal to $\Delta_{\op{geo}}$. 
Moreover, a.s.\ $\dim_{\mcl H}^0 P \geq \Delta_{\op{geo}}$ for each $D_h$-geodesic $P$ from 0 to a point of $\BB C\setminus \{0\}$. 
\end{prop}
\begin{proof}
The proof is similar to that of Theorem~\ref{thm-net-dim}.
Let $c >0$ such that $\BB P[\dim_{\mcl H}^0 P_0^\infty \geq c] > 0$. We claim that a.s.\ for each $D_h$-geodesic $P$ from 0 to a point of $\BB C\setminus \{0\}$,  
\eqb \label{eqn-geo-dim-as}
 \dim_{\mcl H}^0( P  \cap B_r(0) ) \geq c ,\quad \forall r >0  .
\eqe
Applying~\eqref{eqn-geo-dim-as} with $P = P_0^\infty$ shows that $ \dim_{\mcl H}^0 P_0^\infty$ is a.s.\ equal to a deterministic constant, and then applying~\eqref{eqn-geo-dim-as} for an arbitrary choice of $P$ shows that a.s.\ $\dim_{\mcl H}^0 P$ is bounded below by this constant for every $D_h$-geodesic $P$ from 0 to a point of $\BB C\setminus \{0\}$. 

To prove~\eqref{eqn-geo-dim-as}, we first use the countable stability of Hausdorff dimension to get that for each $\delta > 0$, there exists $r_0  > 0$ and $p \in (0,1)$ such that $\BB P\left[  \dim_{\mcl H}^0\left(  P_0^\infty \cap B_{r_0}(0)  \right)    \geq c  - \delta \right] \geq p$. By the scale invariance of the law of $h$, modulo additive constant, and~\eqref{eqn-infty-geo-scale},
\eqb \label{eqn-geodim-scale0}
\BB P\left[  \dim_{\mcl H}^0\left(  P_0^\infty \cap B_r(0)  \right)    \geq c - \delta  \right] \geq p ,\quad\forall r > 0. 
\eqe 
Obviously, $\dim_{\mcl H}^0(P_0^\infty \cap B_r(0))$ is increasing in $r$, so by~\eqref{eqn-geodim-scale0},
\eqb \label{eqn-geodim-scale1}
\BB P\left[ \text{$\exists$ arbitrarily small values of $r > 0$ such that}\: \dim_{\mcl H}^0\left(  P_0^\infty \cap B_r(0)  \right)    \geq c - \delta  \right] \geq p  .
\eqe
By confluence of geodesics started from 0~\cite[Theorem 1.3]{gm-confluence} (plus the fact that the LQG metric induces the same topology as the Euclidean metric), a.s.\ for each $r > 0$ there exists $r'  \in (0,r)$ such that for every $D_h$-geodesic $P$ from 0 to a point of $\BB C\setminus B_r(0)$, we have $P\cap B_{r'}(0) = P_0^\infty \cap B_{r'}(0)$. 
By combining this with~\eqref{eqn-geodim-scale1}, we obtain that for each $r_2 > r_1 > 0$, 
\eqb \label{eqn-geodim-scale2}
\BB P\left[  \dim_{\mcl H}^0(P\cap B_{r_1}(0)) \geq c-\delta,\:\text{$\forall$ geodesic $P$ from 0 to a point of $\BB C\setminus B_{r_2}(0)$} \right] \geq p . 
\eqe 

We will now deduce~\eqref{eqn-geo-dim-as} from~\eqref{eqn-geodim-scale2} together with tail triviality considerations. To do this we will use Lemma~\ref{lem-conf-event} for convenience, but we do not need the full force of the lemma here (we do need all of the conditions from Lemma~\ref{lem-conf-event} to treat the case of the metric ball boundary, however). 
Let $A = A(1-p/2,\gamma)$ be as in Lemma~\ref{lem-conf-event} with $1-p/2$ in place of $p$.
For $r > 0$, let $E_r = E_r(0)$ be the event from that lemma, so that $E_r\in \sigma\left( h |_{B_{A^2 r}(0)} \right)$ and $\BB P[E_r] \geq 1-p/2$. 
Let $G_r$ be the intersection of $E_r$ with the event that every $D_h$-geodesic $P$ from 0 to a point of $\bdy B_{A r}(0)$ satisfies  $ \dim_{\mcl H}^0 (P\cap B_r(0)) \geq c -\delta$.
By~\eqref{eqn-geodim-scale2} (with $r_1 = r$ and $r_2 = A r$),  
\eqb \label{eqn-geodim-pos}
\BB P[G_r] \geq p/2 ,\quad\forall r > 0 .
\eqe 

Recall from Lemma~\ref{lem-conf-event} that on $E_r$, every $D_h$-geodesic from 0 to a point of $\bdy B_{A r}(0)$ is contained in $B_{ A^2 r}(0)$, so on $E_r$ the set of such $D_h$-geodesics is the same as the set of $D_h(\cdot,\cdot ; B_{A^2 r}(0))$-geodesics from 0 to points of $\bdy B_{A r}(0)$. 
Since $E_r \in \sigma\left( h |_{B_{A^2 r}(0)} \right)$ and by Axiom~\ref{item-metric-local} (locality), we get that $G_r \in \sigma\left( h |_{B_{A^2 r}(0)} \right)$. 

By~\eqref{eqn-geodim-pos}, it holds with probability at least $p/2$ that there are arbitrarily small values of $r>0$ for which $G_r$ occurs. 
Since $G_r\in  \sigma\left( h|_{B_{A^2 r}(0)} \right)$ and the tail $\sigma$-algebra $\bigcap_{r  > 0} \sigma(h|_{B_{A^2 r}(0)})$ is trivial~\cite[Lemma 2.2]{hs-euclidean}, this implies that in fact a.s.\ $G_r$ occurs for arbitrarily small values of $r > 0$. 
Henceforth assume that we are working on the (full probability) event that this is the case.

For each $D_h$-geodesic $P$ from 0 to a point $z\in \BB C\setminus \{0\}$ and each $r \in (0,|z|/A)$ there is a segment of $P$ which is a $D_h$-geodesic from 0 to a point of $\bdy B_{A r}(0)$. 
If $G_r$ occurs, then this segment of $P$ has Euclidean dimension at least $c - \delta$. 
From the preceding paragraph, we therefore get that $ \dim_{\mcl H}^0(P\cap B_r(0)) \geq c - \delta$ for arbitrarily small values of $r>0$. 
Since $\delta>0$ can be made arbitrarily small, this implies~\eqref{eqn-geo-dim-as}. 
\end{proof}

We now want to argue that in fact the Euclidean dimension of \emph{any} $D_h$-geodesic started from 0 is bounded above by the constant $\Delta_{\op{geo}}$ from Proposition~\ref{prop-geo-zero-one}. The idea of the proof is that if we see $\mcl B_s^\bullet $ and $h|_{\mcl B_s^\bullet }$ for some $s> 0$, then we cannot tell which $D_h$-geodesic from 0 to $\bdy \mcl B_s^\bullet $ is equal to $P_0^\infty|_{[0,s]}$, so all of these $D_h$-geodesics must have dimension at most $\Delta_{\op{geo}}$. 
The following lemma makes precise the idea that $\mcl B_s^\bullet $ and $h|_{\mcl B_s^\bullet }$ do not determine which point of $\bdy \mcl B_s^\bullet$ is hit by $P_0^\infty$.

\begin{figure}[t!]
 \begin{center}
\includegraphics[scale=1]{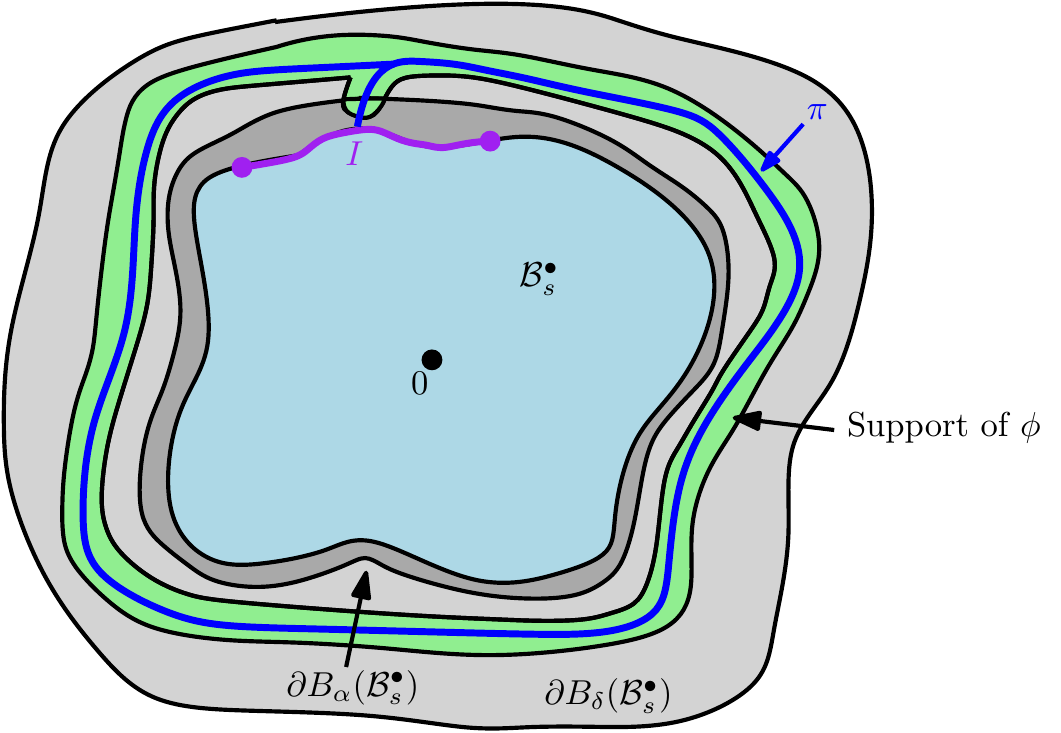}
\vspace{-0.01\textheight}
\caption{Illustration of the proof of Lemma~\ref{lem-geo-arc}. We condition on $(\mcl B_s^\bullet  , h|_{\mcl B_s^\bullet})$ and consider the field $\wt h = h - C \phi$ where $\phi$ is a smooth bump function supported in the light green region and $C$ is large. Then the conditional laws of $h$ and $\wt h$ given $(\mcl B_s^\bullet  , h|_{\mcl B_s^\bullet})$ are mutually absolutely continuous. On the other hand, if we make the constant $C$ large enough and we make the support of $\phi$ sufficiently close to $I$, then the $D_{\wt h}$-distance from any point $z$ on the blue path $\pi$ to $I$ is smaller than the $D_{\wt h}$-distance from $z$ to any point of $\bdy\mcl B_s^\bullet\setminus I$. 
Since every path from 0 to a point outside of $B_\delta(\mcl B_s^\bullet)$ passes through $\pi$, this forces every $D_{\wt h}$-geodesic from 0 to a point outside of $B_\delta(\mcl B_s^\bullet)$ to pass through $I$ (otherwise, we could replace it by a shorter path which did pass through $I$). 
}\label{fig-geo-arc}
\end{center}
\vspace{-1em}
\end{figure}

\begin{lem} \label{lem-geo-arc}
Fix $s  >0$ and $\delta > 0$. 
Let $I$ be a non-trivial arc of $\bdy\mcl B_s^\bullet$ chosen in a $\sigma(\mcl B_s^\bullet  , h|_{\mcl B_s^\bullet})$-measurable manner (recall from Lemma~\ref{lem-metric-curve}that $\bdy\mcl B_s^\bullet$ is a Jordan curve).
Almost surely, it holds with positive conditional probability given $\sigma(\mcl B_s^\bullet  , h|_{\mcl B_s^\bullet})$ that the following is true.
For each $w \in \BB C\setminus B_\delta(\mcl B_s^\bullet)$, every $D_h$-geodesic from 0 to $w$ passes through $I$ (recall that $B_\delta(\cdot)$ denotes the Euclidean $\delta$-neighborhood). 
\end{lem}
\begin{proof}
To lighten notation, let $\mcl F := \sigma(\mcl B_s^\bullet  , h|_{\mcl B_s^\bullet})$. Throughout the proof, we will condition on $\mcl F$ and all choices are required to be made in a $\mcl F$-measurable manner. See Figure~\ref{fig-geo-arc} for an illustration of the argument. 

We first choose (in a $\mcl F$-measurable manner) a path $\pi$ in the annular region $\left( B_\delta(\mcl B_s^\bullet) \setminus \mcl B_s^\bullet \right) \cup I $ which starts from a point of $I$, lies at positive Euclidean distance from $\bdy\mcl B_s^\bullet\setminus I$, and disconnects $\mcl B_s^\bullet$ from $\bdy B_\delta(\mcl B_s^\bullet)$. 
Let $\ep  > 0$ be small enough so that $\pi$ lies at Euclidean distance at least $100\ep$ from each of $\bdy\mcl B_s^\bullet\setminus I$ and $\bdy B_\delta(\mcl B_s^\bullet)$. 

Since $D_h$ a.s.\ induces the Euclidean topology, we can find small enough $\zeta , \alpha \in (0, (\ep\wedge\delta) /100)$ and a large enough $A >1$ (all random and $\mcl F$-measurable) such that the following is true with conditional probability at least $1/2$ given $\mcl F$.
\begin{enumerate}
\item The $D_h$-distance from the $2\ep$-neighborhood $B_{2\ep}(\pi)$ to each of $\bdy\mcl B_s^\bullet\setminus I$ and $\bdy B_\delta(\mcl B_s^\bullet)$ is at least $100 \zeta$. \label{item-geo-arc-far}
\item The $D_h$-distance from each point of $B_{2\alpha}(\mcl B_s^\bullet)$ to $\mcl B_s^\bullet$ is at most $\zeta$.  \label{item-geo-arc-close}
\item For each $z\in \pi \setminus B_{2\alpha}(\mcl B_s^\bullet)$, there is a path from $z$ to a point of $B_{2\alpha}(\mcl B_s^\bullet)$ which is contained in $B_\ep(\pi) \setminus B_{2\alpha}(\mcl B_s^\bullet)$ and has $D_h$-length at most $A$. \label{item-geo-arc-internal}
\end{enumerate} 
Let $E$ be the event that the above numbered conditions hold, so that $\BB P[E  |\mcl F] \geq 1/2$.

Let $\phi : [0,1]\rta \BB C$ be a smooth bump function which supported on a compact subset of $B_{2\ep}(\pi) \setminus \mcl B_s^\bullet$ and which is identically equal to 1 on $B_\ep(\pi) \setminus B_{ \alpha}(\mcl B_s^\bullet)$. 

Recall~\cite[Lemma 2.1]{gm-confluence} that $\mcl B_s^\bullet$ is a local set for $h$, so under the conditional law of $h|_{\BB C\setminus \mcl B_s^\bullet}$ given $\mcl F$ is that of a zero-boundary GFF on $\BB C\setminus\mcl B_s^\bullet$ plus an $\mcl F$-measurable harmonic function. 
By a standard absolute continuity statement for the GFF, if we set
\eqbn
\wt h := h - \frac{1}{\xi} \log (\frac{A}{\zeta}) \phi ,
\eqen
then the conditional laws of $h$ and $\wt h$ given $\mcl F$ are mutually absolutely continuous. 
We also note that by Weyl scaling (Axiom~\ref{item-metric-f}) and since $\phi$ vanishes on $\mcl B_s^\bullet$, the definition of $\mcl B_s^\bullet$ is unaffected by replacing $h$ with $\wt h$. 
Since we know that $\BB P[E | \mcl F]  >0$, it therefore suffices to show that if $E$ occurs then every $D_{\wt h}$-geodesic from 0 to a point outside of $B_\delta(\mcl B_s^\bullet)$ passes through $I$. 

Henceforth assume that $E$ occurs. The rest of the argument is deterministic. 
By Weyl scaling the three conditions in the definition of $E$ lead to the following properties of $D_{\wt h}$.
\begin{enumerate}
\item Since $\phi$ vanishes outside of $B_{2\ep}(\pi)$, the $D_{\wt h}$-distance from the $2\ep$-neighborhood $B_{2\ep}(\pi)$ to each of $\bdy\mcl B_s^\bullet\setminus I$ and $\bdy B_\delta(\mcl B_s^\bullet)$ is at least $100 \zeta$. \label{item-geo-arc-far'}
\item Since $\phi$ is non-negative, $D_{\wt h} \leq D_h$ so in particular the $D_{\wt h}$-distance from each point of $B_{2\alpha}(\mcl B_s^\bullet)$ to $\mcl B_s^\bullet$ is at most $\zeta$.  \label{item-geo-arc-close'}
\item Since $\phi \equiv 1$ on $B_\ep(\pi) \setminus B_{ \alpha}(\mcl B_s^\bullet)$, for each $z\in \pi \setminus B_{2\alpha}(\mcl B_s^\bullet)$, there is a path from $z$ to a point of $B_{2\alpha}(\mcl B_s^\bullet)$ which is contained in $B_\ep(\pi) \setminus B_{2\alpha}(\mcl B_s^\bullet)$ and has $D_{\wt h}$-length at most $\zeta$. \label{item-geo-arc-internal'}
\end{enumerate}

Properties~\ref{item-geo-arc-close'} and~\ref{item-geo-arc-internal'} together imply that the $D_{\wt h}$-distance from each point of $\pi$ to $\mcl B_s^\bullet$ is at most $2\zeta$, and hence the $D_{\wt h}$-distance from each point of $\pi$ to 0 is at most $s + 2\zeta$. Property~\ref{item-geo-arc-far'} implies that the $D_{\wt h}$-distance from each point of $\pi$ to $\mcl B_s^\bullet\setminus I$ is at least $100\zeta$.
Consequently, every path from 0 to any point of $\pi$ which does not pass through $I$ has $D_{\wt h}$-length at least $s+100\zeta$, so cannot be a $D_{\wt h}$-geodesic.
In other words, every $D_{\wt h}$-geodesic from 0 to any point of $\pi$ passes through $I$. 
Since $\pi$ disconnects $\mcl B_s^\bullet$ from $\bdy B_\delta(\mcl B_s^\bullet)$, every $D_{\wt h}$-geodesic from 0 to any point outside of $B_\delta(\mcl B_s^\bullet)$ must pass through $\pi$, and hence must also pass through $I$. 
\end{proof}

\begin{lem} \label{lem-geo-dim-bdy}
Fix $s > 0$. Almost surely, the Euclidean dimension of every $D_h$-geodesic from 0 to a point of $\bdy\mcl B_s^\bullet $ is equal to $\Delta_{\op{geo}}$.
\end{lem}
\begin{proof}
Fix $t\in (0,s)$ and $\ep > 0$. Let $\mcl X = \mcl X_{t-\ep , t} \subset \bdy \mcl B_{t-\ep}^\bullet$ be the set of confluence points from Theorem~\ref{thm-conf} so that every leftmost $D_h$-geodesic from 0 to a point of $\bdy\mcl B_t^\bullet$ passes through some $x\in\mcl X $. 
As in Theorem~\ref{thm-conf}, for $x\in\mcl X$, let $I_x\subset \bdy\mcl B_t^\bullet$ be the set of $y\in\bdy\mcl B_t^\bullet$ such that the leftmost geodesic from 0 to $y$ passes through $x$. 
Note that $\mcl X$ and the sets $I_x$ for $x\in\mcl X$ are $\sigma(\mcl B_t^\bullet , h|_{\mcl B_t^\bullet})$-measurable. 

By assertion~\ref{item-conf-unique} of Theorem~\ref{thm-conf}, a.s.\ there is a unique $D_h$-geodesic $P_x$ from 0 to each $x\in\mcl X $. 
We claim that a.s.\ $\dim_{\mcl H}^0 P_x = \Delta_{\op{geo}}$ for each $x\in\mcl X$ for which $I_x$ is not a singleton. Given the claim, we can conclude the proof as follows.
By Proposition~\ref{prop-conf-endpt}, a.s.\ each $D_h$-geodesic from 0 to a point of $\BB C\setminus \mcl B_t^\bullet$ passes through some $x\in\mcl X$ for which $I_x$ is not a singleton.
By the uniqueness of $P_x$, each such $D_h$-geodesic coincides with $P_x$ on the time interval $[0,t-\ep]$. 
In particular, each $D_h$-geodesic $P$ from 0 to a point of $\bdy\mcl B_s^\bullet$ satisfies $\dim_{\mcl H}^0 P([0,t-\ep]) = \Delta_{\op{geo}}$. Sending $t\rta s$ and $\ep\rta 0$ then concludes the proof. 

To prove the above claim, let $x_*\in\mcl X$ be chosen in a $\sigma(\mcl B_t^\bullet , h|_{\mcl B_t^\bullet})$-measurable manner and assume that $I_{x_*}$ is not a singleton. By condition~\ref{item-conf-arcs} of Theorem~\ref{thm-conf}, $I_{x_*}$ is a non-trivial connected arc of $\bdy\mcl B_t^\bullet$. 
By Lemma~\ref{lem-geo-arc}, a.s.\ it holds with positive conditional probability given $\sigma(\mcl B_t^\bullet , h|_{\mcl B_t^\bullet})$ that the $D_h$-geodesic $P_0^\infty$ from 0 to $\infty$ passes through $I_{x_*}$. If this is the case, then $P_0^\infty|_{[0,t-\ep]} = P_{x_*}$. By Proposition~\ref{prop-geo-zero-one}, this implies that with positive conditional probability given $\sigma(\mcl B_t^\bullet , h|_{\mcl B_t^\bullet})$ we have $\dim_{\mcl H}^0 P_{x_*} = \Delta_{\op{geo}}$. Since $P_{x_*}$ is $\sigma(\mcl B_t^\bullet , h|_{\mcl B_t^\bullet})$-measurable, in fact a.s.\ $\dim_{\mcl H}^0 P_{x_*}  =\Delta_{\op{geo}}$. Applying the same argument to each possible choice of $x_* \in \mcl X$ gives our claim. 
\end{proof}

\begin{lem} \label{lem-geo-dim-bdy-z}
Fix $s > 0 $ and $z\in\BB C$. On the event $\{s < D_h(0,z)\}$, a.s.\ the Euclidean dimension of every $D_h$-geodesic from 0 to a point of $\bdy\mcl B_s^{z,\bullet} $ is equal to $\Delta_{\op{geo}}$.
\end{lem}
\begin{proof}
This follows from exactly the same conformal invariance / absolute continuity argument used in the proof of Proposition~\ref{prop-conf-finite}.
\end{proof}

\begin{proof}[Proof of Theorem~\ref{thm-geo-dim}]
By Lemma~\ref{lem-geo-dim-bdy-z}, it is a.s.\ the case that for each $q\in\BB Q^2$ and each rational $s < D_h(0,q)$, the Euclidean dimension of every $D_h$-geodesic from 0 to a point of $\bdy\mcl B_s^{q,\bullet} $ is equal to $\Delta_{\op{geo}}$.
Henceforth assume that this is the case. 

Now let $P$ be a $D_h$-geodesic from 0 to a point $z\in \BB C\setminus \{0\}$.
Choose a rational time $s\in (0,D_h(0,z))$ and let $q \in \BB Q^2$ be a point which lies in the same connected component of $\BB C\setminus \mcl B_s^{\bullet}$ as $z$.
Then $P|_{[0,s]}$ is a $D_h$-geodesic from 0 to a point of $\bdy\mcl B_s^{q,\bullet} $, so $\dim_{\mcl H}^0(P|_{[0,s]}) = \Delta_{\op{geo}}$. Sending $s\rta D_h(0,z)$ and using the countable stability of Euclidean dimension then shows that $\dim_{\mcl H}^0 P = \Delta_{\op{geo}}$. 
\end{proof}

\section{LQG metric ball boundaries}
\label{sec-ball-bdry}

In the previous section, we proved a zero-one law for LQG geodesics.  For LQG metric ball boundaries, we can prove an even stronger result: Theorem~\ref{thm-ball-bdy-bound} identifies the explicit  a.s.\ $\gamma$-quantum and Euclidean dimensions of an LQG metric ball stopped at the first time it hits a specified point. The reason we can obtain this stronger result is that we can apply the earlier work of~\cite{gwynne-ball-bdy} that identified the \emph{essential supremum} of the $\gamma$-quantum and Euclidean dimensions of LQG metric ball boundaries.  We will also apply a more general result from~\cite{gwynne-ball-bdy} that gives upper bounds for the dimensions of certain subsets of LQG metric ball boundaries.  

In Section~\ref{sec-upper-bound}, we state this generalized upper bound from~\cite{gwynne-ball-bdy}, and we prove a couple of technical lemmas that we need to apply this result to our setting.  In Section~\ref{sec-dim-bdry} we derive a zero-one law for metric ball boundaries, and we use it and the generalized upper bound to prove  Theorem~\ref{thm-ball-bdy-bound}. Finally, in Section~\ref{sec-outer}, we will apply the generalized upper bound to analyze the \emph{exterior boundaries} of LQG metric balls, which we defined in Definition~\ref{def-outer}.

\subsection{A generalized upper bound}
\label{sec-upper-bound}

The proofs in Sections~\ref{sec-dim-bdry} and~\ref{sec-outer} will use the generalized upper bound theorem~\cite[Theorem 2.9]{gwynne-ball-bdy}, which gives upper bounds for  the $\gamma$-quantum and Euclidean dimensions of a large class of subsets of LQG metric ball boundaries.  We restate this theorem here.  We have chosen to state the theorem in slightly less than its full generality only to avoid introducing extra notation that we do not need for our applications.

\begin{thm}[Generalized upper bound] \label{thm-gen-upper}
Suppose that we are given events $\{F_\ep(z) : \ep > 0 , z\in\BB C\}$ and $q>0$ with the following properties. 
For any $\alpha\in [-2,2]$, any $\zeta \in (0,1)$, any bounded open set $V\subset\BB C$ with $\ol V\subset \BB C\setminus \{0\}$, the following is true.
\begin{enumerate}
\item For each $z\in V$,
\eqb \label{eqn-gen-upper-prob}
\BB P\left[ F_\ep(z) \cap \left\{  \sup_{u,v\in B_\ep(z)} D_h(u,v )  \in \left[  \ep^{\xi(Q-\alpha) + \zeta} , \ep^{\xi(Q-\alpha) - \zeta} \right]   \right\} \right] 
\leq \ep^{\alpha^2/2 + q  + o_\zeta(1)+ o_\ep(1) }, 
\eqe
where the rate of the $o_\zeta(1)$ depends only on $\alpha, \gamma$ and the rate of the $o_\ep(1)$ depends only on $ V, \alpha,\zeta,\gamma $ (not on the particular choice of $z$).\footnote{Here and in what follows, for two functions $f,g$ of a positive real number $x$ we write $f(x) = o_x(g(x))$ (resp.\ $f(x) = O_x(g(x))$) if $f(x)/g(x)$ goes to zero (resp.\ remains bounded) as $x\rta 0$. The dependencies of the rate of convergence will always be specified unless they are clear from the context. }
\item There exists an open set $U \subset \BB C$ which contains zero and lies at positive distance from $V$ such that for each small enough $\ep > 0$ (depending on $V$), each of the events $F_\ep(z)$ for $z\in V$ is a.s.\ determined by $h|_{\BB C\setminus U}$.  \label{item-gen-upper-local} 
\end{enumerate}
For $s  >0$, let $\mcl Y_{s}$ be the set of $z\in\bdy\mcl B_{s}$ such that
\eqb \label{eqn-gen-upper-set}
\bigcup_{\ep > 0} \bigcap_{r\in (0,\ep) \cap \BB Q} \bigcap_{w\in B_r(z) \cap \BB Q^2} F_{ r}(w) 
\eqe
occurs, i.e., for each small enough rational $r > 0$, the event $F_{ r}(w)$ occurs for every $w\in B_r(z)\cap \BB Q^2 $ (we consider rational values of $r$ and $w$ to avoid measurability issues).  
Then, almost surely,  
\[
\dim_{\mcl H}^0 Y_{s} \leq
\max\left\{ 0 , 2- \xi Q + \xi^2/2 - q   \right\},
\]
and  
\[
\dim_{\mcl H}^\gamma Y_{s} \leq
\max\left\{ 0 ,   \sup_{\alpha\in [-2,2]} \frac{2 - \alpha^2/2 -q}{ \xi(Q-\alpha) } -1  \right\}.
\] 
\end{thm}

We will apply Theorem~\ref{thm-gen-upper} for a particular type of events $F_\ep(z)$.  Roughly speaking, we define some ``good'' event $G_r(z)$ that depends locally on $h$, viewed modulo additive constant, and has uniformly positive probability across scales. We let $F_\ep(z)$ be the ``very bad'' event that none of the events $G_r(z)$ occur for a particular range of $r$ values (depending on $\ep$).  The following lemma asserts that the very bad events $F_\ep(z)$ satisfy the conditions of the generalized upper bound.    If we think of ``very bad points'' $z$ as points for which $F_\ep(z)$ occurs, then the generalized upper bound a.s.\ bounds the dimension of points ``surrounded'' by very bad points on some sufficiently small scale.

\begin{lem}
\label{lem-check-conditions}
Suppose that $G_r(z)$ is an event defined for each  $z\in \BB C$ and $r>0$ with the following two properties: 
\begin{itemize}
\item
 The probability of $G_r(z)$ is positive and does not depend on $r$ or $z$.
\item
There exists $b > a > 0$ such that for each $z , r$, the event $G_r(z)$ is a.s.\ determined by
$h|_{B_{b r}(z) \setminus B_{a r}(z) }$, viewed modulo additive constant.
\end{itemize}
Then the event
\eqb
\label{eqn-defn-F-ep-z}	
F_\ep(z) := \bigcap_{r \in [\ep^{1/2}/a , 2\ep^{1/4}/b] \cap \BB Q } [G_r(z)]^c 
\eqe
satisfies the conditions of Theorem~\ref{thm-gen-upper} for some $q > 0$. 
\end{lem}

To prove Lemma~\ref{lem-check-conditions}, we first check that $F_\ep(z)$ is sufficiently local and that its probability (not intersected with any other event) decays sufficiently fast as $\ep \rta 0$:

\begin{lem} \label{lem-outer-bdy-msrble}
With $F_\ep(z)$ defined just above, $F_\ep(z)$ is a.s.\ determined by $h|_{ B_{2\ep^{1/4}}(z) \setminus B_{\ep^{1/2} }(z)}$ viewed modulo additive constant.
Furthermore, there is an exponent $q = q(\gamma) > 0$ such that $\BB P[F_\ep(z)] = O_\ep(\ep^q)$ uniformly over all $z\in\BB C$. 
\end{lem}
\begin{proof}
From the locality of $G_r(z)$ and the definition of $F_\ep(z)$, it is immediate that $F_\ep(z)$ is a.s.\ determined by $h|_{B_{2\ep^{1/4}}(z) \setminus B_{\ep^{1/2} }(z) }$ viewed modulo additive constant.

To prove the second part of the lemma, let $p = p(\gamma) > 0$ be a constant such that $\BB P[G_r(z)] \geq p$ for each $z\in\BB C$ and $r  > 0$. 
We can now use a general independence lemma for the restrictions of the GFF in disjoint concentric annuli~\cite[Lemma~3.1]{local-metrics} to find that there is a $q = q(\gamma) > 0$ such that $\BB P\left[ F_\ep(z) \right] = O_\ep( \ep^q ) $.
\end{proof}

We now show that, for any events $F_\ep(z)$ with the two properties we proved in Lemma~\ref{lem-outer-bdy-msrble}, $F_\ep(z)$ satisfies the upper bound~\eqref{eqn-gen-upper-prob} for the probability of $F_\ep(z)$ intersected with the event that the $D_h$-diameter of $B_\ep(z)$ lies in a certain interval. 

\begin{lem} \label{lem-outer-bdy-prob}
Suppose $q > 0$ and we are given events $F_\ep(z)$ for $\ep > 0$ and $z\in\BB C$ such that $F_\ep(z)$ is a.s.\ determined by $ h|_{\BB C\setminus B_{\ep^{1/2}}(z)}  $, viewed modulo additive constant, and $\BB P[F_\ep(z)] = O_\ep(\ep^q)$ as $\ep\rta 0$, uniformly over all $z\in\BB C$. 
Also let $\alpha \in [-2,2]$ and $\zeta\in (0,1)$ and let $V\subset \BB C$ be bounded open set with $\ol V\subset \BB C\setminus \{0\}$. 
Then for each $z\in V$, 
\eqb \label{eqn-outer-bdy-prob}
\BB P\left[ F_\ep(z) \cap \left\{  \sup_{u,v\in B_\ep(z)} D_h(u,v )  \in \left[  \ep^{\xi(Q-\alpha) + \zeta} , \ep^{\xi(Q-\alpha) - \zeta} \right]   \right\} \right] 
\leq \ep^{\alpha^2/2 + q    + o_\zeta(1)+ o_\ep(1) }, 
\eqe
where the rate of convergence of the $o_\zeta(1)$ depends only on $\alpha, \gamma$ and the rate of convergence of the $o_\ep(1)$ depends only on $  V, \alpha,\zeta,\gamma $.
\end{lem}
\begin{proof}
To lighten notation, let 
\eqb
H_\ep(z) := \left\{  \sup_{u,v\in B_\ep(z)} D_h(u,v )  \in \left[  \ep^{\xi(Q-\alpha) + \zeta} , \ep^{\xi(Q-\alpha) - \zeta} \right]   \right\} .
\eqe
By Lemma~\ref{lem-outer-bdy-msrble} and a basic estimate for $D_h$-diameters (see, e.g.,~\cite[Lemma 2.3]{gwynne-ball-bdy}), for each $z\in V$ we have
\eqb
\BB P[F_\ep(z) ] = O_\ep(\ep^q) \quad \text{and} \quad \BB P\left[ H_\ep(z) \right] \leq \ep^{\alpha^2/2 + o_\zeta(1) + o_\ep(1)} .
\eqe
The idea of the proof is that $F_\ep(z)$ depends only on $h|_{\BB C\setminus B_{\ep^{1/2}(z)}}$ viewed modulo additive constant, whereas $H_\ep(z)$ is (almost) determined by $h|_{3\ep}(z)$, so $F_\ep(z)$ and $H_\ep(z)$ are approximately independent. 
However, $H_\ep(z)$ is not exactly determined by $h|_{B_{3\ep}(z)}$ since in $D_h$-geodesics paths between points of $B_\ep(z)$ could get very far from $B_\ep(z)$. 
So, to make the above idea precise we need to introduce a localized version of $H_\ep(z)$. 
\medskip

\noindent\textit{Step 1: localizing $H_\ep(z)$.}
Let $\wt H_\ep(z)$ be the event that the following is true.
\begin{enumerate}
\item There is a path in $B_{2\ep}(z) \setminus B_\ep(z)  $ which disconnects the inner and outer boundaries of $B_{2\ep}(z) \setminus B_\ep(z) $ whose length is at most $\ep^{-\zeta} D_h(\bdy B_{2\ep}(z) , \bdy B_{3\ep}(z))$. \label{item-diam-localize-around}
\item We have $\sup_{u,v\in B_\ep(z)} D_h\left(u,v;B_{3\ep}(z) \right)  \in \left[  \ep^{\xi(Q-\alpha) + \zeta} ,  \ep^{\xi(Q-\alpha) - 2\zeta} \right]  $. \label{item-diam-localize-dist}
\end{enumerate}
By the locality of the metric (Axiom~\ref{item-metric-local}), $\wt H_\ep(z)$ is a.s.\ determined by $h|_{B_{3\ep}(z)}$. 
Furthermore, the proof of~\cite[Lemma 2.3]{gwynne-ball-bdy} shows that
\eqb \label{eqn-diam-localize-prob}
\BB P\left[ \wt H_\ep(z) \right] \leq \ep^{\alpha^2/2     + o_\zeta(1)+ o_\ep(1) } ,\quad \forall z \in V .
\eqe

We now argue that 
\eqb \label{eqn-diam-localize-subpoly}
 \BB P[H_\ep(z) \setminus \wt H_\ep(z)] = O_\ep(\ep^N), \quad\forall N\in\BB N ,  
\eqe
uniformly over all $z\in \BB C$. 
We first claim that the probability that condition~\ref{item-diam-localize-around} in the definition of $\wt H_\ep(z)$ fails to occur decays faster than any positive power of $\ep$ as $\ep\rta 0$. Indeed, this follows from~\cite[Proposition 3.1]{lqg-metric-estimates}, applied to compare each of the distance ``around" $B_{2\ep}(z) \setminus B_\ep(z) $ and the distance ``across" $B_{3\ep}(z) \setminus B_{2\ep}(z) $ to the quantity $(2\ep)^{\xi Q} e^{\xi h_{2\ep}(z)}$. 

We will now conclude the proof of~\eqref{eqn-diam-localize-subpoly} by showing that if both $H_\ep(z)$ and condition~\ref{item-diam-localize-around} in the definition of $\wt H_\ep(z)$ occur, then $\wt H_\ep(z)$ occurs. 
Indeed, if $H_\ep(z)$ occurs then $\sup_{u,v\in B_\ep(z)} D_h\left(u,v;B_{3\ep}(z) \right) \geq  \sup_{u,v\in B_\ep(z)} D_h\left(u,v  \right)  \geq  \ep^{\xi(Q-\alpha) + \zeta}$. 
To get the bound in the other direction, suppose $u,v\in B_\ep(z)$ and let $P$ be a $D_h$-geodesic from $u$ to $v$. If $P$ is contained in $B_{3\ep}(z)$ then $D_h(u,v) = D_h(u,v;B_{3\ep}(z))$. 
Otherwise, we can replace a segment of $P$ by a segment of the path in $B_{2\ep}(z) \setminus B_\ep(z) $ from condition~\ref{item-diam-localize-around} in the definition of $\wt H_\ep(z)$ to get a new path from $u$ to $v$ which stays in $B_{3\ep}(z)$ and whose $D_h$-length is at most $\ep^{-\zeta} D_h(u,v)$. 
Therefore, $D_h(u,v;B_{3\ep}(z)) \leq \ep^{-\zeta} D_h(u,v)  $. 
By the definition of $H_\ep(z)$, we infer that $\wt H_\ep(z)$ occurs, and hence~\eqref{eqn-diam-localize-subpoly} holds.
\medskip

\noindent\textit{Step 2: near-independence of $F_\ep(z)$ and $H_\ep(z)$.}
Let $\ol h$ be an independent copy of $h$. 
By a basic estimate for the GFF (see, e.g.,~\cite[Lemma A.5]{gms-harmonic} applied with $\delta = 3\ep^{1/2}$ combined with the fact that $h(3\ep/\cdot) \eqD h$ modulo additive constant), we get that for some universal constant $a\in (0,1)$, the following is true. 
The conditional law of $h|_{\BB C\setminus B_{\ep^{1/2}}(z)}$, viewed modulo additive constant, given $h|_{B_{3\ep}(z)}$ is mutually absolutely continuous w.r.t.\ the unconditional law of $\ol h|_{\BB C\setminus B_{\ep^{1/2}}(z)}$, viewed modulo additive constant.
Furthermore, if $M = M(h|_{B_{3\ep}(z)} ,  \ol h|_{\BB C\setminus B_{\ep^{1/2}}(z)})$ is the Radon-Nikodym derivative of the former law w.r.t.\ the latter law, then for $\ep \in (0,a]$ the $a/\ep^{1/2}$-th moments of $M$ and its reciprocal are each bounded above by a universal constant.
 
Let $\ol F_\ep(z)$ be defined in the same manner as $F_\ep(z)$ but with $\ol h$ in place of $h$. 
Since $\wt H_\ep(z) \in \sigma\left( h|_{B_{3\ep}(z)} \right)$ and $F_\ep(z) $ is a.s.\ determined by $ h|_{\BB C\setminus B_{\ep^{1/2}}(z)}  $, viewed modulo additive constant, 
\allb \label{eqn-diam-localize-ind}
\BB P\left[ F_\ep(z) \cap \wt H_\ep(z) \right] 
= \BB E\left[ \BB P\left[ F_\ep(z) \,|\,     h|_{B_{3\ep}(z)}  \right] \BB 1_{\wt H_\ep(z)} \right]  
= \BB E\left[  \BB E\left[  M \BB 1_{\ol F_\ep(z)} \,|\,  h|_{B_{3\ep}(z)} \right] \BB 1_{\wt H_\ep(z)} \right]   .
\alle
To bound the conditional expectation in~\eqref{eqn-diam-localize-ind}, we use H\"older's inequality (with exponents $a/\ep^{1/2}$ and $1/(1  -\ep^{1/2}/a)$) to get
\allb \label{eqn-diam-localize-holder}
 \BB E\left[  M \BB 1_{\ol F_\ep(z)} \,|\,  h|_{B_{3\ep}(z)} \right]
&\leq \BB E\left[  M^{a/\ep^{1/2}}  \,|\,  h|_{B_{3\ep}(z)} \right]^{\ep^{1/2}/a} \BB P\left[ \ol F_\ep(z)  \,|\,  h|_{B_{3\ep}(z)}   \right]^{1-\ep^{1/2}/a} \notag\\
&= \BB E\left[  M^{a/\ep^{1/2}}  \,|\,  h|_{B_{3\ep}(z)} \right]^{\ep^{1/2}/a} \BB P\left[  F_\ep(z)   \right]^{1-\ep^{1/2}/a} ,
\alle
where in the last line we used that $\ol h \eqD h$ and $\ol h$ is independent from $h$. We now plug~\eqref{eqn-diam-localize-holder} into~\eqref{eqn-diam-localize-ind} and apply H\"older's inequality a second time (with the same exponents) to get
\allb \label{eqn-diam-localize-end}
\BB P\left[ F_\ep(z) \cap \wt H_\ep(z) \right] 
&\leq  \BB P\left[  F_\ep(z)   \right]^{1-\ep^{1/2}/a} \BB E\left[ \BB E\left[  M^{a/\ep^{1/2}}  \,|\,  h|_{B_{3\ep}(z)} \right]^{\ep^{1/2}/a} \BB 1_{\wt H_\ep(z)} \right] \notag\\
&\leq  \BB P\left[  F_\ep(z)   \right]^{1-\ep^{1/2}/a} \BB E\left[  M^{a/\ep^{1/2}}   \right]^{\ep^{1/2}/a} \BB P\left[ \wt H_\ep(z)  \right]^{1-\ep^{1/2}/a} \quad \text{(by H\"older)} \notag\\
&\leq \text{const.} \times \BB P\left[  F_\ep(z)   \right]^{1-\ep^{1/2}/a}  \BB P\left[ \wt H_\ep(z)  \right]^{1-\ep^{1/2}/a} \quad \text{(by our estimate for $M$)} .
\alle
The lemma now follows by using~\eqref{eqn-diam-localize-subpoly} to upper-bound $\BB P[F_\ep(z) \cap H_\ep(z)]$ in terms of $\BB P[F_\ep(z) \cap \wt H_\ep(z)]$, 
then using~\eqref{eqn-diam-localize-prob} and the fact that $\BB P[F_\ep(z)] = O_\ep(\ep^q)$ to upper-bound the right side of~\eqref{eqn-diam-localize-end}.
\end{proof}

\begin{proof}[Proof of Lemma~\ref{lem-check-conditions}]
The lemma follows from Lemmas~\ref{lem-outer-bdy-msrble} and~\ref{lem-outer-bdy-prob}.
\end{proof}

\subsection{The Hausdorff dimension of metric ball boundaries}
\label{sec-dim-bdry}




In this subsection, we prove  Theorem~\ref{thm-ball-bdy-bound}.  We begin by proving the following zero-one law for the dimension of metric balls started from $\infty$. 

\begin{prop} \label{prop-ball-infty-dim}
Define $\mcl B_0^\infty$ as in Proposition~\ref{prop-ball-infty}. 
There are deterministic constants $\Delta_{\op{ball}}^0, \Delta_{\op{ball}}^\gamma > 0$ such that a.s.\ $\dim_{\mcl H}^0 \bdy\mcl B_0^\infty = \Delta_{\op{ball}}^0 $ and $\dim_{\mcl H}^\gamma \bdy\mcl B_0^\infty = \Delta_{\op{ball}}^\gamma $.
Furthermore, for each fixed $z \in\BB C$ a.s.\ $\dim_{\mcl H}^0 \bdy\mcl B_{D_h(z,w)}(w;D_h) \geq \Delta_{\op{ball}}^0$  and $\dim_{\mcl H}^\gamma \bdy\mcl B_{D_h(z,w)}(w;D_h) \geq \Delta_{\op{ball}}^\gamma$  simultaneously for each $w\in\BB C$. 
\end{prop}

\begin{proof}
We prove the result for Euclidean dimensions; the proof of the result for $\gamma$-quantum dimensions is essentially the same. The basic strategy of the proof is similar to the proofs of Theorem~\ref{thm-net-dim} and Proposition~\ref{prop-geo-zero-one}. 

Let $c > 0$ such that $\BB P\left[ \dim_{\mcl H}^0 \bdy\mcl B_0^\infty \geq c \right]  > 0$. 
We will show that a.s.\ 
\eqb \label{eqn-ball-infty-dim-as}
 \dim_{\mcl H}^0 (\bdy\mcl B_0^\infty \cap B_r(0)) \geq c  ,\quad \forall r >  0. 
\eqe
From~\eqref{eqn-ball-infty-dim-as}, we immediately get that $\dim_{\mcl H}^0 \bdy\mcl B_0^\infty$ is a.s.\ equal to a deterministic constant $\Delta_{\op{ball}}^0$.
Furthermore, by combining~\eqref{eqn-ball-infty-dim-as} (applied with $c=\Delta_{\op{ball}}^0$) with~\eqref{eqn-ball-infty-conv} of Proposition~\ref{prop-ball-infty}, we get that there is a large $R > 0$ such that a.s.\ $\dim_{\mcl H}^0 \bdy\mcl B_{D_h(0,w)}(w;D_h) \geq \Delta_{\op{ball}}^0$ for each $w\in \BB C\setminus B_R(0)$. By the scale invariance of the law of $h$, modulo additive constant, and~\eqref{eqn-ball-infty-scale} we can remove the restriction that $w\in \BB C\setminus B_R(0)$. This gives the second statement of the proposition with $z =0$. The statement for a general $z\in\BB C$ follows from the translation invariance of the law of $h$ modulo additive constant. 

Let us now prove~\eqref{eqn-ball-infty-dim-as}. 
By the countable stability of Hausdorff dimension, for each $\delta \in (0,1)$ there exists $r_0 >0$ and $p\in (0,1)$ such that $\BB P\left[ \dim_{\mcl H}^0 (\bdy\mcl B_0^\infty \cap B_{r_0}(0)) \geq c -\delta \right]  \geq p$. 
By the scale invariance of the law of $\mcl B_0^\infty$ (see~\eqref{eqn-ball-infty-scale}), this implies that in fact
\eqb \label{eqn-ball-infty-dim-pos}
\BB P\left[\dim_{\mcl H}^0 (\bdy\mcl B_0^\infty \cap B_r(0)) \geq c -\delta \right] \geq p ,\quad\forall r > 0. 
\eqe

We now want to use a tail triviality argument to deduce~\eqref{eqn-ball-infty-dim-as} from~\eqref{eqn-ball-infty-dim-pos}.
To this end, let $A = A(1-p/2,\gamma)$ be as in Lemma~\ref{lem-conf-event} with $1-p/2$ in place of $p$.
For $r > 0$, let $E_r = E_r(0)$ and $Z_r = Z_r(0)$ be the point and event from that lemma, so that $E_r\in \sigma\left( h |_{B_{A^2 r}(0)} \right)$ and $\BB P[E_r] \geq 1-p/2$. 
Let
\eqb \label{eqn-ball-infty-event}
G_r := E_r\cap \left\{ \dim_{\mcl H}^0 (\bdy\mcl B_0^\infty \cap B_{A^2 r}(0)) \geq c -\delta \right\} .
\eqe
By~\eqref{eqn-ball-infty-dim-pos}, we have $\BB P[G_r] \geq p/2$. 

We now argue that $G_r\in \sigma\left( h|_{B_{ A^2 r}(0)} \right)$. 
Indeed, we recall from Lemma~\ref{lem-conf-event} that on $E_r$, every $D_h$-geodesic from a point of $B_r(0)$ to a point of $\BB C\setminus B_{A r}(0)$ passes through $Z_r$. 
Hence the proof of~\eqref{eqn-ball-infty-conv} of Proposition~\ref{prop-ball-infty} shows that if $E_r$ occurs, then
\eqb \label{eqn-ball-infty-dim-local}
\mcl B_0^\infty \cap B_r(0) = \mcl B_{D_h(0,Z_r)}\left( Z_r ;D_h   \right)  \cap B_r(0) .
\eqe
On the other hand, Lemma~\ref{lem-conf-event} shows that on $E_r$, every $D_h$-geodesic from $Z_r$ to a point of $B_r(0)$ stays in $B_{A^2 r}(0)$, which means that $D_h(0,Z_r) = D_h(0,Z_r ; B_{A^2 r}(0))$ and
\eqb \label{eqn-ball-infty-dim-internal}
\mcl B_{D_h(0,Z_r)}\left( Z_r ;D_h   \right)  \cap B_r(0) = \mcl B_{D_h(0,Z_r ; B_{A^2 r}(0))}\left( Z_r ;D_h(\cdot,\cdot; B_{A^2 r}(0) )   \right)  \cap B_r(0) .
\eqe
The right side of~\eqref{eqn-ball-infty-dim-internal} is $\sigma\left( h|_{B_{A^2 r}(0)} \right)$-measurable due to Axiom~\ref{item-metric-local}. 
By~\eqref{eqn-ball-infty-dim-local} and~\eqref{eqn-ball-infty-dim-internal}, we therefore obtain that $G_r\in  \sigma\left( h|_{B_{A^2 r}(0)} \right)$, as desired.

Since $\BB P[G_r] \geq p/2$, it holds with positive probability that there is a sequence $r_k\rta 0$ for which $G_{r_k}$ occurs.
Since $G_r\in  \sigma\left( h|_{B_{A^2 r}(0)} \right)$ and $\bigcap_{r>0}  \sigma\left( h|_{B_{A^2 r}(0)} \right)$ is trivial, the probability that such a sequence exists is equal to zero or one, so such a sequence must exist a.s. 
Recalling the definition~\eqref{eqn-ball-infty-event} of $G_r$, we now obtain~\eqref{eqn-ball-infty-dim-as}.
\end{proof}

We devote the remainder of this subsection to proving Theorem~\ref{thm-ball-bdy-bound} from Proposition~\ref{prop-ball-infty-dim}.  As in the above proof of Proposition~\ref{prop-ball-infty-dim}, we will prove just the result for Euclidean dimension, since the proof for $\gamma$-quantum dimension is essentially the same.

To deduce Theorem~\ref{thm-ball-bdy-bound} from Proposition~\ref{prop-ball-infty-dim}, it suffices to show that $\Delta_{\op{ball}}^0 \geq 2 - \xi Q + \xi^2/2$ and a.s.\ $\dim_{\mcl H}^0 \bdy \mcl B_{D_h(0,z)}  \leq  2 - \xi Q + \xi^2/2$.  We prove the second inequality in Lemma~\ref{lem-bdy-dim-compare2revised} below.  First we turn to the first inequality, which we state as a proposition.

\begin{prop}
\label{prop-ball-bdy-bound} $\Delta_{\op{ball}}^0 \geq 2 - \xi Q + \xi^2/2$.
\end{prop}

We will extract Proposition~\ref{prop-ball-bdy-bound} from~\cite[Theorem 1.1]{gwynne-ball-bdy}, which says the following.
Suppose we fix $t>0$ and consider the ball $\mcl B_t = \mcl B_t(0;D_h)$.  
Then the essential supremum of the law of the random variable $\dim_{\mcl H}^0 (\bdy\mcl B_t)$ is equal to $2-\xi Q + \xi^2/2$.   

Roughly speaking, we will deduce Proposition~\ref{prop-ball-bdy-bound} from this statement as follows.
For a ``typical" time $t >0$ and a ``typical" point $z\in \bdy\mcl B_t$, we expect that near $z$, $\bdy\mcl B_t$ locally looks like a segment of the boundary of the ball $ \mcl B_0^\infty$ started from $\infty$ and stopped upon hitting $0$, as constructed in Proposition~\ref{prop-ball-infty}. 
In particular, for such a point $z$ and a small enough $\ep  > 0$ it should be that $\dim_{\mcl H}^0(B_\ep(z) \cap \bdy\mcl B_t) \leq \Delta_{\op{ball}}^0$. 
Say that a point which satisfies this condition for some $\ep > 0$ is ``normal". 

Using Theorem~\ref{thm-gen-upper}, we can show that, for any $t>0$, the Hausdorff dimension of the complementary set of ``abnormal" points $z\in \bdy\mcl B_t$ is a.s.\ at most $2-\xi Q - \xi^2/2 - q$ for some $q = q(\gamma)  >0$. 
On the other hand, for Lebesgue-a.e. $t$, the set of ``normal" points in $\bdy\mcl B_t$ has Hausdorff dimension at most $\Delta_{\op{ball}}^0$.  
From this, we deduce that, for such a choice of $t$, $\dim_{\mcl H}^0 (\bdy\mcl B_t) \leq \max\left\{2-\xi Q + \xi^2/2 - q , \Delta_{\op{ball}}^0 \right\}$ almost surely. 
But,~\cite[Theorem 1.1]{gwynne-ball-bdy} tells us that the essential supremum of the law of the random variable $\dim_{\mcl H}^0 (\bdy\mcl B_t)$ is equal to $2-\xi Q + \xi^2/2$. 
Hence we must have $\Delta_{\op{ball}}^0 \geq 2-\xi Q + \xi^2/2$. 

To implement this strategy, we will apply Theorem~\ref{thm-gen-upper} via Lemma~\ref{lem-check-conditions}. 
Let $A = A(1/2,\gamma) > 0$ be as in the statement of Lemma~\ref{lem-conf-event} with $p = 1/2$. 
We define the ``bad" events $F_\ep(z)$ of Theorem~\ref{thm-gen-upper} as in Lemma~\ref{lem-check-conditions}, with $a=1/A$ and $b=A^2$ and the ``good'' events $G_r(z)$ taken to be the confluence events  $E_r(z)$ of Lemma~\ref{lem-conf-event} with $p=1/2$. 
We also let $\mcl Y_t$ be the set of ``bad" points as in Theorem~\ref{thm-gen-upper} with this choice of $F_\ep(z)$. 

We will motivate the choice of events $G_r(z)$ in a moment.  First, with $\mcl Y_t$ as in Theorem~\ref{thm-gen-upper}, we obtain the following upper bound on the Hausdorff dimension of $\mcl Y_t$.

\begin{lem}\label{lem-dim-exceptional-points}
There exists $q = q(\gamma) >0$ such that a.s.\ $\dim_{\mcl H}^0 \mcl Y_t \leq 2 - \xi Q + \xi^2/2 - q$.
\end{lem}

\begin{proof}
The result follows from applying Theorem~\ref{thm-gen-upper} via Lemma~\ref{lem-check-conditions}; we just need to check that, with our above definitions, the events $G_r(z)$ satisfy the two conditions of  Lemma~\ref{lem-check-conditions}. 
Indeed, Lemma~\ref{lem-conf-event} implies that the event $G_r(z)$ is a.s. determined by $h|_{B_{A^2 r}(z) \backslash B_{r/A}(z)}$ viewed modulo additive constant, and that $\BB{P}[G_r(z)] \geq 1/2$ for each $z \in \BB{C}$ and $r>0$. 
\end{proof}

The remaining ingredient we need to prove Proposition~\ref{prop-ball-bdy-bound} is to show that $\partial \mcl B_t \backslash \mcl Y_t$ has Hausdorff dimension at most $\Delta_{\op{ball}}^0$ almost surely.  This is where the particular definition of the ``good'' events $G_r(z)$---as the confluence events $E_r(z)$ of Lemma~\ref{lem-conf-event}---plays a crucial role.  Specifically, we will use the following two properties that hold on the event $G_r(z)$ by Lemma~\ref{lem-conf-event}:
\begin{enumerate}[(a)]
\item
Each $D_h$-geodesic from a point of $B_r(z)$ to a point of $\BB C\setminus B_{A r}(z)$ passes through the single (random) point $Z_r(z) \in B_{A r}(z) \setminus B_r(z)$.
\label{item-confluence-g}
\item Each $D_h$-geodesic between points of $\BB C \setminus B_{r}(z)$ is contained in $\BB C \setminus \ol{B_{r/A}(z)}$. More strongly, there is a path $\pi \subset B_r(z) \setminus \ol{B_{r/A}(z)}$ such that 
\eqb \label{eqn-avoids-path-g}
\left(\text{$D_h$-length of $\pi$} \right) < D_h\left(\pi , B_{r/A}(z) \right) .
\eqe 
\label{item-avoids-g}
\end{enumerate}
Here as above, $A > 0$ is as in Lemma~\ref{lem-conf-event} with $p = 1/2$. 

We first prove a lemma that states that, if two ``good'' events $G_r(z)$ and $G_{r'}(z')$ corresponding to nested annuli occur, then we can compare the ball started from infinity and run until it hits $z$ (see Proposition~\ref{prop-ball-infty}) to the ball centered at a point in $\BB{C}$ and run until a specified time.  This comparison is useful because we already know from Proposition~\ref{prop-ball-infty-dim} that the boundary of the ball started from infinity has dimension $\Delta_{\op{ball}}^0$ almost surely.

\begin{lem}
\label{lem-bdy-dim-compare}
If $z,z' \in \BB{C}$ and $r' > r > 0$ are such that $B_{r}(z) \subset B_{r'}(z')$, $B_{r/A}(z) \subset B_{r'/A}(z')$, and $0 \notin B_{A r'}(z')$, then a.s.\ for Lebesgue-a.e.\ $t > D_h(0, B_{r/A}(z))$, the inequality
\[
\dim_{\mcl H}^0(\bdy \mcl{B}_t \cap [B_{r'}(z') \setminus B_r(z)] )   \leq  \Delta_{\op{ball}}^0
\]
holds on the event  $G_r(z) \cap G_{r'}(z')$. 
\end{lem}

Before presenting the proof of the lemma, we sketch the main steps of the argument.
Due to the definition of the metric ball $\mcl B_z^\infty$ started from $\infty$ and stopped upon hitting $z$, property~\eqref{item-confluence-g} of $G_{r'}(z')$ directly allows us to show that on $G_{r'}(z')$, we have $\mcl B_z^\infty \cap B_r(z) = \mcl B_{D_h(0,z)}(0;D_h) \cap B_r(z)$. 
The trickier part of the argument is applying property~\eqref{item-avoids-g} of $G_{r}(z)$ to compare the LQG balls of radii $D_h(0,z)$ and $t$ centered at $0$. To do this, we start the proof by replacing our field $h$ with a field $h^*$ which is equal in distribution to $h$ modulo additive constant, defined so that $ h^* - h $ is a random multiple of a bump function which is supported on $B_{r/A}(z)$. Since $h^* = h$ outside of $B_{r/A}(z)$, property~\eqref{item-avoids-g} of $G_{r}(z)$ says that geodesics between points outside the larger ball $B_{r}(z)$ are the same for both metrics $D_h$ and $D_{h^*}$.  On the other hand, the conditional law given $h$ of the random variable $D_{h^*}(0,z)$ is absolutely continuous with respect to Lebesgue measure on the set of times $t >  D_h(0, B_{r/A}(z))$. So, we can translate an almost sure statement about a $D_h$-ball with radius $D_{h^*}(0,z)$ to a statement about a $D_h$-ball of radius $t  >  D_h(0, B_{r/A}(z))$.

\begin{figure}[t!]
\begin{center}
\includegraphics[scale=.7]{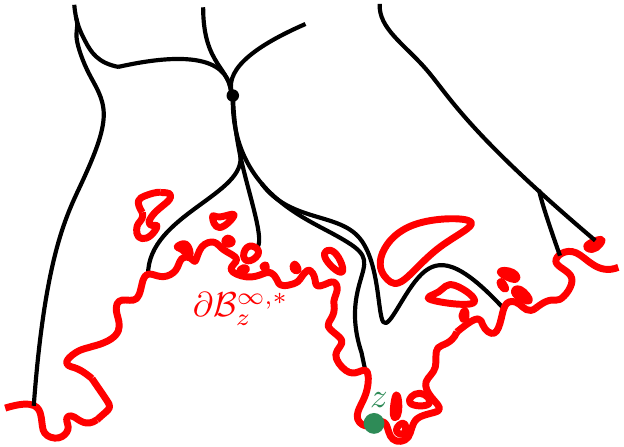}  \hspace{10pt}
\includegraphics[scale=.7]{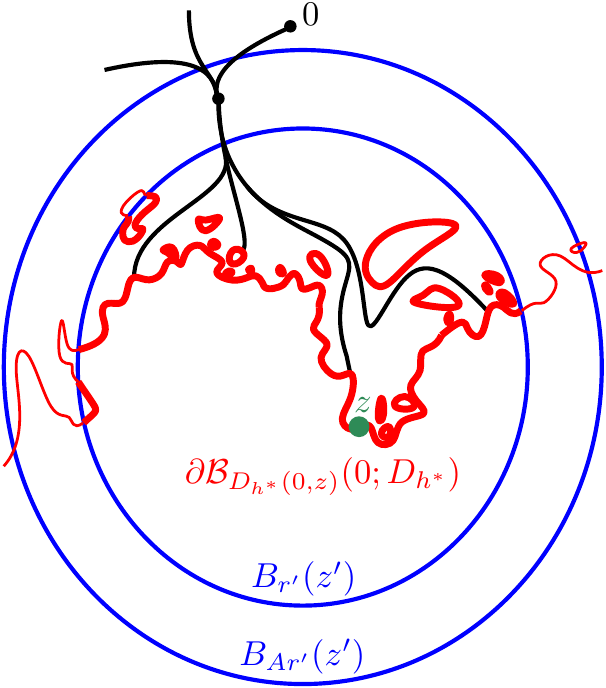}  \hspace{10pt}
\includegraphics[scale=.7]{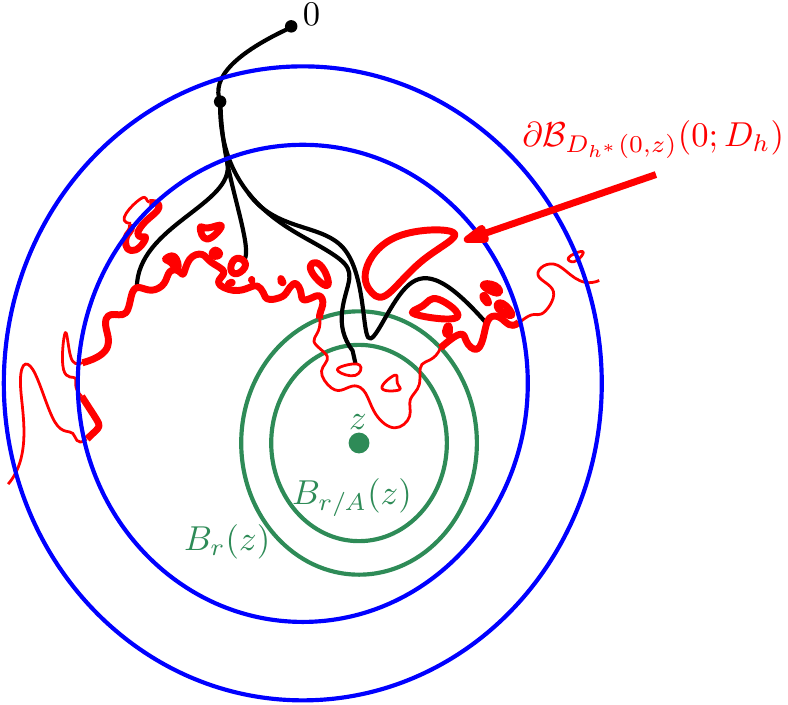} 
\end{center}
\caption{An illustration of the proof of Lemma~\ref{lem-bdy-dim-compare}.  In each figure, the bold part of the red set has Euclidean dimension at most $\Delta_{\op{ball}}^0$ almost surely on the event we are considering. 
\textbf{Left}:  We start with the $D_{h^*}$-metric ball $\mcl B_z^{\infty,*}$ started from $\infty$ and run until hitting $z$; the red set represents its boundary, and the black curves are geodesic rays.  
\textbf{Middle}:  The red set now represents $\bdy \mcl B_{D_{h^*}(0,z)}(0;D_{h^*})$. By property~\eqref{item-confluence-g} of $G_{r'}(z')$, on $G_{r'}(z')$ we have $\bdy\mcl B_z^{\infty,*} \cap B_{r'}(z')= \bdy \mcl B_{D_{h^*}(0,z)}(0;D_{h^*}) \cap B_{r'}(z')$.  
\textbf{Right}: The red curve now represents $\bdy\mcl B_{D_{h^*}(0,z)}(0;D_h)$. By property~\eqref{item-avoids-g} of $G_{r}(z)$, on $G_r(z)$ we have $\bdy \mcl B_{D_{h^*}(0,z)}(0;D_{h }) \setminus B_{r }(z ) = \bdy \mcl B_{D_{h^*}(0,z)}(0;D_{h^*}) \setminus B_{r }(z )$.  The rest of the argument consists of showing the conditional law of $D_{h^*}(0,z)$ given $h$ is mutually absolutely continuous w.r.t.\ Lebesgue measure on its support.
}\label{fig-bdy-dim-compare}
\end{figure}

\begin{proof}[Proof of Lemma~\ref{lem-bdy-dim-compare}]
Let $\phi : \BB C\rta [0,\infty)$ be a smooth bump function which is positive on $B_{r/A}(z)$, is zero outside $B_{r/A}(z)$, and is normalized so that the Dirichlet energy of $\phi$ satisfies $(\phi,\phi)_\nabla =1$. 
Let $Z$ be a standard Gaussian random variable which is independent from $h$. 
Consider the field
\eqb
h^* := h - (h,\phi)_\nabla \phi + Z \phi
\eqe
where $(h,\phi)_\nabla$ is the Dirichlet inner product. 
By the definition of the whole-plane GFF as a sum of i.i.d.\ standard Gaussians times the elements of an orthonormal basis for the Dirichlet inner product, $h^*$ and $h$ are equal in distribution modulo additive constant. 

Let $\mcl B_z^{\infty, *}$ be defined in the same manner as the ball $\mcl B_z^\infty$ of Proposition~\ref{prop-ball-infty} but with $h^*$ in place of $h$.
Since $B_z^{\infty,*}$ is determined by $h^*$ viewed modulo additive constant, we can apply  Proposition~\ref{prop-ball-infty-dim} with $h^*$ in place of $h$ to get that a.s.\ $\dim_{\mcl H}^0 \bdy\mcl B_z^{\infty,*} = \Delta_{\op{ball}}^0$.  We convert this to a statement about the dimension of $\mcl B_t(0;D_h) \cap B_r(z)$ in two stages. 
\begin{enumerate}
\item
Since $B_{r/A}(z) \subset B_{r'/A}(z')$, each of $G_r(z)$ and $G_{r'}(z')$ is determined by $h|_{\BB C\setminus B_{r /A}(z )}$, viewed modulo additive constant. 
So, the definitions of these events are unaffected by replacing $h$ by $h^*$.
By Property~\eqref{item-confluence-g} for $G_{r'}(z')$, on $G_{r'}(z')$ every $D_{h^*}$-geodesic from a point outside $B_{A r'}(z')$ to a point inside $B_{r'}(z')$ passes through the single point $Z_{r'}(z')$.  This implies that $\bdy\mcl B_{D_{h^*}(0,z)}(0;D_{h^*}) \cap B_{r'}(z') = \bdy\mcl B_z^{\infty,*} \cap B_{r'}(z')$. 
\item By Property~\eqref{item-avoids-g} of $G_r(z)$, on $G_r(z)$ every $D_h$-geodesic between two points outside of $B_{r}(z)$ is contained in $\BB C\setminus B_{r/A}(z)$. Since $h^* = h$ outside of $B_{r/A}(z)$, we can use the stronger statement~\eqref{eqn-avoids-path-g} to get that also every $D_{h^*}$-geodesic between two points outside of $B_r(z)$ is contained in $\BB C\setminus B_{r/A}(z)$. Since $h^* = h$ outside of $B_{r/A}(z)$, this implies that $\bdy \mcl B_{D_{h^*}(0,z)}(0;D_{h^*})\setminus B_{r }(z) = \bdy\mcl B_{D_{h^*}(0,z)}(0;D_h) \setminus B_{r }(z)$. 
\end{enumerate}
Thus, on the event $G_{r'}(z') \cap G_{r}(z)$, a.s.\ $\partial \mcl B_{D_{h^*}(0,z)}(0;D_h) \cap [B_{r'}(z') \setminus B_r(z)] = \bdy\mcl B_z^{\infty,*} \cap  [B_{r'}(z') \setminus B_r(z)]$. Hence on this event a.s.\ 
\eqb \label{eqn-ball-cap-mixed}
\dim_{\mcl H}^0\left( \partial \mcl B_{D_{h^*}(0,z)}(0;D_h) \cap [B_{r'}(z') \setminus B_r(z)]  \right) \leq \Delta_{\op{ball}}^0 .  
\eqe

For the rest of the proof we assume that the event $G_{r'}(z') \cap G_{r}(z)$ occurs. 
Let
\eqbn
\Lambda = \Lambda(h) := \left\{ t\in \BB R : \dim_{\mcl H}^0\left( \partial \mcl B_t(0;D_h) \cap [B_{r'}(z') \setminus B_r(z)]  \right) > \Delta_{\op{ball}}^0 \right\}. 
\eqen
By~\eqref{eqn-ball-cap-mixed}, on the $\sigma(h)$-measurable event $G_{r'}(z') \cap G_r(z)$, a.s.
\eqb \label{eqn-lambda-cond}
\BB P[ D_{h^*}(0,z) \in \Lambda \,|\, h ] = 0 .
\eqe
To prove the lemma, we need to show that a.s.\ the Lebesgue measure of $\Lambda \cap (D_h(0,B_{r/A}(z)),\infty)$ is zero.  

For $x\in\BB R$, let $f(x) := D_{h - (h,\phi)_\nabla \phi + x \phi}(0,z)$. 
By~\eqref{eqn-lambda-cond} and since we are assuming that $G_{r'}(z')\cap G_r(z)$ occurs, a.s.\ 
\eqbn
0 = \BB P[ D_{h^*}(0,z) \in \Lambda \,|\, h ]  = \BB P[f(Z) \in \Lambda \,|\, h ] = P[Z \in f^{-1}(\Lambda)\,|\, h ] .
\eqen
Since $Z$ is a standard Gaussian random variable independent from $h$, its conditional law given $h$ is mutually absolutely continuous w.r.t.\ Lebesgue measure on $\BB R$. So, a.s.\ the set $f^{-1}(\Lambda)$ has zero Lebesgue measure. Since $\phi$ is positive on $B_{r/A}(z)$ and zero outside $B_{r/A}(z)$, we deduce from the Weyl scaling property of the metric (Axiom~\ref{item-metric-f}) that a.s.\ for every $x<y$, 
\eqb
0 \leq f(y) - f(x) \leq ( e^{\xi(y-x) \|\phi\|_\infty } - 1) f(x) \leq ( e^{\xi(y-x) \|\phi\|_\infty } - 1) e^{\xi x \|\phi\|_\infty } D_{h - (h,\phi)_\nabla\phi}(0,z) ,
\eqe
where $\|\phi\|_\infty $ is the $L^\infty$ norm.
We deduce that a.s.\  $f$ is locally Lipschitz, and so a.s.\ $f(f^{-1}(\Lambda))$ has Lebesgue measure zero. Moreover, again by Axiom~\ref{item-metric-f}, a.s.\ $f(\BB R) = (D_h(0,B_{r/A}(z)),\infty)$, and therefore $f(f^{-1}(\Lambda)) = \Lambda \cap (D_h(0,B_{r/A}(z)),\infty)$. Thus, a.s.\ the Lebesgue measure of $\Lambda \cap (D_h(0,B_{r/A}(z)),\infty)$ is zero, as desired.  
\end{proof}

We now apply Lemma~\ref{lem-bdy-dim-compare} to prove the desired dimension result for $\partial \mcl B_t \backslash \mcl Y_t$.

\begin{lem}
\label{lem-dim-nonexceptional-points}
It holds almost surely that, for Lebesgue-a.e.\ $t>0$, the set $\partial\mcl B_t \setminus \mcl Y_t$ has Hausdorff dimension at most $\Delta_{\op{ball}}^0$. 
\end{lem}

\begin{proof}
Suppose that $w \in \partial \mcl B_t \backslash \mcl Y_t$ for some $t>0$. 
By the definition~\eqref{eqn-gen-upper-set} of $\mcl Y_t$, there exists a sequence of positive rational numbers $r_n \rta 0$ and a sequence of points $z_n \in B_{r_n}(w)  \cap \BB Q^2 $ such that, for each $n$, the event $\left[F_{ r_n}(z_n)\right]^c$ occurs. Recalling the definition~\eqref{eqn-defn-F-ep-z} of $F_\ep(z)$ (with $a$ and $b$ in~\eqref{eqn-defn-F-ep-z} equal to $1/A$ and $A^2$, respectively), we deduce that, for each $n$, the event $G_{r_n^*}(z_n)$ occurs for some $r_n^* \in  [A  r_n^{1/2}, 2r_n^{1/4}/A^2] \cap \BB Q $.   
Observe that the balls $B_{r_n^*}(z_n)$ contain the point $w$ for all sufficiently large $n$ and that $r_n\rta 0$ as $n\rta\infty$.
In other words, we have a sequence of balls $B_r(z)$ of arbitrarily small radii, all containing $w$, for which the corresponding events $G_r(z)$ occur. Since $t>0$ and $w \in \partial \mcl B_t \backslash \mcl Y_t$ were chosen arbitrarily, we deduce that, for each positive integer $k$, we can cover the set 
\eqb
\bigcup_{t > 0} (\partial \mcl B_t \backslash \mcl Y_t)
\label{eqn-union-complements}
\eqe
by a collection $\mathfrak{B}_k$ of balls $B_r(z)$, with rational centers and rational radii which are at most $1/k$, for which the corresponding events $G_r(z)$ occur.   The union $\mathfrak{B} = \bigcup_k \mathfrak{B}_k$ of these collections is countable since the balls all have rational centers and radii.  By Lemma~\ref{lem-bdy-dim-compare}, the following holds almost surely: for almost every $t>0$, 
\[
\dim_{\mcl H} ( \partial \mcl B_t \cap [B_{r'}(z') \backslash B_{r}(z) ]) \leq \Delta_{\op{ball}}^0
\] 
for any nested pair of balls $B_r(z) \subset B_{r'}(z')$ in $\mathfrak{B}$ for which $B_{r/A}(z) \subset B_{r'/A}(z')$.  Since every point in the set~\eqref{eqn-union-complements} is contained in a sequence of balls in $\mathfrak{B}$ of arbitrarily small radii, we deduce that a.s., it is the case that, for almost every $t$ and any point $w \in \bdy \mcl B_t \backslash \mcl Y_t$, the intersection of $\partial \mcl B_t$ with some element in $\mathfrak{B}$ containing $w$ has Hausdorff dimension at most $\Delta_{\op{ball}}^0$.  By the countable stability of Hausdorff dimension, we conclude that a.s.,  the Hausdorff dimension of the entire set $\partial\mcl B_t \setminus \mcl Y_t$ is bounded from above by $\Delta_{\op{ball}}^0$ for almost every $t>0$. 
\end{proof}

We now combine Lemmas~\ref{lem-dim-exceptional-points} and~\ref{lem-dim-nonexceptional-points} to prove Proposition~\ref{prop-ball-bdy-bound}, implementing the strategy we outlined above.

\begin{proof}[Proof of Proposition~\ref{prop-ball-bdy-bound}]
By Lemma~\ref{lem-dim-exceptional-points}, there exists $q>0$ such that for each $t>0$, the Hausdorff dimension of the set $\mcl Y_t$ is almost surely bounded from above by $2 - \xi Q + \xi^2/2 - q$.  On the other hand, Lemma~\ref{lem-dim-nonexceptional-points} asserts that it is a.s. the case that, for almost every $t>0$, the set $\partial \mcl B_t \backslash \mcl Y_t$ a.s. has Hausdorff dimension at most $\Delta_{\op{ball}}^0$. Therefore, for such a choice of $t$, we have a.s.\ $\dim_{\mcl H}^0 \bdy\mcl B_t \leq \max\{\Delta_{\op{ball}}^0, 2-\xi Q + \xi^2/2 - q\}$. By~\cite[Theorem 1.1]{gwynne-ball-bdy}, if we sample $t$ uniformly at random from $[0,1]$, independently from $h$, then for every $\ep > 0$ the Hausdorff dimension of $\partial \mcl B_t$ is at least $2 - \xi Q + \xi^2/2 -\ep $ with positive probability.
Therefore, on an event with positive probability,
\[
2-\xi Q + \xi^2/2 -\ep \leq \dim_{\mcl H}^0 \bdy\mcl B_t \leq \max\{\Delta_{\op{ball}}^0, 2-\xi Q + \xi^2/2 - q\} .
\]
Since $\ep > 0$ is arbitrary, this implies that $\Delta_{\op{ball}}^0 \geq 2 - \xi Q + \xi^2/2$, as desired.
\end{proof}

We now complete the proof of Theorem~\ref{thm-ball-bdy-bound} by proving the following lemma, which we alluded to at the beginning of this subsection.

\begin{lem}
\label{lem-bdy-dim-compare2revised}
For each fixed $z \in \BB{C}$, a.s.\ 
\[
\dim_{\mcl H}^0 \bdy \mcl B_{D_h(0,z)}  \leq  2 - \xi Q + \xi^2/2 .
\] 
\end{lem}

\begin{proof}
To prove the lemma, we will show that for each $r, \ep > 0$,  
\eqb
\BB P\left[ \dim_{\mcl H}^0(\bdy \mcl B_{D_h(0,z)} \setminus B_r(z) ) \leq  2 - \xi Q + \xi^2/2  \right] \geq 1-\ep .
\label{eqn-implies-lemma}
\eqe
The lemma then follows by sending $r,\ep \rta 0$.

By Lemma~\ref{lem-conf-event}, we can choose $\wt A = \wt A(\ep) >1$ such that, with probability $1-\ep$, each $D_h$-geodesic between points of $\BB{C} \backslash B_r(z)$ is contained in $\BB{C} \backslash \overline{B_{r/\wt A}(z)}$ and more strongly there is a path $\pi$ in $B_r(z) \setminus \ol{B_{r/\wt A}(z)}$ such that 
\eqbn
\left(\text{$D_h$-length of $\pi$} \right) < D_h\left(\pi , B_{r/\wt A}(z) \right) .
\eqen 
(We have introduced the tilde to avoid confusing this $\wt A$ with the constant $A$ we have been referencing throughout this subsection.)
Let $E_\ep$ be the event that this is the case. We henceforth work on the event $E_\ep$. 

As in the proof of Lemma~\ref{lem-bdy-dim-compare}, let $\phi : \BB C\rta [0,\infty)$ be a smooth bump function which is positive on $B_{r/\wt A}(z)$, is zero outside $B_{r/\wt A}(z)$, and is normalized so that the Dirichlet energy $(\phi,\phi)_\nabla$ is 1.  Here, we also stipulate that $\phi$ is identically equal to some constant $c$ on the annulus $B_{r/(2 \wt A)}(z) \setminus B_{r/(3\wt A)}(z)$.  As before, we let $Z$ be a standard Gaussian random variable which is independent from $h$, and we consider the field
\eqb
h^* := h - (h,\phi)_\nabla \phi + Z \phi
\eqe
By the definition of the whole-plane GFF, $h^*$ and $h$ are equal in distribution modulo additive constant.  

Define 
\eqbn
\Lambda := \left\{t > 0 : \dim_{\mcl H}^0(\bdy \mcl B_t(0;D_{h}) \setminus B_r(z) ) > 2 - \xi Q + \xi^2/2 \right\} .
\eqen
By~\cite[Theorem 1.1]{gwynne-ball-bdy}, the set $\Lambda$ a.s.\ has Lebesgue measure zero.  

For $x \in \BB R$, let $f(x) := D_{h - (h,\phi)_\nabla \phi + x \phi}(0,z)$, and note that $D_{h^*}(0,z) = f(Z)$. By~\cite[Lemma 2.5]{gwynne-ball-bdy} (with the field $h$ in that lemma replaced by $h - (h,\phi)_\nabla \phi$ and $\phi$ defined to equal $c$ on the annulus $\mcl A$ in the lemma instead of $1$), it is a.s.\ the case that for every $x <y$,
\[
f(y) - f(x) \geq C_h (y-x)  e^{\xi c x},
\]
where $C_h > 0$ is measurable with respect to $\sigma(h)$.  We deduce that a.s.\ (w.r.t.\ the law of $h$) the set $f^{-1}(\Lambda)$ has Lebesgue measure zero.
Since the conditional law of $Z$ given $h$ is mutually absolutely continuous w.r.t.\ Lebesgue measure, $D_{h^*}(0,z) \notin \Lambda$ almost surely.  In other words, 
\eqb
\dim_{\mcl H}^0\left(\bdy \mcl B_{D_{h^*}(0,z)}(0;D_{h}) \setminus B_r(z) \right) \leq 2 - \xi Q + \xi^2/2 	\qquad \text{a.s.}
\label{eqn-dim-hstar-bound}
\eqe
Since we are working on the event $E_{\ep}$ and $h|_{\BB C\setminus B_{r/\wt A}(z)} = h^*|_{\BB C\setminus B_{r/\wt A}(z)}$, the same argument preceding~\eqref{eqn-ball-cap-mixed} in the proof of Lemma~\ref{lem-bdy-dim-compare} gives
\eqbn
\bdy\mcl B_{D_{h^*}(0,z)}(0;D_{h}) \setminus B_r(z) = \bdy\mcl B_{D_{h^*}(0,z)}(0;D_{h^*}) \setminus B_r(z) .
\eqen
Hence~\eqref{eqn-dim-hstar-bound} implies that
\[
\dim_{\mcl H}^0\left(\bdy \mcl B_{D_{h^*}(0,z)}(0;D_{h^*}) \setminus B_r(z) \right) \leq 2 - \xi Q + \xi^2/2 	\qquad \text{a.s.\ on $E_\ep$}
\]
Since $h^*\eqD h$ and $\BB P[E_\ep] \geq 1-\ep$, this proves~\eqref{eqn-implies-lemma}, and hence the lemma.
\end{proof}

\begin{proof}[Proof of Theorem~\ref{thm-ball-bdy-bound}]
Combining the results of Proposition~\ref{prop-ball-infty-dim}, Lemma~\ref{lem-bdy-dim-compare2revised}, and Proposition~\ref{prop-ball-bdy-bound} yields that for each fixed $z\in\BB C$, a.s.\ 
\[
\Delta_{\op{ball}}^0 \leq \dim_{\mcl H}^0 \bdy\mcl B_{D_h(0,z)}  \leq 2 - \xi Q + \xi^2/2 \leq \Delta_{\op{ball}}^0 .
\] 
\end{proof}

\subsection{The exterior boundary of a metric ball}
\label{sec-outer}

We now study the exterior boundary of an LQG metric ball, which we defined in Definition~\ref{def-outer}.
This random fractal satisfies a zero-one law analogous to the zero-one law for LQG metric ball boundaries that we stated in  Theorem~\ref{prop-ball-infty-dim}. Consider the infinite-volume ``metric ball'' $\mcl B_0^\infty$ we defined in Proposition~\ref{prop-ball-infty}, and define its exterior boundary $\mcl O_0^\infty$ in a manner analogous to Definition~\ref{def-outer}---i.e., as the union of the boundaries of the connected components of $\BB{C} \backslash \mcl B_0^\infty$.

\begin{prop} \label{prop-ball-infty-dim-out} 
There are deterministic constants $\Delta_{\op{out}}^0,  \Delta_{\op{out}}^\gamma > 0$ such that a.s.\ $\dim_{\mcl H}^0 \bdy\mcl O_0^\infty = \Delta_{\op{out}}^0$ and $\dim_{\mcl H}^\gamma \bdy\mcl O_0^\infty = \Delta_{\op{out}}^\gamma$.
Furthermore, for each fixed $z \in\BB C$ a.s.\ $\dim_{\mcl H}^0 \mcl O_z(w;D_h) \geq \Delta_{\op{out}}^0$ and $\dim_{\mcl H}^\gamma \mcl O_z(w;D_h) \geq \Delta_{\op{out}}^\gamma$ simultaneously for each $w\in\BB C$. 
\end{prop}
\begin{proof}
This follows from exactly the same argument used in the proof of Proposition~\ref{prop-ball-infty-dim}.
\end{proof}

The set $\BB C\setminus \ol{\mcl B_0^\infty}$ has at most countably many connected components (since each component contains a point of $\BB Q^2$). By the countable stability of Hausdorff dimension $\dim_{\mcl H}^0 \mcl O_0^\infty$ (resp. $\dim_{\mcl H}^\gamma \mcl O_0^\infty$) is a.s.\ equal to the supremum of the Euclidean (resp. $\gamma$-quantum) dimensions of the boundaries of the connected components of $\BB C\setminus \ol{\mcl B_0^\infty}$. 
We expect that a.s.\ the  boundary of each of these connected components have Euclidean dimension $\Delta_{\op{out}}^0$ and $\gamma$-quantum dimension $\Delta_{\op{out}}^\gamma$, but we do not prove this here.

As we described in Section~\ref{sec-intro}, the points in $\bdy \mcl B_s \setminus \mcl O_s$ are the points which are not on the boundary of any connected component of $\BB C\setminus \mcl B_{s}$, which can arise as accumulation points of connected components of $\BB C\setminus \mcl B_s$ with arbitrarily small diameters. The rest of this subsection is devoted to proving Theorem~\ref{thm-outer-bdy-compare}, which asserts that, at least with positive probability, the Euclidean and $\gamma$-quantum dimensions of $\mcl O_s$ are strictly smaller than those of $\bdy\mcl B_s$.  See Figure~\ref{fig-outer-bdy-compare} for an illustration and outline of the proof of Theorem~\ref{thm-outer-bdy-compare}.

\begin{figure}[t!]
 \begin{center}
\includegraphics[scale=1]{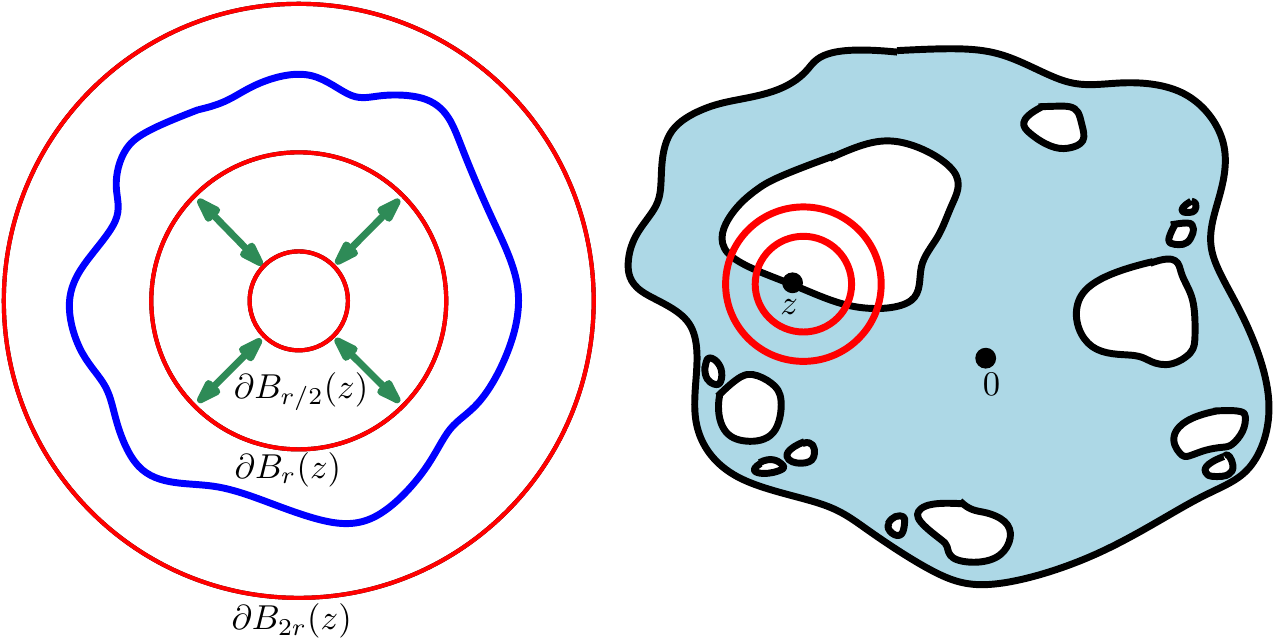}
\vspace{-0.01\textheight}
\caption{Illustration of the main ideas of the proof of Theorem~\ref{thm-outer-bdy-compare}. \textbf{Left.} We define $G_r(z)$ to be the event that there is a path in $B_{2 r}(z) \setminus B_r(z) $ which disconnects its inner and exterior boundaries (blue) whose $D_h$-length is less than $D_h(\bdy B_r(z) , \bdy B_{r/2}(z))$. Then $G_r(z)$ satisfies the two conditions in Lemma~\ref{lem-check-conditions} for a ``good'' event.  This means that, if we define the event $F_\ep(z)$ as in the statement of that lemma, $F_\ep(z)$ satisfies the conditions of the generalized upper bound (Theorem~\ref{thm-gen-upper}).
\textbf{Right.} Theorem~\ref{thm-gen-upper} allows us to reduce our task of proving Theorem~\ref{thm-outer-bdy-compare} to showing that, if $z\in \mcl O_s$, then the event $F_\ep(w)$ occurs for every small enough $\ep > 0$ and every $w \in B_\ep(z) \cap \BB Q^2$.  This is the case because a point on the exterior boundary $\mcl O_s$ has the following property: for every sufficiently small Euclidean annulus $A$ whose inner disk contains $z$, the $D_h$-distance from the inner boundary of $A$ to $z$ must be shorter than the minimal $D_h$-length of the paths in $A$ which disconnect its inner and exterior boundaries. Otherwise, $z$ would not be on the exterior boundary, since the metric ball growth $\{\mcl B_s\}_{s\geq 0}$ would form arbitrarily small ``bubbles" containing $z$ before reaching $z$. 
}\label{fig-outer-bdy-compare}
\end{center}
\vspace{-1em}
\end{figure} 

The proof is based on the generalized upper bound in Theorem~\ref{thm-gen-upper}.  We will apply Theorem~\ref{thm-gen-upper} for the events $F_\ep(z)$ we constructed in Lemma~\ref{lem-check-conditions}, with $a = 1/2$ and $b=1$, and with the ``good'' events $G_r(z)$ defined as   
\[
G_r(z) := \begin{tabular}{l}
 \text{the event that there is a path in the annulus $B_{2r}(z) \setminus B_r(z)$}\\
  \text{which disconnects the inner and outer boundaries of this annulus}\\
  \text{and whose $D_h$-length is shorter than $D_h(\bdy B_{r/2}(z) , \bdy B_r(z))$}
\end{tabular}
\]

To apply Theorem~\ref{thm-gen-upper} via Lemma~\ref{lem-check-conditions}, we need to check that $G_r(z)$ satisfies the conditions of the lemma.

\begin{lem}
\label{lem-outer-good-event-check}
The events $G_r(z)$ satisfy the conditions of Lemma~\ref{lem-check-conditions}.
\end{lem}

\begin{proof}
First, by locality and Weyl scaling (Axioms~\ref{item-metric-local} and~\ref{item-metric-f}), $G_r(z)$ is a.s.\ determined by $h|_{B_{2r}(z) \setminus B_{r/2}(z) }$ viewed modulo additive constant.

Second, by Axioms~\ref{item-metric-f} and~\ref{item-metric-coord} and the scale invariance of the law of $h$ modulo additive constant, $\BB P[G_r(z)] $ does not depend on $r$ or $z$.
Moreover, it is easy to check that $\BB P[G_r(z)] > 0$ for each fixed choice of $r$ and $z$ (see, e.g.,~\cite[Lemma 6.1]{gwynne-ball-bdy}). 
\end{proof}

\begin{proof}[Proof of Theorem~\ref{thm-outer-bdy-compare}]
We will prove the dimension upper bounds for the exterior boundary of a metric balls of a fixed radius.  One can then use exactly the same argument we used in our proof of Lemma~\ref{lem-bdy-dim-compare2revised} to deduce the result for the exterior boundary of a metric ball stopped when it hits a fixed point. 

By combining Lemmas~\ref{lem-outer-good-event-check} and~\ref{lem-check-conditions}, we deduce that the hypotheses of Theorem~\ref{thm-gen-upper} are satisfied for the events $F_\ep(z)$ of~\eqref{eqn-defn-F-ep-z} with the above choice of $G_r(z)$. 
Define $\mcl Y_{s}$ for $s>0$ as in Theorem~\ref{thm-gen-upper} for the above choice of $F_\ep(z)$. 
Then $\mcl Y_{s}$ satisfies the dimension upper bounds of Theorem~\ref{thm-gen-upper}.
These are exactly the bounds we want to prove for  $\mcl O_{s}$.  We will prove these bounds for $\mcl O_{s}$ by showing that $\mcl O_{s} \subset \mcl Y_{s}$. 

Suppose $z \in \mcl O_{s}$. 
Let $\ep   > 0$ be small enough that $0 \notin B_{100\ep^{1/4}}(z)$ and $\bdy B_{100\ep^{1/4}}(z)$ intersects the connected component of $\bdy \mcl B_{s}$ which contains $z$. Let $w\in B_\ep(z) \cap \BB Q^2$.
We claim that $F_{ \ep}(w)$ occurs.

Indeed, if $F_\ep(w)$ does not occur then by definition there is some $r \in [\ep^{1/2} , \ep^{1/4}] \cap \BB Q$ for which $G_r(w)$ occurs, i.e., there is a path $\pi$ in $B_{2r}(w) \setminus B_r(w)$ which disconnects the inner and exterior boundaries of $B_{2r}(w) \setminus B_r(w)$ and whose $D_h$-length is shorter than $D_h(\bdy B_{r/2}(w) , \bdy B_r(w))$. 
Let $P$ be a $D_h$-geodesic from 0 to $z$. 
Since $z\in B_{r/2}(w)$, $P$ must hit the path $\pi$ and then subsequently cross from $\bdy B_r(w)$ to $\bdy B_{r/2}(w)$. 
Since $P$ is a $D_h$-geodesic and the $D_h$-length of $\pi$ is shorter than $D_h(\bdy B_{r/2}(w) , \bdy B_r(w))$, this implies that $\pi \subset \mcl B_{s}$. 
But, $\pi$ disconnects $z$ from $\bdy B_{100\ep^{1/4}}(z)$. This is a contradiction since we have assumed that $\ep$ is small enough so that $\bdy B_{100\ep^{1/4}}(z)$ intersects the connected component of $\bdy \mcl B_{s}$ which contains $z$.
Therefore $F_\ep(w)$ occurs, and so $\mcl O_{s} \subset \mcl Y_{s}$.

Finally, to get that with positive probability $\bdy\mcl B_{s} \setminus \mcl O_s$ is uncountable, we observe that~\cite[Theorem 1.1]{gwynne-ball-bdy} shows that, with positive probability,  $\dim_{\mcl H}^0 \bdy\mcl B_{s}  > 2-\xi Q + \xi^2/2 - q \geq \dim_{\mcl H}^0 \mcl O_{s}$.  If this is the case then there must be uncountably many points in $\bdy\mcl B_{s}\setminus \mcl O_{s}$. Similarly, since Theorem~\ref{thm-ball-bdy-bound} shows that, \emph{almost surely}, $\dim_{\mcl H}^0 \bdy\mcl B_{D_h(0,z)} > 2-\xi Q + \xi^2/2 - q \geq \ \dim_{\mcl H}^0 \mcl O_{D_h(0,z)}$, we deduce that  $\bdy\mcl B_{s}\setminus \mcl O_{s}$ a.s. contains uncountably many points.
 \end{proof}

\bibliography{cibib}

\newcommand{\etalchar}[1]{$^{#1}$}
\def\cprime{$'$}
\begin{thebibliography}{{Gwy}20b}

\bibitem[AKM17]{akm-geodesics}
O.~Angel, B.~Kolesnik, and G.~Miermont.
\newblock Stability of geodesics in the {B}rownian map.
\newblock {\em Ann. Probab.}, 45(5):3451--3479, 2017, \arxiv{1502.04576}.
  \MR{3706747}

\bibitem[Ang19]{ang-discrete-lfpp}
M.~Ang.
\newblock Comparison of discrete and continuum {L}iouville first passage
  percolation.
\newblock {\em Electron. Commun. Probab.}, 24:Paper No. 64, 12, 2019,
  \arxiv{1904.09285}. \MR{4029433}

\bibitem[BBI01]{bbi-metric-geometry}
D.~Burago, Y.~Burago, and S.~Ivanov.
\newblock {\em A course in metric geometry}, volume~33 of {\em Graduate Studies
  in Mathematics}.
\newblock American Mathematical Society, Providence, RI, 2001. \MR{1835418}

\bibitem[Bef08]{beffara-dim}
V.~Beffara.
\newblock The dimension of the {SLE} curves.
\newblock {\em Ann. Probab.}, 36(4):1421--1452, 2008, \arxiv{math/0211322}.
  \MR{2435854 (2009e:60026)}

\bibitem[{Ber}]{berestycki-lqg-notes}
N.~{Berestycki}.
\newblock {I}ntroduction to the {G}aussian {F}ree {F}ield and {L}iouville
  {Q}uantum {G}ravity.
\newblock {A}vailable at
  \url{https://homepage.univie.ac.at/nathanael.berestycki/articles.html}.

\bibitem[BM17]{bet-mier-disk}
J.~Bettinelli and G.~Miermont.
\newblock Compact {B}rownian surfaces {I}: {B}rownian disks.
\newblock {\em Probab. Theory Related Fields}, 167(3-4):555--614, 2017,
  \arxiv{1507.08776}. \MR{3627425}

\bibitem[Dav88]{david-conformal-gauge}
F.~David.
\newblock Conformal field theories coupled to {2-D} gravity in the conformal
  gauge.
\newblock {\em {M}od. {P}hys. {L}ett. {A}}, 3(17), 1988.

\bibitem[DDDF20]{dddf-lfpp}
J.~Ding, J.~Dub\'{e}dat, A.~Dunlap, and H.~Falconet.
\newblock Tightness of {L}iouville first passage percolation for {$\gamma \in
  (0,2)$}.
\newblock {\em Publ. Math. Inst. Hautes \'{E}tudes Sci.}, 132:353--403, 2020,
  \arxiv{1904.08021}. \MR{4179836}

\bibitem[DFG{\etalchar{+}}20]{lqg-metric-estimates}
J.~Dub\'{e}dat, H.~Falconet, E.~Gwynne, J.~Pfeffer, and X.~Sun.
\newblock Weak {LQG} metrics and {L}iouville first passage percolation.
\newblock {\em Probab. Theory Related Fields}, 178(1-2):369--436, 2020,
  \arxiv{1905.00380}. \MR{4146541}

\bibitem[DG18]{dg-lqg-dim}
J.~{Ding} and E.~{Gwynne}.
\newblock {The fractal dimension of {L}iouville quantum gravity: universality,
  monotonicity, and bounds}.
\newblock {\em {C}ommunications in {M}athematical {P}hysics}, 374:1877--1934,
  2018, \arxiv{1807.01072}.

\bibitem[DG19]{ding-goswami-watabiki}
J.~Ding and S.~Goswami.
\newblock Upper bounds on {L}iouville first-passage percolation and
  {W}atabiki's prediction.
\newblock {\em Comm. Pure Appl. Math.}, 72(11):2331--2384, 2019,
  \arxiv{1610.09998}. \MR{4011862}

\bibitem[DK89]{dk-qg}
J.~Distler and H.~Kawai.
\newblock Conformal field theory and {2D} quantum gravity.
\newblock {\em {N}ucl.{P}hys. {B}}, 321(2), 1989.

\bibitem[DS11]{shef-kpz}
B.~Duplantier and S.~Sheffield.
\newblock Liouville quantum gravity and {KPZ}.
\newblock {\em Invent. Math.}, 185(2):333--393, 2011, \arxiv{1206.0212}.
  \MR{2819163 (2012f:81251)}

\bibitem[DZZ19]{dzz-heat-kernel}
J.~Ding, O.~Zeitouni, and F.~Zhang.
\newblock Heat kernel for {L}iouville {B}rownian motion and {L}iouville graph
  distance.
\newblock {\em Comm. Math. Phys.}, 371(2):561--618, 2019, \arxiv{1807.00422}.
  \MR{4019914}

\bibitem[GM20a]{gm-confluence}
E.~Gwynne and J.~Miller.
\newblock Confluence of geodesics in {L}iouville quantum gravity for {$\gamma
  \in (0,2)$}.
\newblock {\em Ann. Probab.}, 48(4):1861--1901, 2020, \arxiv{1905.00381}.
  \MR{4124527}

\bibitem[GM20b]{local-metrics}
E.~Gwynne and J.~Miller.
\newblock Local metrics of the {G}aussian free field.
\newblock {\em Ann. Inst. Fourier (Grenoble)}, 70(5):2049--2075, 2020,
  \arxiv{1905.00379}. \MR{4245606}

\bibitem[GM21a]{gm-coord-change}
E.~Gwynne and J.~Miller.
\newblock Conformal covariance of the {L}iouville quantum gravity metric for
  {$\gamma\in(0,2)$}.
\newblock {\em Ann. Inst. Henri Poincar\'{e} Probab. Stat.}, 57(2):--, 2021,
  \arxiv{1905.00384}. \MR{4260493}

\bibitem[GM21b]{gm-uniqueness}
E.~Gwynne and J.~Miller.
\newblock Existence and uniqueness of the {L}iouville quantum gravity metric
  for {$\gamma\in(0,2)$}.
\newblock {\em Invent. Math.}, 223(1):213--333, 2021, \arxiv{1905.00383}.
  \MR{4199443}

\bibitem[GMS19]{gms-harmonic}
E.~{Gwynne}, J.~{Miller}, and S.~{Sheffield}.
\newblock {Harmonic functions on mated-{CRT} maps}.
\newblock {\em Electron. J. Probab.}, 24:no. 58, 55, 2019, \arxiv{1807.07511}.

\bibitem[GP19a]{gp-lfpp-bounds}
E.~{Gwynne} and J.~{Pfeffer}.
\newblock {Bounds for distances and geodesic dimension in Liouville first
  passage percolation}.
\newblock {\em {E}lectronic {C}ommunications in {P}robability}, 24:no. 56, 12,
  2019, \arxiv{1903.09561}.

\bibitem[GP19b]{gp-kpz}
E.~{Gwynne} and J.~{Pfeffer}.
\newblock {KPZ formulas for the Liouville quantum gravity metric}.
\newblock {\em {T}ransactions of the {A}merican {M}athematical {S}ociety}, to
  appear, 2019.

\bibitem[Gwy20a]{gwynne-ball-bdy}
E.~Gwynne.
\newblock The {D}imension of the {B}oundary of a {L}iouville {Q}uantum
  {G}ravity {M}etric {B}all.
\newblock {\em Comm. Math. Phys.}, 378(1):625--689, 2020, \arxiv{1909.08588}.
  \MR{4124998}

\bibitem[{Gwy}20b]{gwynne-geodesic-network}
E.~{Gwynne}.
\newblock {Geodesic networks in Liouville quantum gravity surfaces}.
\newblock {\em {Probability and Mathematical Physics}}, to appear, 2020,
  \arxiv{2010.11260}.

\bibitem[Gwy20c]{gwynne-ams-survey}
E.~Gwynne.
\newblock Random surfaces and {L}iouville quantum gravity.
\newblock {\em Notices Amer. Math. Soc.}, 67(4):484--491, 2020,
  \arxiv{1908.05573}. \MR{4186266}

\bibitem[HMP10]{hmp-thick-pts}
X.~Hu, J.~Miller, and Y.~Peres.
\newblock Thick points of the {G}aussian free field.
\newblock {\em Ann. Probab.}, 38(2):896--926, 2010, \arxiv{0902.3842}.
  \MR{2642894 (2011c:60117)}

\bibitem[HS18]{hs-euclidean}
N.~Holden and X.~Sun.
\newblock S{LE} as a mating of trees in {E}uclidean geometry.
\newblock {\em Comm. Math. Phys.}, 364(1):171--201, 2018, \arxiv{1610.05272}.
  \MR{3861296}

\bibitem[Kah85]{kahane}
J.-P. Kahane.
\newblock Sur le chaos multiplicatif.
\newblock {\em Ann. Sci. Math. Qu\'ebec}, 9(2):105--150, 1985. \MR{829798
  (88h:60099a)}

\bibitem[KPZ88]{kpz-scaling}
V.~Knizhnik, A.~Polyakov, and A.~Zamolodchikov.
\newblock {Fractal structure of 2D-quantum gravity}.
\newblock {\em {Modern Phys. Lett A}}, 3(8):819--826, 1988.

\bibitem[{Le }10]{legall-geodesics}
J.-F. {Le Gall}.
\newblock Geodesics in large planar maps and in the {B}rownian map.
\newblock {\em Acta Math.}, 205(2):287--360, 2010, \arxiv{0804.3012}.
  \MR{2746349 (2012b:60272)}

\bibitem[{Le }13]{legall-uniqueness}
J.-F. {Le Gall}.
\newblock Uniqueness and universality of the {B}rownian map.
\newblock {\em Ann. Probab.}, 41(4):2880--2960, 2013, \arxiv{1105.4842}.
  \MR{3112934}

\bibitem[Mie13]{miermont-brownian-map}
G.~Miermont.
\newblock The {B}rownian map is the scaling limit of uniform random plane
  quadrangulations.
\newblock {\em Acta Math.}, 210(2):319--401, 2013, \arxiv{1104.1606}.
  \MR{3070569}

\bibitem[MP10]{peres-bm}
P.~M{\"o}rters and Y.~Peres.
\newblock {\em Brownian motion}.
\newblock Cambridge Series in Statistical and Probabilistic Mathematics.
  Cambridge University Press, Cambridge, 2010.
\newblock With an appendix by Oded Schramm and Wendelin Werner. \MR{2604525
  (2011i:60152)}

\bibitem[MQ20a]{mq-strong-confluence}
J.~{Miller} and W.~{Qian}.
\newblock {Geodesics in the Brownian map: Strong confluence and geometric
  structure}.
\newblock {\em ArXiv e-prints}, August 2020, \arxiv{2008.02242}.

\bibitem[MQ20b]{mq-geodesics}
J.~{Miller} and W.~{Qian}.
\newblock {The geodesics in Liouville quantum gravity are not Schramm-Loewner
  evolutions}.
\newblock {\em {Probab. Theory Related Fields}}, 177(3-4):677--709, 2020,
  \arxiv{1812.03913}.

\bibitem[MS17]{ig4}
J.~Miller and S.~Sheffield.
\newblock Imaginary geometry {IV}: interior rays, whole-plane reversibility,
  and space-filling trees.
\newblock {\em Probab. Theory Related Fields}, 169(3-4):729--869, 2017,
  \arxiv{1302.4738}. \MR{3719057}

\bibitem[MS20]{lqg-tbm1}
J.~Miller and S.~Sheffield.
\newblock Liouville quantum gravity and the {B}rownian map {I}: the {${\mathrm
  QLE}(8/3,0)$} metric.
\newblock {\em Invent. Math.}, 219(1):75--152, 2020, \arxiv{1507.00719}.
  \MR{4050102}

\bibitem[MS21a]{tbm-characterization}
J.~Miller and S.~Sheffield.
\newblock An axiomatic characterization of the {B}rownian map.
\newblock {\em J. \'{E}c. polytech. Math.}, 8:609--731, 2021,
  \arxiv{1506.03806}. \MR{4225028}

\bibitem[MS21b]{lqg-tbm2}
J.~Miller and S.~Sheffield.
\newblock Liouville quantum gravity and the {B}rownian map {II}: {G}eodesics
  and continuity of the embedding.
\newblock {\em Ann. Probab.}, 49(6):2732--2829, 2021, \arxiv{1605.03563}.
  \MR{4348679}

\bibitem[MS21c]{lqg-tbm3}
J.~Miller and S.~Sheffield.
\newblock Liouville quantum gravity and the {B}rownian map {III}: the conformal
  structure is determined.
\newblock {\em Probab. Theory Related Fields}, 179(3-4):1183--1211, 2021,
  \arxiv{1608.05391}. \MR{4242633}

\bibitem[MWW16]{mww-nesting}
J.~Miller, S.~S. Watson, and D.~B. Wilson.
\newblock Extreme nesting in the conformal loop ensemble.
\newblock {\em Ann. Probab.}, 44(2):1013--1052, 2016, \arxiv{1401.0218}.
  \MR{3474466}

\bibitem[Pol81]{polyakov-qg1}
A.~M. Polyakov.
\newblock Quantum geometry of bosonic strings.
\newblock {\em Phys. Lett. B}, 103(3):207--210, 1981. \MR{623209 (84h:81093a)}

\bibitem[Pom92]{pom-book}
C.~Pommerenke.
\newblock {\em Boundary behaviour of conformal maps}, volume 299 of {\em
  Grundlehren der Mathematischen Wissenschaften [Fundamental Principles of
  Mathematical Sciences]}.
\newblock Springer-Verlag, Berlin, 1992. \MR{1217706 (95b:30008)}

\bibitem[RV11]{rhodes-vargas-log-kpz}
R.~Rhodes and V.~Vargas.
\newblock K{PZ} formula for log-infinitely divisible multifractal random
  measures.
\newblock {\em ESAIM Probab. Stat.}, 15:358--371, 2011, \arxiv{0807.1036}.
  \MR{2870520}

\bibitem[She07]{shef-gff}
S.~Sheffield.
\newblock Gaussian free fields for mathematicians.
\newblock {\em Probab. Theory Related Fields}, 139(3-4):521--541, 2007,
  \arxiv{math/0312099}. \MR{2322706 (2008d:60120)}

\bibitem[Wat93]{watabiki-lqg}
Y.~Watabiki.
\newblock {Analytic study of fractal structure of quantized surface in
  two-dimensional quantum gravity}.
\newblock {\em Progr. Theor. Phys. Suppl.}, (114):1--17, 1993.
\newblock Quantum gravity (Kyoto, 1992).

\bibitem[WP20]{pw-gff-notes}
W.~{Werner} and E.~{Powell}.
\newblock {Lecture notes on the Gaussian Free Field}.
\newblock {\em ArXiv e-prints}, April 2020, \arxiv{2004.04720}.

\end{thebibliography}
\bibliographystyle{hmralphaabbrv}

\end{document}